%% file: cansch.tex
\documentclass{amsart}

\usepackage{amsmath}
\usepackage{amssymb}
\usepackage{amsthm}
\usepackage{tikz}
\usetikzlibrary{matrix,arrows,backgrounds,shapes.misc,shapes.geometric,patterns,calc,positioning}
\usetikzlibrary{calc,shapes}

\usepackage{epsfig}  
\usepackage{color}	 %
\input{xy}
\xyoption{poly}
\xyoption{2cell}
\xyoption{all}

\usepackage{lscape}

\usetikzlibrary{calc,shapes}
\usetikzlibrary{matrix,arrows,backgrounds,shapes.misc,shapes.geometric,patterns,calc,positioning}
\usetikzlibrary{calc,shapes}

\newcommand{\calf}{\mathcal{F}}
\newcommand{\cala}{\mathcal{A}}

\newcommand{\calg}{\mathcal{G}}
\newcommand{\call}{\mathcal{L}}
\newcommand{\ocalg}{\overline{\mathcal{G}}}

\newcommand{\tcalg}{\widetilde {\mathcal{G}}}

\newcommand{\bx}{{\bf x}}
\newcommand{\by}{{\bf y}}

\newcommand{\compcent}[1]{\vcenter{\hbox{$#1\circ$}}} 

\newcommand{\comp}{\mathbin{\mathchoice 
{\compcent\scriptstyle}{\compcent\scriptstyle} 
{\compcent\scriptscriptstyle}{\compcent\scriptscriptstyle}}} 

\newtheorem{thm}{Theorem}[section]

\newtheorem{lem}[thm]{Lemma}
\newtheorem{cor}[thm]{Corollary}

\theoremstyle{defn}
\newtheorem{defn}[thm]{Definition}

\theoremstyle{remark}

\newtheorem{remark}[thm]{Remark}

\numberwithin{equation}{section}

\DeclareMathSizes{5}{3}{2}{2}

\newcommand{\g}[1]{G_{#1}}
\newcommand{\gp}[1]{G'_{#1}}

\newcommand{\gi}[3]{G_{#1}, G_{#2}, \ldots, G_{#3}}
\newcommand{\gii}[3]{G'_{#1}, G'_{#2}, \ldots, G'_{#3}}

\newcommand{\gr}[3] {\calg_1 = (G_{#1}, G_{#2}, \ldots, G_{#3})}
\newcommand{\grr}[3] {\calg'_1 = (G'_{#1}, G'_{#2}, \dots, G'_{#3})}
\newcommand{\bij}[2] {\varphi ( P _{#1},  P_{#2})}
\newcommand{\ma}[2] {\match (\calg_{#1} \sqcup \calg_{#2})}
\newcommand{\re}[2] {\res_{\calg} ( \calg_{#1}, \calg_{#2})}
\newcommand{\gt}[4] {\graft_{{#1}, {#2}} ( \calg_{#3}, \calg_{#4})}

\newcommand{\la}[2] {\call ( \calg_{#1} \sqcup \calg_{#2}) }

\newcommand{\ta}[2]{(\tau_{i_{#1}}, \ldots, \tau_{i_{#2}})}
\newcommand{\taa}[2]{(\tau'_{i'_{#1}}, \ldots, \tau'_{i'_{#2}})}

\newcommand{\tp}{s>1,\ s'>1,\ t<d, \mbox{ and } t'<d' } 
\newcommand{\tpp}{s=1,\ s'=1,\ t=d, \mbox{ or } t'=d' } 

\newcommand{\prodd}{\displaystyle \prod_{\tau \in \partial \calg}} 

\newcommand{\pprod}[3]{\displaystyle \prod_{#1}^{#2} {#3}} 

\newcommand\restr[2]{{
  \left.\kern-\nulldelimiterspace 
  #1 
  \vphantom{\big|} 
  \right|_{#2} 
  }}
  
\newcommand{\e}{\delta}

\newcommand{\overunder}[2]{
\!\begin{array}{c}
\scriptstyle{#1}\\[-.1in]
-\!\!\!-\!\!\!-\\[-.1in]
\scriptstyle{#2}
\end{array}
\!
}

\def\sgn{\operatorname{sgn}}
\def\TT{\mathbb{T}}

\def\Trop{\operatorname{Trop}}
\def\QQ{\mathbb{Q}}

\newcommand{\za}{\alpha}
\newcommand{\zb}{\beta}

\newcommand{\zD}{\Delta}
\newcommand{\ze}{\epsilon}
\newcommand{\zg}{\gamma}

\DeclareMathOperator{\Match}{Match}

\newcommand{\A}{\mathcal{A}}

\newcommand{\PP}{\mathbb{P}}

\def\Aprin{\Acal_\bullet}

\def\Xcal{\mathcal{X}}
\def\Acal{\mathcal{A}}
\def\Fcal{\mathcal{F}}

\def\yy{\mathbf{y}}

\def\xx{\mathbf{x}}

\def\ZZ{\mathbb{Z}}

\DeclareMathOperator{\res}{Res \,} 
\DeclareMathOperator{\graft}{Graft \,  } 
 
\DeclareMathOperator{\match}{Match \, }
\DeclareMathOperator{\NE}{NE}
\DeclareMathOperator{\SW}{SW}

\begin{document}
\title{Snake graph calculus and cluster algebras from surfaces}
\author{Ilke Canakci}
\address{Department of Mathematics, University of Connecticut, 
Storrs, CT 06269-3009}
\email{ilke.canakci@uconn.edu}
\author{Ralf Schiffler}\thanks{The  authors were supported by the NSF grant  DMS-1001637 and by the University of Connecticut.}
\address{Department of Mathematics, University of Connecticut, 
Storrs, CT 06269-3009}
\email{schiffler@math.uconn.edu}





\begin{abstract}
Snake graphs appear naturally in the theory of cluster algebras. For cluster algebras from surfaces, each cluster variable is given by a   formula which is parametrized by the perfect matchings of a snake graph. 
 In this paper, we identify each cluster variable with its snake graph, and interpret relations among the cluster variables  in terms of these graphs. In particular, we give a new proof of skein relations of two cluster variables.
\end{abstract}


 \maketitle





\section{Introduction}
Cluster algebras were introduced in \cite{FZ1}, and further developed in \cite{FZ2,BFZ,FZ4}, motivated by combinatorial aspects of canonical bases in Lie theory \cite{Lusztig1,Lusztig2}. A cluster algebra is a subalgebra of a field of rational functions in several variables, and is given by constructing a distinguished set of generators, the \emph{cluster variables}. The cluster variables are constructed recursively and their computation is rather complicated in general. By construction cluster variables are rational functions, but it was shown in \cite{FZ1} that they are actually Laurent polynomials with integer coefficients. Moreover, these coefficients are conjectured to be positive; this is the positivity conjecture.

An important class of cluster algebras is given by the cluster algebras from surfaces \cite{GSV,FG1,FG2,FST,FT}. From a classification point of view, this class is very important since it has been shown in \cite{FeShTu} that almost all (skew-symmetric) mutation finite cluster algebras are from surfaces. For generalizations to the skew-symmetrizalbe case, see \cite{FeShTu2,FeShTu3}. If $\mathcal {A}$  is a cluster algebra from a surface, then there exists a marked surface with boundary, such that the cluster variables of $\mathcal{A}$ are in bijection with certain curves, called \emph{arcs}, in the surface and the relations between the cluster variables are given by the crossing patterns of the arcs in the surface.

In a collaboration with Musiker and Williams \cite{MSW}, building on earlier work \cite{S2,ST,S3,MS}, the second author used this  geometric interpretation to obtain a direct combinatorial formula for the cluster variables in cluster algebras from surfaces. This formula is manifestly positive and thus proves the positivity conjecture. In \cite{MSW2}, the formula was the key ingredient in the construction of two bases for the cluster algebra in the case where the surface has no punctures.

The formula is parametrized by perfect matchings of certain graphs, the \emph{snake graphs}, which are the subject of the present paper. Snake graphs had  appeared earlier in \cite{Propp} in the special case of triangulated polygons. To compute the cluster variable associated to an arc $\gamma$, one constructs the snake graph $\calg_{\gamma}$
as a union of square shaped graphs, called \emph{tiles}, which correspond to quadrilaterals in a fixed triangulation of the surface: one tile for  each quadrilateral traversed by $\gamma$. These tiles are glued together according the the geometry of the surface. Thus to every cluster variable corresponds an arc, and to every arc corresponds a snake graph. A natural question is then how much of the relations among the cluster variables or, equivalently, how much of the geometry of the arcs, can be recovered from the snake graphs alone? 

For example, the product of two cluster variables can be thought of as the union of the two corresponding snake graphs. It is known \cite{MW,MSW2} that if the two arcs cross, then one can rewrite the product of the cluster variables as a linear combination of elements corresponding to arcs without crossings. The process of resolving the crossings is given on the arcs by a \emph{smoothing operation} and the relation in the cluster algebra is called a \emph{skein relation}. On the level of snake graphs, one needs to know the following.
\begin{enumerate}
\item When do the two arcs corresponding to two snake graphs cross?
\item What are the snake graphs corresponding to the skein relations?
\end{enumerate}

In this paper, we introduce the notion of an \emph{abstract snake graph}, which is not necessarily related to an arc in a surface. Then we define what it means for two abstract snake graphs to cross. Given two crossing snake graphs, we construct the resolution of the crossing  as two pairs of snake graphs from the original pair of crossing snake graphs. We then prove that there is a bijection $\varphi$ between the set of perfect matchings of the two crossing snake graphs and the set of perfect matchings of the resolution, Theorem \ref{bijections}.

We then apply our constructions to snake graphs arising from unpunctured surfaces and prove that \begin{enumerate}
\item two snake graphs cross if and only if the corresponding arcs cross in the surface, Theorem \ref{lem 53};
\item the resolution of the crossing of the snake graphs coincides with the snake graphs of the curves obtained from the two crossing arcs by smoothing, Theorems \ref{smoothing1} and \ref{smoothing2};
\item the bijection $\varphi$ above is weight preserving, where the weights of the edges are the initial cluster variables associated to the arcs of the initial triangulation of the surface. Consequently, the Laurent polynomials associated to a pair of crossing snake graphs and to its resolution are identical, Theorems \ref{laurent} and \ref{laurent2}.
\end{enumerate}

As an application, we obtain a combinatorial formula in terms of snake graphs for the product of cluster variables, and a new proof of the skein relations for cluster variables, Corollary \ref{skein}.

The paper is organized as follows.
We introduce the notions of abstract snake graphs, as well as their crossings and resolutions in section \ref{sect 2}. In section \ref{sect 3}, we give the definition of the bijection $\varphi$. After recalling the definitions and results on cluster algebras from surfaces in section \ref{sect 4}, we prove our the results (1)-(3) in sections \ref{sect 5} and \ref{sect 6}. Section \ref{sect 7} contains the proof that the map $\varphi$ is a bijection.



\section{Abstract snake graphs}\label{sect 2}
Fix an orthonormal basis of the plane.

A {\bf tile} $G$ is a square of fixed side-length in the plane whose sides are parallel or orthogonal  to the fixed basis.
\begin{center}
  { \small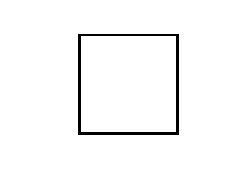}
\end{center}
We consider a tile $G$ as  a graph with four vertices and four edges in the obvious way. A {\em snake graph} $\calg$ is a connected graph consisting of a finite sequence of tiles $ \gi 12d$ with $d \geq 1,$ such that for each $i=1,\dots,d-1$

\begin{itemize}
 \item[(i)] $G_i$ and $G_{i+1}$ share exactly one edge $e_i$ and this edge is either the north edge of $G_i$ and the south edge of $G_{i+1}$ or the east edge of $G_i$ and the west edge of $G_{i+1}.$
 \item[(ii)] $G_i $ and $G_j$ have no edge in common whenever $|i-j| \geq 2.$
 \item[(ii)] $G_i $ and $G_j$ are disjoint whenever $|i-j| \geq 3.$
\end{itemize}
An example is given in Figure \ref{signfigure}.
\begin{figure}
\begin{center}
  {\tiny \scalebox{0.9}{ 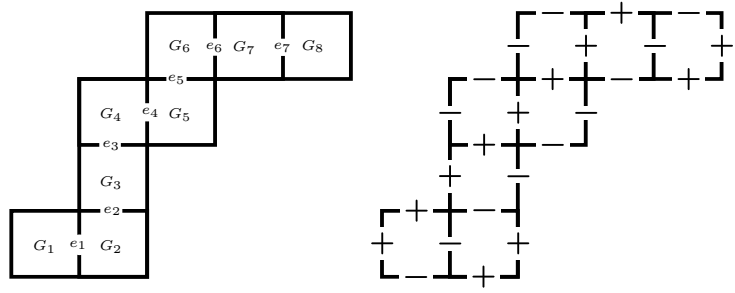}}
 \caption{A snake graph with 8 tiles and 7 interior edges (left);
 a sign function on the same snake graph (right)} 
 \label{signfigure}
\end{center}
\end{figure}

 We sometimes use the notation $\calg =(\gi 12d)$ for the snake graph and $\calg [i, i+t] = (\gi i{i+1}{i+t})$ for the subgraph of $\calg$ consisting of the tiles $\gi i{i+1}{i+t}.$

The $d-1$ edges $e_1,e_2, \dots, e_{d-1}$ which are contained in two tiles are called {\em interior edges} of $\calg$ and the other edges are called {\em boundary edges.} 

By convention we call a set containing a single edge an \emph{empty snake graph}.

A snake graph $\calg$ is called {\em straight} if all its tiles lie in one column or one row, and a snake graph is called {\em zigzag} if no three consecutive tiles are straight.
\subsection{Sign function} 

A {\em sign function} $f$ on a snake graph $\calg$ is a map $f$ from the set of edges of $\calg$ to $\{ +,- \}$ such that on every tile in $\calg$ the north and the west edge have the same sign, the south and the east edge have the same sign and the sign on the north edge is opposite to the sign on the south edge. See Figure \ref{signfigure} for an example.
%
%

Note that on every nonempty snake graph there are exactly two sign functions.
\subsection{Overlaps}

Let $\calg_1=(\gi 12d)$ and $\calg_2=(\gii 12{d'})$ be two snake graphs. We say that $\calg_1$ and $\calg_2$ have an {\em overlap} $\calg$ if $\calg$ is a snake graph and there exist two embeddings $i_1 : \calg \rightarrow \calg_1,$ $i_2: \calg \rightarrow \calg_2$ which are maximal in the following sense.

\begin{itemize}
\item[(i)] If $\calg$ has at least two tiles and if there exists a snake graph $\calg'$ with two embeddings $i'_1: \calg' \to \calg_1,$ $i'_2: \calg \to \calg_2$ such that $i_1(\calg) \subseteq i'_1(\calg')$ and $i_2(\calg) \subseteq i'_2(\calg')$ then $i_1(\calg) = i'_1(\calg')$ and $i_2(\calg) = i'_2(\calg').$
\item[(ii)] If $\calg$ consists of a single tile then using the notation $G_k=i_1(\calg)$ and $G'_{k'}= i_2(\calg),$ we have
\begin{itemize}
 \item[(a)] $k \in \{ 1,d \}$ or $k' \in \{ 1, d' \}$ or
 \item[(b)] $1<k<d,$ $1<k'<d'$ and the subgraphs $(G_{k-1},G_k,G_{k+1})$ and $(G'_{k'-1},G'_{k'},G'_{k'+1})$ are either both straight or both zigzag subgraphs.
\end{itemize}
\end{itemize}
An example of type (i) is shown in Figure \ref{overlap} and an example of type (ii)(b)   in Figure~\ref{overlapb}.

\begin{figure}
\begin{center}
\scalebox{0.8} {  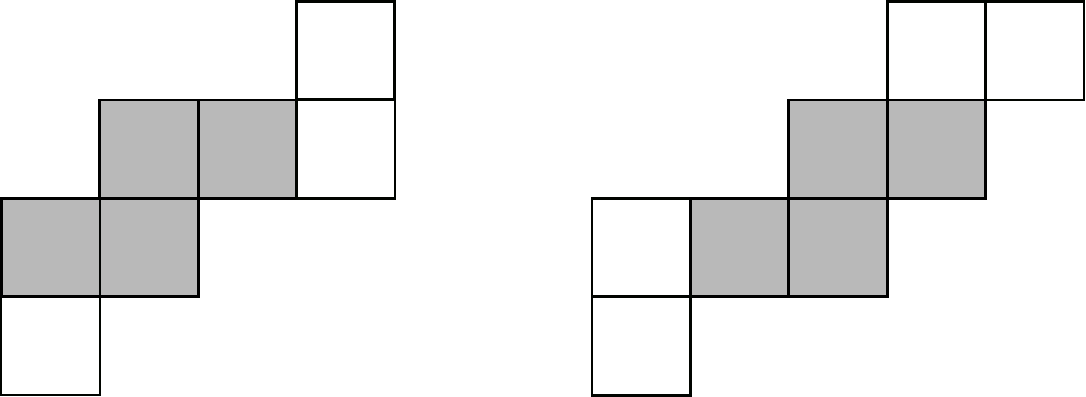}
 \caption{Two snake graphs with overlap (shaded)}
 \label{overlap}
 \end{center}
\end{figure}

\begin{figure}
\begin{center}
\scalebox{0.8} {  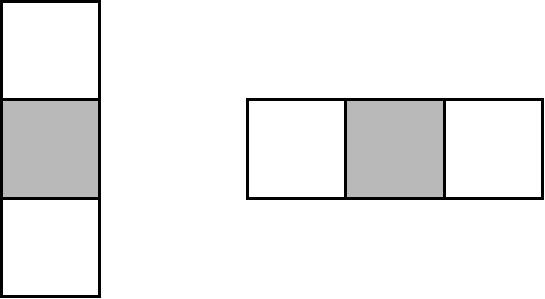}
 \caption{Two snake graphs with overlap consisting of a single tile (shaded)}
 \label{overlapb}
 \end{center}
\end{figure}

Note that two snake graphs may have several overlaps with respect to different snake graphs $\calg$.
\subsection{Crossing}

Let $\gr 12d,$ $\grr 12{d'}$ be two snake graphs with overlap $\calg$ and embeddings $i_1(\calg)=(\gi s{s+1}t)$ and $i_2(\calg)=(\gii {s'}{s'+1}{t'})$ and suppose without loss of generality that $s \leq t$ and $s' \leq t'.$ Let $e_1, \dots, e_{d-1} $ (respectively $e'_1, \dots, e'_{d'-1}$) be the interior edges of $\calg_1$ (respectively $\calg_2$.) Let $f$ be a sign function on $\calg.$ Then $f$ induces a sign function $f_1$ on $\calg_1$ and $f_2$
 on $\calg _2.$ Moreover, since the overlap $\calg$ is maximal, we have 
 
\[\begin{array}{rcll}
 f_1(e_{s-1}) &=&-f_2(e'_{s'-1}) & \mbox{ if }  s>1, s'>1\\
f_1(e_{t})&=&-f_2(e'_{t'})&\mbox{ if }t <d, t'< d'.
\end{array}
\]%
 
\begin{defn} \label{crossing} We say that $\calg_1$ and $\calg_2$ {\em cross in} $\calg$ if one of the following conditions hold.

\begin{itemize}
 \item[(i)] \label{i}
\[ \begin{array}{lrclll}
 &f_1(e_{s-1})& =& -f_1(e_t) & \mbox{ if } &s>1, t<d   \\
 \mbox{ or}\\
&  f_2 (e'_{s'-1})&=&-f_2(e'_{t'}) & \mbox{ if } &s'>1, t'<d'  
\end{array}\]
%
 \item[(ii)]  \[\begin{array}{lrcllcc}
 &f_1(e_{t}) &=& f_2(e'_{s'-1})& \mbox{ if } &s=1, t<d, s'>1, t'=d'  \\
  \mbox{ or}\\
&f_1 (e_{s-1})&=&f_2(e'_{t'}) & \mbox{ if } &s>1 , t=d, s'=1, t'<d' &
\end{array}\] 
\end{itemize}
\end{defn}

\begin{remark}
\begin{itemize}
 \item[1.] The definition does not depend on the choice of the sign function $f.$
 \item[2.] $\calg_1$ and $\calg_2$ may still cross if $s=1$ and $t=d$ because they may satisfy condition \ref{i}(i).
 \item[3.] The terminology `cross' comes from snake graphs that are associated to arcs in a surface. We shall show in  Theorem \ref{lem 53} that two such arcs cross if and only if the corresponding snake graphs cross in an overlap.
\end{itemize}
\end{remark}

\subsection {Resolution of crossing} Given two snake graphs that cross, we construct two pairs of new snake graphs, which we call the resolution of the crossing. In section~\ref{sect pm}, we show that there is a bijection between the set of perfect matchings of the two crossing snake graphs and the set of perfect matchings of the resolution. In sections~\ref{sect smoothing} and \ref{sect products}, we show that this construction is related to multiplication formulas given by Skein relations in cluster algebras.

 Let $\calg_1,$ $\calg_2$ be two snake graphs crossing in an overlap $\calg=\calg_1[s,t]=\calg_2[s',t']$. Recall that $\calg_k[i,j]$ is the subgraph of $\calg_k$ given by the tiles with indices $i, i+1, \dots, j.$ Let $\ocalg_k[j,i]$ be the snake graph obtained by reflecting $\calg_k[i,j]$ such that the order of the tiles is reversed.

We define four connected subgraphs as follows, see Figure \ref{figres} for examples.
\begin{align*}
 \calg_3&=\calg_1[1,t] \cup \calg_2[t'+1, d'] \mbox{ where the adjacency of the two subgraphs is induced by } \calg_2. \\
\calg_4&=\calg_2[1,t'] \cup \calg_1[t+1,d] \mbox{ where the adjacency of the two subgraphs is induced by } \calg_1.\\
 \calg_5 &=
\begin{cases}
\calg_1[1,s-1] \cup \ocalg_2[s'-1,1]&  \parbox[t]{.55\textwidth}{ if $s>1$, $s'>1$ where the two subgraphs are glued along the north of $G_{s-1}$ and the east of $G'_{s'-1}$ if $G_s$ is east of $G_{s-1}$ in $\calg_1$; and along the east of $G_{s-1}$ and the north of $G'_{s'-1}$ if $G_s$ is north of $G_{s-1}$ in $\calg_1$;}\\
 \calg_1[1,k] &     \parbox[t]{.55\textwidth}{ if $s'=1$ where $k < s-1$ is the largest integer such that $f_1(e_k)=f_1(e_{s-1})$ if such a $k$ exists; }  \\
  \{ e_0 \}&    \parbox[t]{.55\textwidth}{ if $s'=1$ and no such $k$ exists, where $e_0$ is the unique edge of $G_1$ which is south or west and satisfies $f_1(e_0)=f_1(e_{s-1})$; }\\
 \ocalg_2[k',1] &    \parbox[t]{.55\textwidth}{ if $s=1$ where $k' < s'-1$ is the largest integer such that $f_2(e_{k'}) = f_2(e_{s'-1})$ if such a $k$ exists; }\\   
 \{ e'_0 \}&   \parbox[t]{.55\textwidth}{ if $s=1$ and no such $k'$ exists, where $e'_0$ is the unique edge of $G'_1$ which is south or west and satisfies $f_2(e'_0)=f_2(e'_{s'-1})$; } 
\end{cases}\\ \end{align*}
\begin{align*}
\calg_6&=
\begin{cases}
\ocalg_2[d',t'+1] \cup \calg_1[t+1,d]&  \parbox[t]{.55\textwidth}{ if $t<d$, $t' <d'$ where the two subgraphs are glued along the west of $G_{t+1}$ and the south of $G'_{t'+1}$ if $G_{t+1}$ is north of $G_{t}$ in $\calg_1$; and along the south of $G_{t+1}$ and the west of $G'_{t'+1}$ if $G_{t+1}$ is east of $G_{t}$ in $\calg_1$;}\\
 \ocalg_2[d',k'] &     \parbox[t]{.55\textwidth}{ if $t=d$ where $k' > t'+1$ is the least integer such that $f_2(e'_{t'})=f_2(e'_{k'-1})$ if such a $k'$ exists; }  \\
  \{ e'_{d'} \}&    \parbox[t]{.55\textwidth}{ if $t=d$ and no such $k'$ exists, where $e'_{d'}$ is the unique edge of $G'_{d'}$ which is north or east and satisfies $f_2(e'_{d'})=f_2(e'_{t'})$; }\\
 \calg_1[k,d] &    \parbox[t]{.55\textwidth}{ if $t'=d'$ where $k > t+1$ is the least integer such that $f_1(e_{t}) = f_1(e_{k-1})$ if such a $k$ exists; }\\   
 \{ e_d \}&   \parbox[t]{.55\textwidth}{ if $t'=d'$ and no such $k$ exists, where $e_d$ is the unique edge of $G_d$ which is north or east and satisfies $f_1(e_d)=f_1(e_{t})$; } 
\end{cases}
\end{align*}

\begin{defn} \label{resolution}
 In the above situation, we say that the pair $(\calg_3 \sqcup \calg_4, \calg_5 \sqcup \calg_6 )$  is the {\em resolution of the crossing} of $\calg_1$ and $\calg_2$ at the overlap $\calg$ and we denote it by $\res_{\calg} (\calg_1,\calg_2).$ 
\end{defn}

If $\calg_1, \calg_2$ have no crossing in $\calg$ we let $\res_{\calg}  (\calg_1,\calg_2)=  \calg_1 \sqcup \calg_2. $

\begin{remark}
The pair $(\calg_3,\calg_4)$ still has an overlap in $\calg$ but without crossing. The pair  $(\calg_5,\calg_6)$ can be thought of as a reduced symmetric difference of $\calg_1$ and $\calg_2$ with respect to the overlap $\calg$.
\end{remark}

\begin{figure}
\begin{center}
\scalebox{0.8}{ 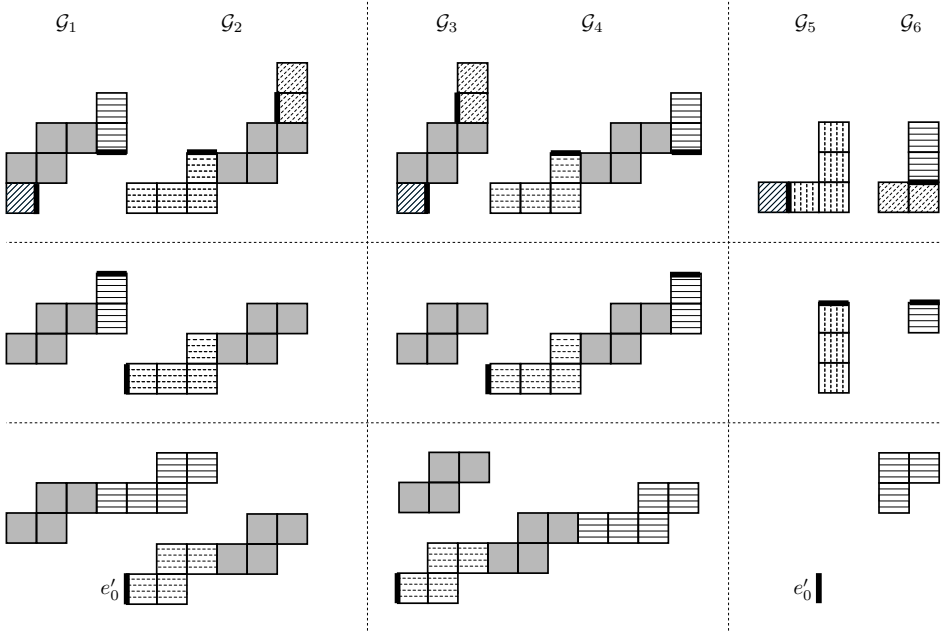}
\caption{Examples of resolutions: $s>1,s'>1,t<d,t'<d'$ in the first row;  $s=1, t'=d'$ in the second and third row;}
\label{figres}
\end{center}
\end{figure}

\subsection{Grafting} 
In this subsection, we define another operation which to two snake graphs associates two pairs of snake graphs. Here however, we do not suppose that the original two snake graphs have an overlap. In section~\ref{sect pm}, we show that there is a bijection between the set of perfect matchings of the two original snake graphs and the set of perfect matchings of the two new pairs. In sections~\ref{sect smoothing} and \ref{sect products}, we show that this construction is related to multiplication formulas in cluster algebras.

Let $\calg_1=(\gi 12d),$ $\calg_2=(\gii 12{d'})$ be two snake graphs and let $f_1$ be a sign function on $\calg_1.$

\emph{Case 1.}
Let $s$ be such that $1<s<d.$

If $G_{s+1}$ is north of $G_s$ in $\calg_1$ then let $\e_3$ denote the east edge of $G_s,$  $\e_5$ the west edge of $G_{s+1}$ and $\e'_3$ the west edge of $G'_1,$ $\e'_5$ the south edge of $G'_1.$

If $G_{s+1}$ is east of $G_s$ in $\calg_1,$ then let 
$\e_3 $  denote the north edge of $G_s $,
$\e_5$   the south edge of $G_{s+1}$ and
$\e_{3'}$   the south edge of  $ G'_1 $, 
$\e_{5'} $   the west edge of $ G'_1  $. 
Thus we have one of the following two situations.
\begin{center}
\scalebox{0.8}{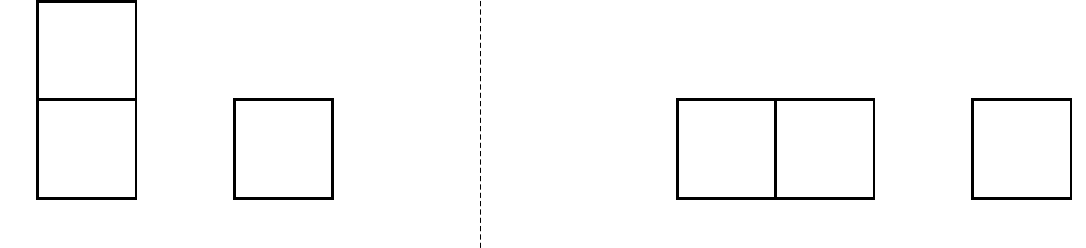}
\end{center}
Define four snake graphs as follows; see Figure \ref{figrafting} for examples.
\begin{align*}
 \calg_3=&\calg_1[1,s] \cup \calg_2 \mbox{ where the  two subgraphs are glued along the edges $\e_3$ and $\e'_3$}. \\
\calg_4=&
\begin{cases}
 \calg_1 [k_4,d] &  \parbox[t]{.75\textwidth}{ where $k_4 > s+1$ is the least integer such that $f_1(e_s)=-f_1(e_{k_4-1})$ if such a $k_4$ exists; } \\
 \{ e_d \} &  \parbox[t]{.75\textwidth}{ otherwise where $e_d$ is the unique edge of $G_d$ which is north or east and such that $f_1(e_s)=-f_1(e_d)$ }
\end{cases}
\\
 \calg_5 =&
\begin{cases}
\calg_1 [1,k_5] &  \parbox[t]{.75\textwidth}{ where $k_5 < s$ is the largest integer such that $f_1(e_{k_5})=-f_1(e_{s})$ if such a $k_5$ exists; } \\
 \{ e_0 \} &  \parbox[t]{.75\textwidth}{ otherwise where $e_0$ is the unique edge of $G_1$ which is south or west and such that $f_1(e_0)=-f_1(e_s)$ }
\end{cases}\\
 \calg_6=&\ocalg_2[d',1] \cup \calg_1 [s+1,d] \mbox{ where the two subgraphs are glued along the edges } & \\ &\mbox{ $\e_5$ and $\e'_5$ .} &
\end{align*}

\emph{Case 2.} Now let $s=d.$ Choose a pair of edges $(\e_3,\e'_3)$ such that either $\e_3$ is the north edge in $G_s$ and $\e'_3$ is the south edge in $G'_1$ or $\e_3$ is the east edge in $G_s$ and $\e'_3$ is the west edge in $G'_1.$ Let $f_2$ be  a sign function on $ \calg_2$ such that $f_2(\e'_3)=f_1(\e_3).$ Then define four snake graphs as follows.

\begin{align*}
 \calg_3=&\calg_1[1,s] \cup \calg_2 \mbox{ where the  two subgraphs are glued along the edges $\e_3$ and $\e'_3$.} &\\
\calg_4=& \{ \e_3 \}
\\
 \calg_5 =& 
\begin{cases}
\calg_1 [1,k_5] &  \parbox[t]{.75\textwidth}{ where $k_5 < s$ is the largest integer such that $f_1(e_{k_5})=f_1(\e_{3})$, if such a $k_5$ exists; }\\
 \{ e_0 \} &  \parbox[t]{.75\textwidth}{ otherwise, where $e_0$ is the unique edge of $G_1$ which is south or west and such that $f_1(e_0)=f_1(\e_3)$ }
\end{cases}\\
 \calg_6=&
 \begin{cases}
\ocalg_2 [d',k_6] &  \parbox[t]{.75\textwidth}{ where $k_6>1$ is the least integer such that $f_2(e'_{k_6-1})=f_1(\e_{3})$, if such a $k_6$ exists; } \\
 \{ e_0 \} &  \parbox[t]{.75\textwidth}{ otherwise, where $e'_{d'}$ is the unique edge of $G'_{d'}$ which is south or west and such that $f_1(e'_{d'})=f_1(\e_3)$. }
\end{cases}
\end{align*}

\begin{defn} \label {grafting}
 In the above situation, we say that the pair $(\calg_3 \sqcup \calg_4, \calg_5 \sqcup \calg_6 )$ is called the {\em resolution of the grafting of $\calg_2$ on $\calg_1$ in $\calg_s$ }and we denote it by $\graft_{s,\e_3} (\calg_1,\calg_2).$ 
\end{defn}

\begin{figure}
\begin{center}
\scalebox{0.8}{ 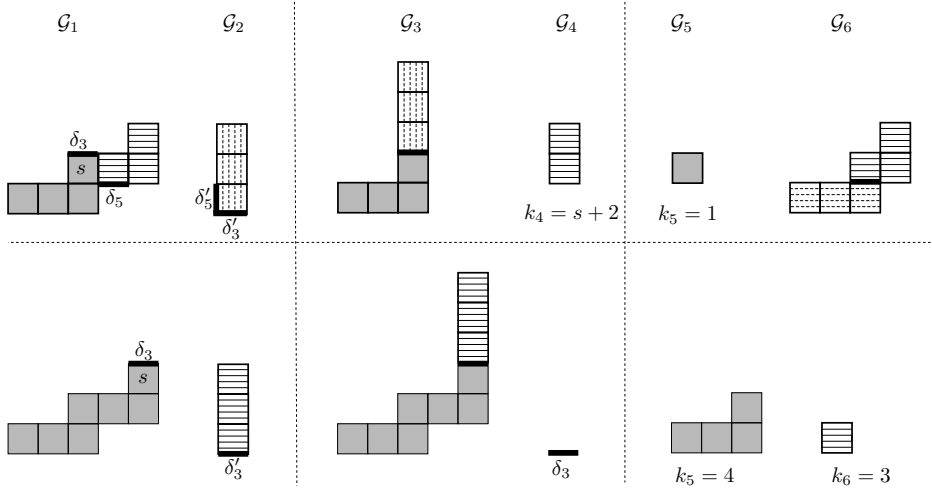}
\caption{Examples of resolutions of graftings: $1<s<d$ in the first row;  $s=d'$ in the second row.}
\label{figrafting}
\end{center}
\end{figure}

\section{Perfect Matchings}\label{sect pm}\label{sect 3}

Recall that a {\em perfect matching} $ P$ of a graph $G$ is a subset of the set of edges of $G$ such that each vertex of $G$ is incident to exactly one edge in $ P.$ Let $\match (G)$ denote the set of all perfect matchings of the graph $G$ .

The main result of this section is the following.

\begin{thm} \label {bijections} Let $\calg_1, \calg_2$ be two snake graphs. Then there are bijections
\begin{itemize}
 \item[(1)] $\match (\calg_1 \sqcup \calg_2) \xrightarrow {\varphi =(\varphi_{34}, \varphi_{56})}  \match( \res_{\calg} (\calg_1 \sqcup \calg_2))$
 \item[(2)] $\match (\calg_1 \sqcup \calg_2) \xrightarrow {\varphi =(\varphi_{34}, \varphi_{56})}   \match( \graft_{s, e_3} (\calg_1 \sqcup \calg_2))$
\end{itemize}
\end{thm}

\begin{proof}
We shall explicitly construct the bijection $\varphi$.
The proof that $\varphi  $ is a bijection is given in section \ref{sect 7}.

The idea for the bijection $\varphi$ is simple: Look for the first place in the overlap where one can `switch'  the matchings of $\calg_1 $ and $\calg_2$ in order to get a matching of $\calg_3\sqcup \calg_4$. If the given matching does not have such a  switching place at all, then its image under $\varphi $ is a matching of $\calg_5\sqcup \calg_6$.

If $ P$ is a matching of the snake graph $\calg,$ we denote by $ P[j_1,j_2]$ its restriction to the subgraph $\calg [j_1,j_2],$ and we define 
\begin{align*}
P (j_1,j_2]&= P [j_1,j_2] \backslash \{ e_{j_1-1} \}, \\
P[j_1,j_2) &=  P [j_1,j_2] \backslash \{ e_{j_2} \}.
\end{align*}

We first define $\varphi$ for statement (1). The statement is trivial if $\calg_1$ and $\calg_2$ do not have a crossing in $\calg.$ Suppose therefore $\calg_1$ and $\calg_2$ cross in $\calg$ and use the notation of the   Definition \ref{resolution} and Definition \ref{grafting}.
 Let $ P_i \in \match \calg_i.$

We define the map $\varphi$ by the following procedure. 
Suppose first that $s, s' \neq 1, t \neq d, t' \neq d'$ and $s \neq t.$
\begin{enumerate}
 \item[(i)] \label{?} If the pair of matchings $( P_1[s-1,s+1], \  P_2[s'-1,s'+1])$ on the pair of subgraphs $(\calg_1[s-1,s+1], \calg_2[s'-1,s'+1])$ is one of the eight configurations on the left in Figure \ref{mu} then let $\varphi ( P_1, P_2)$ be
\begin{align*}
(  P_1[1,{s-1}) \cup  \mu_{s,1} \cup  P_2(s'+1, d'], \  P_2[1,{s'-1}) \cup  \mu_{s,2} \cup  P_1(s+1, d]) 
\end{align*}
 \item[(ii)] \label{??} If  (i) does not apply, let $j$ be the least integer such that $1 < j < t-s-2$ and the local configuration of $( P_1, P_2)$ on $(\calg_1[s+j,s+j+1], \calg_2[s'+j, s'+j+1])$ is one of the four configurations shown in Figure \ref{figrho}, if such $j$ exists, and let $\varphi ( P_1, P_2)$ be 
\begin{align*}
 (  P_1[1,{s+j-1}) \cup  \rho_{j,1} \cup  P_2(s'+j+2, d'],  \\P_2[1,{s'+j-1}) \cup  \rho_{j,2} \cup  P_1(s+j+2, d])
\end{align*}
\item[(iii)] \label{???} If (i) and (ii) do not apply and the local configuration of $( P_1, P_2)$ on $(\calg_1[t-1,t+1], \ \calg_2[t'-1, t'+1])$ is one of the eight shown in Figure \ref{mu}, relabeling $s=t, s-1=t+1, s+1=t-1, s'=t', s'-1=t'+1, s'+1=t'-1,$ let $\varphi ( P_1, P_2)$ be
\begin{align*}
(  P_1[1,t-1) \cup  \mu_{t,1} \cup  P_2(t'+1, d'],\  P_2[1,t'-1) \cup  \mu_{t,2} \cup  P_1(t+1, d]) 
\end{align*}
\item[(iv)] \label{????} If (i)-(iii) do not apply then let $a_5$ (respectively $a_6$) be the glueing edge in the definition of $\calg_5$ (respecetively $\calg_6$) in Definition \ref{resolution}. Then let $\varphi ( P_1, P_2)$ be
\begin{align*}
(  P_1[1,s-1] \sqcup   P_2 [1, s'-1]\backslash \{a_5\},  P_2[t'+1,d'] \sqcup   P_1[t+2,d] \backslash \{a_6\}) 
\end{align*}
where the notation $A \sqcup B \backslash \{a\}$ means
\begin{align*}
& A \cup B \backslash \{a\} &\mbox{  if  } a \not \in A \cap B\\
& A \cup B & \mbox{  if  } a  \in A \cap B
\end{align*}
\end{enumerate}

\begin{figure}
\begin{center}
 \scalebox{0.8}{ 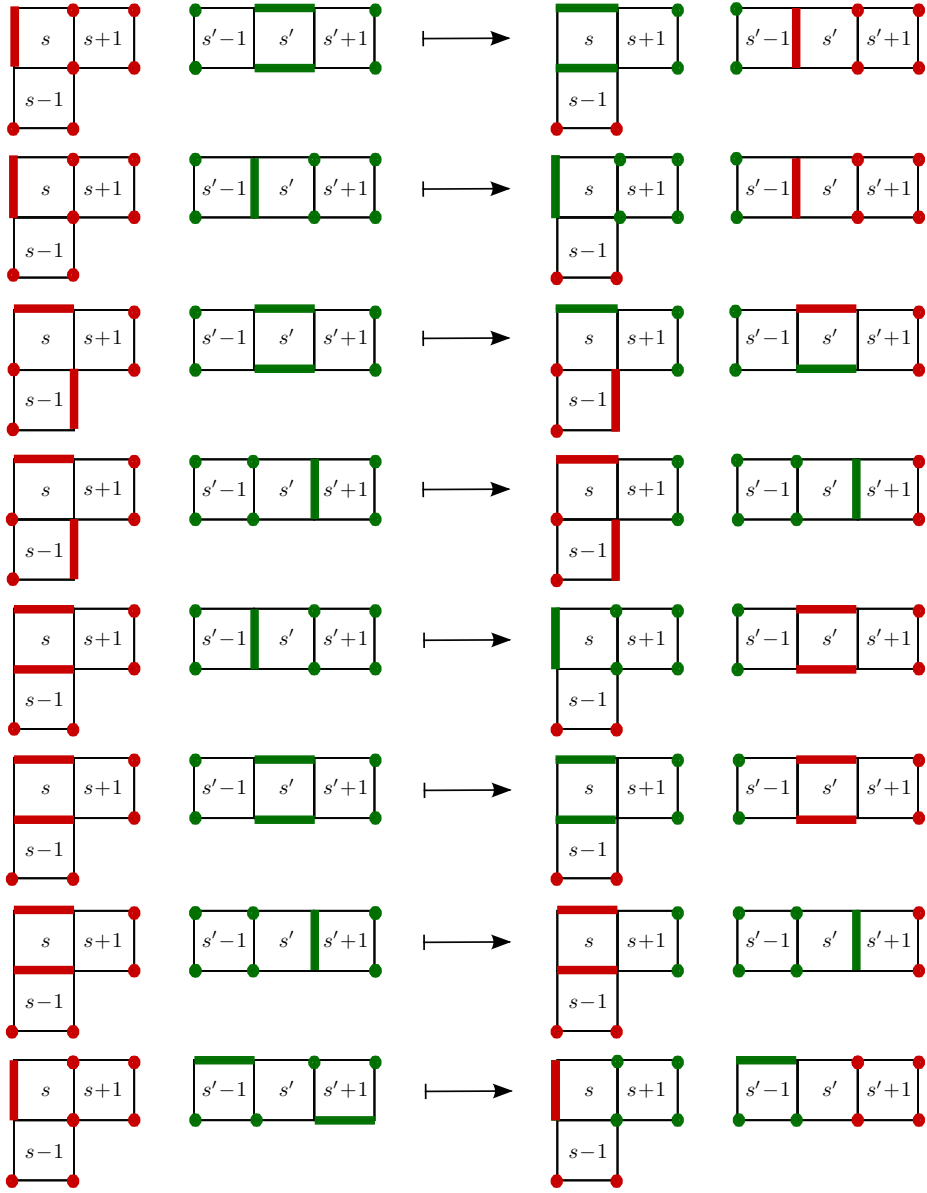}
 \caption{The operation $\mu$ applies in the 8 cases for $(P_1[s-1,s+1],P_2[s'-1,s'+1])$ shown on the left. Colored edges must belong to the matchings, colored vertices can be matched arbitrarily. The resulting pairs $(\mu_{s,1},\mu_{s,2})$ are shown on the right. Colors indicate whether the edges belong to $P_1$ or $P_2$.}
 \label{mu}
\end{center}
\end{figure}

\begin{figure} 
\begin{center}
 \scalebox{0.8}
 { 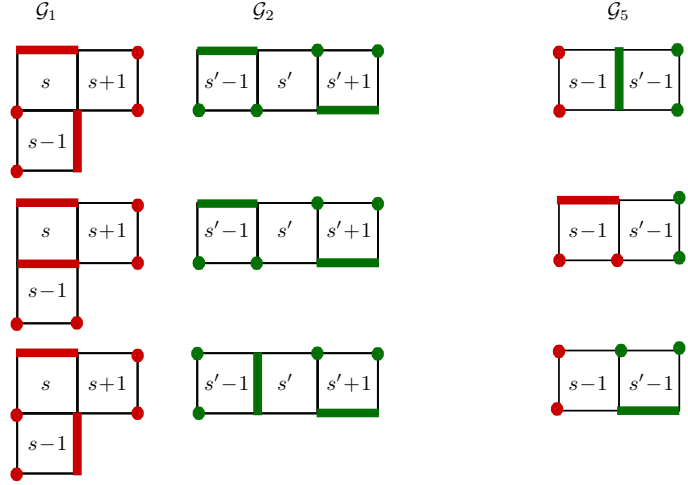}
 \caption{The 3 cases in which the operation $\mu$ does not apply}
 \label{mu2}
\end{center}
\end{figure}

\begin{figure}
\begin{center}
 \scalebox{0.8}
 { 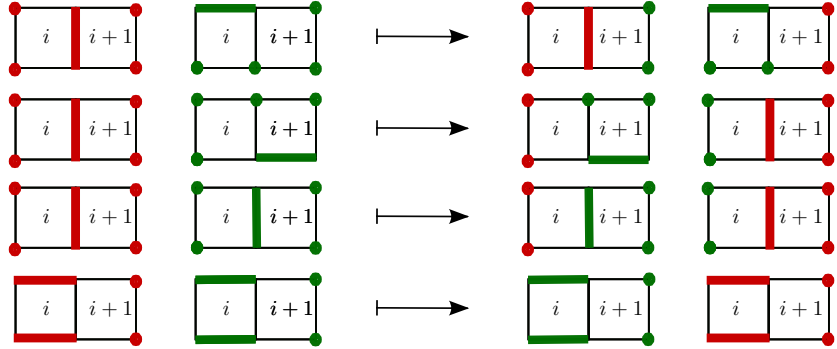}
 \caption{The operation $\rho$ applies in the 4 cases shown on the left. The resulting pairs $(\rho_{i,1},\rho_{i,2})$ are shown on the right.}
 \label{figrho}
\end{center}
\end{figure}

\begin{figure} 
\begin{center}
 \scalebox{0.8}
 { 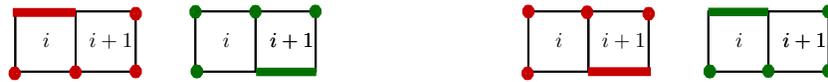}
 \caption{The 2 cases in which the operation $\rho$ does not apply}
 \label{figrho2}
\end{center}
\end{figure}

Note that in cases (i)-(iii), $\bij 12 \in \ma 34$ and in case   (iv), $\bij 12 \in \ma 56.$

If $s\neq t$ and, $\tpp$ respectively, we define $\bij 12$ by following the steps (i)-(iv) ignoring tiles $\g {s-1}, \gp {s'-1}, \g {t+1} $ or $\gp {t'+1}$ respectively, and restricting to $\calg_5 \cup \calg_6$   if step (iv) has been applied.

Finally if $s=t$ then $\varphi$ is defined by step (i) replacing the operation $\mu$ by the operation $\nu$ in Figure \ref{nu} and step (iv) only.

\begin{figure}\begin{center}
 \scalebox{0.70}{{  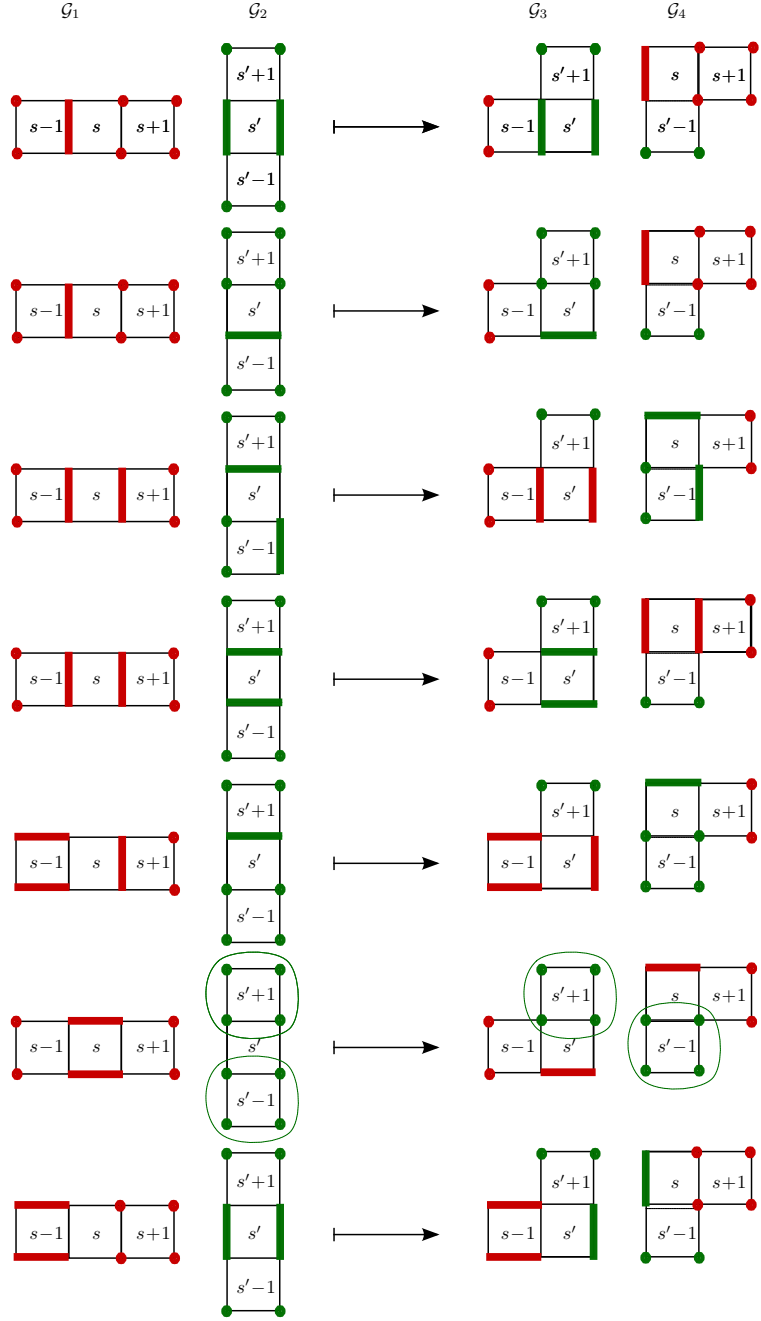}}
 \caption{The operation $\nu$. The snake graphs $\mathcal{G}_1[s-1,s+1]$ and $\mathcal{G}_2[s'-1,s'+1]$ (left) are straight snake graphs and the overlap is a single tile. The resulting pair $(\nu_{s,1},\nu_{s,2})$ is shown on the right. Two vertices in two different green circles cannot be matched to each other.}
 \label{nu}
 \end{center}
\end{figure}

Now we define $\varphi$ for the statement (2).  Let $\sigma_s=(\sigma_{s,3}, \sigma_{s,5})$ be the map described in Figure \ref{sigma}. Then define 
\begin{align*}
 \varphi : \match (\calg_1 \sqcup \calg_2) \longrightarrow \match (\graft _{s,e_3} (\calg_1,\calg_2))
\end{align*}
as follows.
\begin{figure}
\begin{center}
\scalebox{0.8}{ 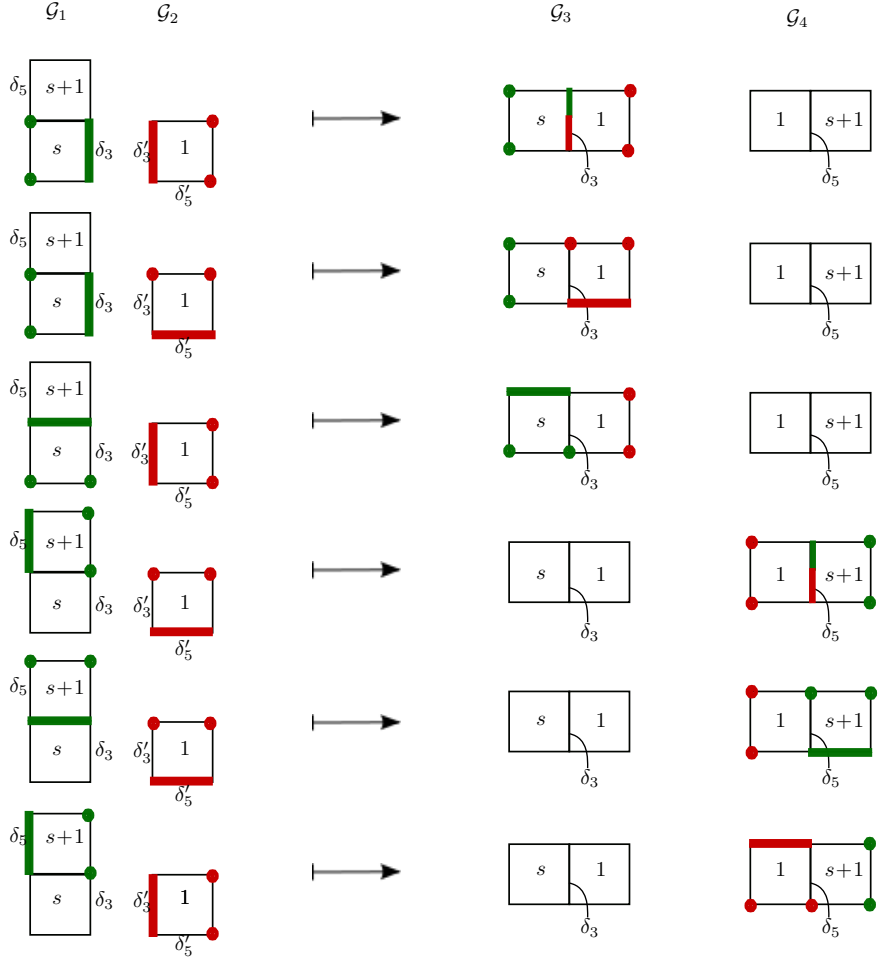}
\caption{The operation $\sigma$}
\label{sigma}
\end{center}
\end{figure}

\begin{itemize}
 \item[(i)] $\varphi (P_1,P_2)=(P[1,s-1] \cup \sigma_{s,3} \cup P_2[2,d'], P_1[k_4,d] )$ if the local configuration of $(P_1,P_2)$ on $(\calg_1[s,s+1],\calg_2[1,1])$ is one of the first three cases in Figure~\ref{sigma}, and
 \item[(ii)] $\varphi(P_1,P_2)=(P_1[s+2,d] \cup \sigma_{s,5} \cup P_2[2,d'], P_1[1,k_6])$ if the local configuration of $(P_1,P_2)$ on $(\calg_1[s,s+1],\calg_2[1,1])$ is one of the last three cases in Figure~\ref{sigma}.
\end{itemize}

Here we agree that if $s=d$ then $\varphi$ is defined by step (i) in the first three cases of Figure \ref{sigma}, where we delete the tiles $G$, and, in case five of Figure \ref{sigma}, we have $\varphi(P_1,P_2)=(P_2[k_5,d'], P_1[1,k_6]).$
In section  \ref{sect 7}, we proof the theorem by constructing the inverse map of $\varphi$.\end{proof}

\section{Snake graphs of cluster variables}\label{sect 4}

In this section we recall how snake graphs arise naturally in the theory of cluster algebras. We follow the exposition in \cite{MSW2}.

\subsection{Cluster algebras}\label{sect cluster algebras}
To define  a cluster algebra~$\Acal$ we must first fix its
ground ring.
Let $(\PP,\oplus, \cdot)$ be a \emph{semifield}, i.e.,
an abelian multiplicative group endowed with a binary operation of
\emph{(auxiliary) addition}~$\oplus$ which is commutative, associative, and
distributive with respect to the multiplication in~$\PP$.
The group ring~$\ZZ\PP$ will be
used as a \emph{ground ring} for~$\Acal$.
One important choice for $\PP$ is the tropical semifield; in this case we say that the
corresponding cluster algebra is of {\it geometric type}.
Let $\Trop (u_1, \dots, u_{m})$ be an abelian group (written
multiplicatively) freely generated by the $u_j$.
We define  $\oplus$ in $\Trop (u_1,\dots, u_{m})$ by
\begin{equation}
\label{eq:tropical-addition}
\prod_j u_j^{a_j} \oplus \prod_j u_j^{b_j} =
\prod_j u_j^{\min (a_j, b_j)} \,,
\end{equation}
and call $(\Trop (u_1,\dots,u_{m}),\oplus,\cdot)$ a \emph{tropical
 semifield}.
Note that the group ring of $\Trop (u_1,\dots,u_{m})$ is the ring of Laurent
polynomials in the variables~$u_j\,$.

As an \emph{ambient field} for
$\Acal$, we take a field $\Fcal$
isomorphic to the field of rational functions in $n$ independent
variables (here $n$ is the \emph{rank} of~$\Acal$),
with coefficients in~$\QQ \PP$.
Note that the definition of $\Fcal$ does not involve
the auxiliary addition
in~$\PP$.

\begin{defn}
\label{def:seed}
A \emph{labeled seed} in~$\Fcal$ is
a triple $(\xx, \yy, B)$, where
\begin{itemize}
\item
$\xx = (x_1, \dots, x_n)$ is an $n$-tuple 
from $\Fcal$
forming a \emph{free generating set} over $\QQ \PP$,
\item
$\yy = (y_1, \dots, y_n)$ is an $n$-tuple
from $\PP$, and
\item
$B = (b_{ij})$ is an $n\!\times\! n$ integer matrix
which is \emph{skew-symmetrizable}.
\end{itemize}
That is, $x_1, \dots, x_n$
are algebraically independent over~$\QQ \PP$, and
$\Fcal = \QQ \PP(x_1, \dots, x_n)$.
We refer to~$\xx$ as the (labeled)
\emph{cluster} of a labeled seed $(\xx, \yy, B)$,
to the tuple~$\yy$ as the \emph{coefficient tuple}, and to the
matrix~$B$ as the \emph{exchange matrix}.
\end{defn}

We obtain ({\it unlabeled}) {\it seeds} from labeled seeds
by identifying labeled seeds that differ from
each other by simultaneous permutations of
the components in $\xx$ and~$\yy$, and of the rows and columns of~$B$.

We  use the notation
$[x]_+ = \max(x,0)$,
$[1,n]=\{1, \dots, n\}$, and
\begin{align*}
\sgn(x) &=
\begin{cases}
-1 & \text{if $x<0$;}\\
0  & \text{if $x=0$;}\\
 1 & \text{if $x>0$.}
\end{cases}
\end{align*}

\begin{defn}
\label{def:seed-mutation}
Let $(\xx, \yy, B)$ be a labeled seed in $\Fcal$,
and let $k \in [1,n]$.
The \emph{seed mutation} $\mu_k$ in direction~$k$ transforms
$(\xx, \yy, B)$ into the labeled seed
$\mu_k(\xx, \yy, B)=(\xx', \yy', B')$ defined as follows:
\begin{itemize}
\item
The entries of $B'=(b'_{ij})$ are given by
\begin{equation}
\label{eq:matrix-mutation}
b'_{ij} =
\begin{cases}
-b_{ij} & \text{if $i=k$ or $j=k$;} \\[.05in]
b_{ij} + \sgn(b_{ik}) \ [b_{ik}b_{kj}]_+
 & \text{otherwise.}
\end{cases}
\end{equation}
\item
The coefficient tuple $\yy'=(y_1',\dots,y_n')$ is given by
\begin{equation}
\label{eq:y-mutation}
y'_j =
\begin{cases}
y_k^{-1} & \text{if $j = k$};\\[.05in]
y_j y_k^{[b_{kj}]_+}
(y_k \oplus 1)^{- b_{kj}} & \text{if $j \neq k$}.
\end{cases}
\end{equation}
\item
The cluster $\xx'=(x_1',\dots,x_n')$ is given by
$x_j'=x_j$ for $j\neq k$,
whereas $x'_k \in \Fcal$ is determined
by the \emph{exchange relation}
\begin{equation}
\label{exchange relation}
x'_k = \frac
{y_k \ \prod x_i^{[b_{ik}]_+}
+ \ \prod x_i^{[-b_{ik}]_+}}{(y_k \oplus 1) x_k} \, .
\end{equation}
\end{itemize}
\end{defn}

We say that two exchange matrices $B$ and $B'$ are {\it mutation-equivalent}
if one can get from $B$ to $B'$ by a sequence of mutations.
\begin{defn}
\label{def:patterns}
Consider the \emph{$n$-regular tree}~$\TT_n$
whose edges are labeled by the numbers $1, \dots, n$,
so that the $n$ edges emanating from each vertex receive
different labels.
A \emph{cluster pattern}  is an assignment
of a labeled seed $\Sigma_t=(\xx_t, \yy_t, B_t)$
to every vertex $t \in \TT_n$, such that the seeds assigned to the
endpoints of any edge $t \overunder{k}{} t'$ are obtained from each
other by the seed mutation in direction~$k$.
The components of $\Sigma_t$ are written as:
\begin{equation}
\label{eq:seed-labeling}
\xx_t = (x_{1;t}\,,\dots,x_{n;t})\,,\quad
\yy_t = (y_{1;t}\,,\dots,y_{n;t})\,,\quad
B_t = (b^t_{ij})\,.
\end{equation}
\end{defn}

Clearly, a cluster pattern  is uniquely determined
by an arbitrary  seed.

\begin{defn}
\label{def:cluster-algebra}
Given a cluster pattern, we denote
\begin{equation}
\label{eq:cluster-variables}
\Xcal
= \bigcup_{t \in \TT_n} \xx_t
= \{ x_{i,t}\,:\, t \in \TT_n\,,\ 1\leq i\leq n \} \ ,
\end{equation}
the union of clusters of all the seeds in the pattern.
The elements $x_{i,t}\in \Xcal$ are called \emph{cluster variables}.
The 
\emph{cluster algebra} $\Acal$ associated with a
given pattern is the $\ZZ \PP$-subalgebra of the ambient field $\Fcal$
generated by all cluster variables: $\Acal = \ZZ \PP[\Xcal]$.
We denote $\Acal = \Acal(\xx, \yy, B)$, where
$(\xx,\yy,B)$
is any seed in the underlying cluster pattern.
\end{defn}

\subsection{Cluster algebras arising from unpunctured
    surfaces}\label{sect surfaces} 
%

Let $S$ be a connected oriented 2-dimensional Riemann surface with
nonempty
boundary, and let $M$ be a nonempty finite subset of the boundary of $S$, such that each boundary component of $S$ contains at least one point of $M$. The elements of $M$ are called {\it marked points}. The
pair $(S,M)$ is called a \emph{bordered surface with marked points}.
 
For technical reasons, we require that $(S,M)$ is not
a disk with 1,2 or 3 marked points.

\begin{defn}
An \emph{arc} $\zg$ in $(S,M)$ is a curve in $S$, considered up
to isotopy, such that: 
\begin{itemize}
\item[(a)] 
the endpoints of $\zg$ are in $M$;
\item[(b)] 
$\zg$ does not cross itself, except that its endpoints may coincide;
\item[(c)] 
except for the endpoints, $\zg$ is disjoint from   the boundary of $S$; and
\item[(d)] 
$\zg$ does not cut out a monogon or a bigon. 
\end{itemize}   
\end{defn}     

Curves that connect two
marked points and lie entirely on the boundary of $S$ without passing
through a third marked point are \emph{boundary segments}.
Note that boundary segments are not   arcs.

\begin{defn}
For any two arcs $\zg,\zg'$ in $S$, let $e(\zg,\zg')$ be the minimal
number of crossings of 
arcs $\za$ and $\za'$, where $\za$ 
and $\za'$ range over all arcs isotopic to 
$\zg$ and $\zg'$, respectively.
We say that arcs $\zg$ and $\zg'$ are  \emph{compatible} if $e(\zg,\zg')=0$. 
\end{defn}

\begin{defn}
A \emph{triangulation} is a maximal collection of
pairwise compatible arcs (together with all boundary segments). 
\end{defn}

\begin{defn}
Triangulations are connected to each other by sequences of 
{\it flips}.  Each flip replaces a single arc $\gamma$ 
in a triangulation $T$ by a (unique) arc $\gamma' \neq \gamma$
that, together with the remaining arcs in $T$, forms a new 
triangulation.
\end{defn}

\begin{defn}
Choose any   triangulation
$T$ of $(S,M)$, and let $\tau_1,\tau_2,\ldots,\tau_n$ be the $n$ arcs of
$T$.
For any triangle $\Delta$ in $T$, we define a matrix 
$B^\Delta=(b^\Delta_{ij})_{1\le i\le n, 1\le j\le n}$  as follows.
\begin{itemize}
\item $b_{ij}^\Delta=1$ and $b_{ji}^{\Delta}=-1$ if $\tau_i$ and $\tau_j$ are sides of 
  $\Delta$ with  $\tau_j$ following $\tau_i$  in the 
  clockwise order.
\item $b_{ij}^\Delta=0$ otherwise.
\end{itemize}
 
Then define the matrix 
$ B_{T}=(b_{ij})_{1\le i\le n, 1\le j\le n}$  by
$b_{ij}=\sum_\Delta b_{ij}^\Delta$, where the sum is taken over all
triangles in $T$.
\end{defn}

Note that  $B_{T}$ is skew-symmetric and each entry  $b_{ij}$ is either
$0,\pm 1$, or $\pm 2$, since every arc $\tau$ is in at most two triangles.

\begin{thm} \cite[Theorem 7.11]{FST} and \cite[Theorem 5.1]{FT}
\label{clust-surface}
Fix a bordered surface $(S,M)$ and let $\Acal$ be the cluster algebra associated to
the signed adjacency matrix of a   triangulation. Then the (unlabeled) seeds $\Sigma_{T}$ of $\Acal$ are in bijection
with  the triangulations $T$ of $(S,M)$, and
the cluster variables are  in bijection
with the arcs of $(S,M)$ (so we can denote each by
$x_{\gamma}$, where $\gamma$ is an arc). Moreover, each seed in $\Acal$ is uniquely determined by its cluster.  Furthermore,
if a   triangulation $T'$ is obtained from another
  triangulation $T$ by flipping an arc $\gamma\in T$
and obtaining $\gamma'$,
then $\Sigma_{T'}$ is obtained from $\Sigma_{T}$ by the seed mutation
replacing $x_{\gamma}$ by $x_{\gamma'}$.
\end{thm}

\subsection{Skein relations}
In this section we review some results from \cite{MW}.
\begin{defn}   \label{gen-arc}
A \emph{generalized arc}  in $(S,M)$ is a curve $\gamma$ in $S$ such that:
\begin{itemize}
\item[(a)] the endpoints of $\gamma$ are in $M$; 
\item[(b)] except for the endpoints,
$\gamma$ is disjoint  from the boundary of $S$; and
\item[(c)]
$\gamma$ does not cut out a monogon or a bigon.
\end{itemize}
\end{defn}

Note that we allow a generalized arc to cross itself a finite 
number of times.  We consider generalized 
arcs up to isotopy (of immersed
arcs).

\begin{defn}
A \emph{closed loop} in $(S,M)$ is a closed curve
$\gamma$ in $S$ which is disjoint from the 
boundary of $S$.  We allow a closed loop to have a finite
number of self-crossings.
As in Definition \ref{gen-arc}, we consider closed
loops up to isotopy.
\end{defn}

\begin{defn}
A closed loop in $(S,M)$ is called \emph{essential} if
  it is not contractible
and it does not have self-crossings.
\end{defn}

\begin{defn} (Multicurve)
We define a \emph{multicurve} to be a finite multiset of generalized  
arcs and closed loops such that there are only a finite number of pairwise crossings among the collection.
We say that a multicurve is \emph{simple}
if there are no pairwise crossings among the collection and 
no self-crossings.
\end{defn}

If a multicurve is not simple, 
then there are two ways to \emph{resolve} a crossing to obtain a multicurve that no longer contains this crossing and has no additional crossings.  This process is known as \emph{smoothing}.

\begin{defn}\label{def:smoothing} (Smoothing) Let $\gamma, \gamma_1$ and $\gamma_2$ be generalized  
arcs or closed loops such that we have
one of the following two cases:

\begin{enumerate}
 \item $\gamma_1$ crosses $\gamma_2$ at a point $x$,
  \item $\gamma$ has a self-crossing at a point $x$.
\end{enumerate}

\noindent Then we let $C$ be the multicurve $\{\gamma_1,\gamma_2\}$ or $\{\gamma\}$ depending on which of the two cases we are in.  We define the \emph{smoothing of $C$ at the point $x$} to be the pair of multicurves $C_+ = \{\alpha_1,\alpha_2\}$ (resp. $\{\alpha\}$) and $C_- = \{\beta_1,\beta_2\}$ (resp. $\{\beta\}$).

Here, the multicurve $C_+$ (resp. $C_-$) is the same as $C$ except for the local change that replaces the crossing {\Large $\times$} with the pair of segments 
  $\genfrac{}{}{0pt}{5pt}{\displaystyle\smile}{\displaystyle\frown}$ (resp. { $\supset \subset$}).    
\end{defn}   

Since a multicurve  may contain only a finite number of crossings, by repeatedly applying smoothings, we can associate to any multicurve  a collection of simple multicurves.  
  We call this resulting multiset of multicurves the \emph{smooth resolution} of the multicurve $C$.

\begin{thm}\label{th:skein1}(Skein relations) \cite[Propositions 6.4,6.5,6.6]{MW}
Let $C,$ $ C_{+}$, and $C_{-}$ be as in Definition \ref{def:smoothing}. 
Then we have the following identity in $\Aprin(B_T)$,
\begin{equation*}
x_C = \pm Y_1 x_{C_+} \pm Y_2 x_{C_-},
\end{equation*}
where $Y_1$ and $Y_2$ are monomials in the variables $y_{\tau_i}$.    
The monomials $Y_1$ and $Y_2$ can be expressed using the intersection numbers of the elementary laminations (associated to triangulation $T$) with  
the curves in $C,C_+$ and $C_-$.
\end{thm}

\subsection{Snake graphs from surfaces}\label{sect tiles}\label{sect graph}
%
%

 Let 
$\zg$ be an   arc in $(S,M)$ which is not in $T$. 
Choose an orientation on $\zg$, let $s\in M$ be its starting point, and let $t\in M$ be its endpoint.
We denote by
$s=p_0, p_1, p_2, \ldots, p_{d+1}=t$
the points of intersection of $\zg$ and $T$ in order.  
Let $\tau_{i_j}$ be the arc of $T$ containing $p_j$, and let 
$\zD_{j-1}$ and 
$\zD_{j}$ be the two   triangles in $T$ 
on either side of 
$\tau_{i_j}$. Note that each of these triangles has three distinct sides, but not necessarily three distinct vertices, see Figure \ref{figr1}.
\begin{figure}
\includegraphics{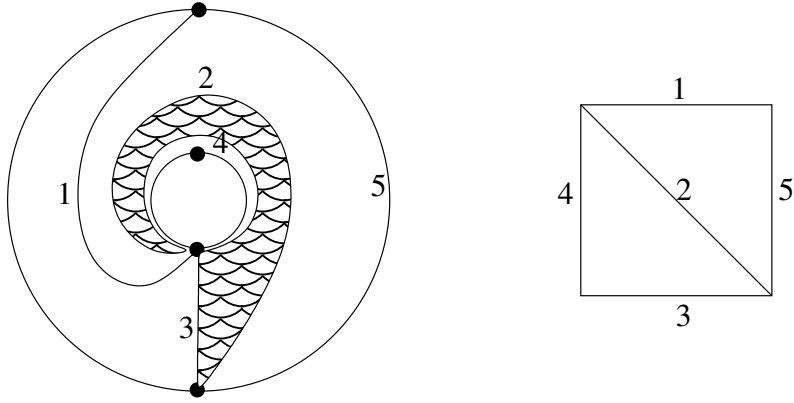}
\caption{On the left, a triangle with two vertices; on the right the tile $G_{j}$ where $i_j=2$. }\label{figr1}
\end{figure}
Let $G_j$ be the graph with 4 vertices and 5 edges, having the shape of a square with a diagonal, such that there is a bijection between the edges of $G_j$ and the 5 arcs in the two triangles $\zD_{j-1}$ and $\zD_j$, which preserves the signed adjacency of the arcs up to sign and such that the diagonal in $G_j$ corresponds to the arc $\tau_{i_j}$ containing the crossing point $p_j$. Thus $G_j$ is given by the quadrilateral in the triangulation $T$ whose diagonal is $\tau_{i_j}$. 

\begin{defn} \label{relorient}  
Given a planar embedding $\tilde G_j$ 
of $G_j$, we define the \emph{relative orientation} 
$\mathrm{rel}(\tilde G_j, T)$ 
of $\tilde G_j$ with respect to $T$ 
to be $\pm 1$, based on whether its triangles agree or disagree in orientation with those of $T$.  
\end{defn}
For example, in Figure \ref{figr1},  $\tilde G_j$ has relative orientation $+1$.

Using the notation above, 
the arcs $\tau_{i_j}$ and $\tau_{i_{j+1}}$ form two edges of a triangle $\zD_j$ in $T$.  Define $\tau_{[j]}$ to be the third arc in this triangle.

We now recursively glue together the tiles $G_1,\dots,G_d$
in order from $1$ to $d$, so that for two adjacent   tiles, 
 we glue  $G_{j+1}$ to $\tilde G_j$ along the edge 
labeled $\tau_{[j]}$, choosing a planar embedding $\tilde G_{j+1}$ for $G_{j+1}$
so that $\mathrm{rel}(\tilde G_{j+1},T) \not= \mathrm{rel}(\tilde G_j,T).$  See Figure \ref{figglue}.

\begin{figure}
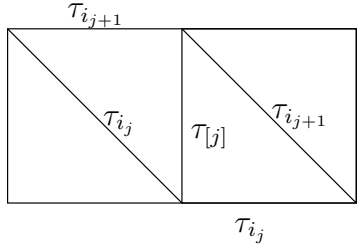
\caption{Gluing tiles $\tilde G_j$ and $\tilde G_{j+1}$ along the edge labeled  $\tau_{[j]}$}
\label{figglue}
\end{figure}

After gluing together the $d$ tiles, we obtain a graph (embedded in 
the plane),
which we denote  by
${\calg}^\triangle_{\zg}$. 
\begin{defn}
The \emph{snake graph} $\calg_{\gamma}$ associated to $\gamma$ is obtained 
from ${\calg}^\triangle_{\zg}$ by removing the diagonal in each tile.
\end{defn}

In Figure \ref{figsnake}, we give an example of an 
 arc $\gamma$ and the corresponding snake graph 
${\calg}_{\zg}$. Since $\gamma$ intersects $T$
five times,  
${\calg}_{\zg}$ has five tiles. 

 \begin{defn}If $\tau \in T$ then we define its snake graph $\calg_{\tau}$ to be the graph consisting of one single edge with weight $x_{\tau}$ and two distinct endpoints (regardless whether the endpoints of $\tau$ are distinct).
 \end{defn}

\begin{figure}
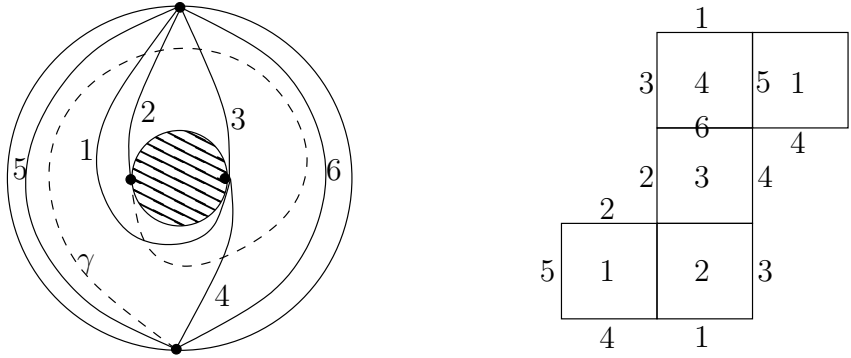
\caption{An arc $\gamma$ in a triangulated annulus on the left and the corresponding snake graph $\calg_{\gamma}$ on the right. The tiles labeled 1,3,1 have positive relative orientation and the tiles 2,4 have negative relative orientation.}\label{figsnake}
\end{figure}

\subsection{Snake graph formula for cluster variables}
Recall that if $\tau$ is a boundary segment then $x_{\tau} = 1$,

\begin{defn}   
If $\zg$ is a generalized  arc and  $\tau_{i_1}, \tau_{i_2},\dots, \tau_{i_d}$
is the sequence of arcs in $T$ which $\zg$ crosses, we define the \emph{crossing monomial} 
of $\gamma$ with respect to $T$ to be
$$\mathrm{cross}(T, \gamma) = \prod_{j=1}^d x_{\tau_{i_j}}.$$
\end{defn} 

\begin{defn}  
A \emph{perfect matching} of a graph $\calg$ is a subset $P$ of the 
edges of $\calg$ such that
each vertex of $\calg$ is incident to exactly one edge of $P$. 
If $\calg$ is a snake graph and the edges of  a perfect matching $P$ of  
$\calg $ are labeled $\tau_{j_1},\dots,\tau_{j_r}$, then 
we define the {\it weight} $x(P)$ of $P$ to be 
$x_{\tau_{j_1}} \dots x_{\tau_{j_r}}$.
\end{defn}

\begin{defn}   Let $\zg$ be a generalized arc.
By induction on the number of tiles it is easy to see that the snake graph
$\calg_{\zg}$  
has  precisely two perfect matchings which we call
the {\it minimal matching} $P_-=P_-(\calg_{\zg})$ and 
the {\it maximal matching} $P_+
=P_+(\calg_{\zg})$, 
which contain only boundary edges.
To distinguish them, 
if  $\mathrm{rel}(\tilde G_1,T)=1$ (respectively, $-1$),
we define 
$e_1$ and $e_2$ to be the two edges of 
${\calg}^\triangle_{\zg}$ which lie in the counterclockwise 
(respectively, clockwise) direction from 
the diagonal of $\tilde G_1$.  Then  $P_-$ is defined as
the unique matching which contains only boundary 
edges and does not contain edges $e_1$ or $e_2$.  $P_+$
is the other matching with only boundary edges.
\end{defn} 
In the example of Figure \ref{figsnake}, the minimal matching $P_-$ contains the bottom edge of the first tile labeled 4.

\begin{lem}\cite[Theorem 5.1]{MS}
\label{thm y}
The symmetric difference $P_-\ominus P$ is the set of boundary edges of a 
(possibly disconnected) subgraph $\calg_P$ of $\calg_\zg$,
which is a union of cycles.  These cycles enclose a set of tiles 
$\cup_{j\in J} G_{j}$,  where $J$ is a finite index set.
\end{lem}

\begin{defn} \label{height} With the notation of Lemma \ref{thm y},
we  define the \emph{height monomial} $y(P)$ of a perfect matching $P$ of a snake graph $\calg_\gamma$ by
\begin{equation*}
y(P) = \prod_{j\in J} y_{\tau_{i_j}}.
\end{equation*}\end{defn}

Following \cite{MSW2}, for each generalized arc $\gamma$, we now define a Laurent polynomial $x_\zg$, as well as a polynomial 
$F_\zg^T$ obtained from $x_{\gamma}$ by specialization.
\begin{defn} \label{def:matching}
Let $\zg$ 
be a generalized arc and let $\calg_\zg$,
be its snake graph.  
\begin{enumerate}
\item If $\zg$ 
has a contractible kink, let $\overline{\zg}$ denote the 
corresponding generalized arc with this kink removed, and define 
$x_{\zg} = (-1) x_{\overline{\zg}}$.  
\item Otherwise, define
\[ x_{\gamma}= \frac{1}{\mathrm{cross}(T,\zg)} \sum_P 
x(P) y(P),\]
 where the sum is over all perfect matchings $P$ of $G_{\zg}$.
\end{enumerate}
Define $F_{\gamma}^T$ to be the polynomial obtained from 
$x_{\gamma}$ by specializing all the $x_{\tau_i}$ to $1$.

If $\gamma$ is a curve that 
cuts out a contractible monogon, then we define $\gamma =0$.     
\end{defn}

\begin{thm}\cite[Thm 4.9]{MSW}
\label{thm MSW}
If $\gamma$ is an arc, then 
$x_{\gamma}$ 
is a the cluster variable in $ \A$,
written as a Laurent expansion with respect to the seed $\Sigma_T$,
and $F_{\gamma}^T$ is its \emph{F-polynomial}.
\end{thm}

\subsection{Fans}
Let $T$ be a triangulation and $\zg$ be an arc.
Let $\zD$ be a triangle in $T$ with sides $\zb_1,\zb_2$, and  $\tau$, that is crossed by $\zg$ in the following way: $\zg$ crosses  $\zb_1 $ at the point $p_1$ and crosses $\zb_2$ at the point $p_2$, and the segment of  $\zg$ from $p_1$ to $p_2$ lies entirely in $\zD$, see the left of Figure \ref{figv}. Then there exists a unique vertex $v$ of the triangle $\zD$ and a unique contractible closed curve $\ze$ given as the homotopy class of a curve starting at the point $v$, then following $\zb_1$
until the point $p_1$, then following  $\gamma$ until the point $p_2$ and then following $\zb_2$ until $v$. We will use the following notation to describe this definition: 

\[\epsilon =\xymatrix{ v \ar@{-}[r]^{\zb_1} &p_1\ar@{-}[r]^\zg & p_2 \ar@{-}[r]^{\zb_2} &v}.\]

\begin{defn} \label{def fan}
A \emph{$(T,\gamma)$-fan with vertex $v$} is a collection of arcs $\beta_0,\beta_1,\ldots,\beta_k$, with $\beta_i\in T$ and $k\ge 0$ with the following properties (see the right of Figure \ref{figv}):
\begin{enumerate}
\item $\gamma$ crosses $\beta_0,\beta_1,\ldots,\beta_k$ in order at the points $p_0,p_1,\ldots, p_k$, such that $p_i$ is a crossing point of $\gamma$ and $\beta_i$, and the segment of $\zg$ from $p_0$ to $p_k$ does not have any other crossing points with $T$;
\item each $\beta_i$ is incident to $v$;
\item for each $i<k$, let $\ze_i$  be the unique contractible closed curve given by 
\[\xymatrix{ v \ar@{-}[r]^{\zb_i} &p_i\ar@{-}[r]^-\zg & p_{i+1} \ar@{-}[r]^-{\zb_{i+1}} &v};\]
then 
 for each $i<k-1$, the concatenation of the curves $\ze_i\ze_{i+1}$ is homotopic to 
\[\xymatrix{ v \ar@{-}[r]^{\zb_i} &p_i\ar@{-}[r]^-\zg&p_{i+1}\ar@{-}[r]^-\zg & p_{i+2} \ar@{-}[r]^-{\zb_{i+2}} &v}.\]
\end{enumerate}
\end{defn}
Condition (3) in the above definition is equivalent
to the condition that \[\xymatrix{ v \ar@{-}[r]^{\zb_i} &p_i\ar@{-}[r]^-\zg & p_{i+2} \ar@{-}[r]^-{\zb_{i+2}} &v}\]
is contractible.
\begin{defn} A $(T,\gamma)$-fan
 $\beta_0,\beta_1,\ldots,\beta_k$ 
 is called \emph{maximal} if there is no arc $\za\in T$ such that 
 $\beta_0,\beta_1,\ldots,\beta_{k},\za$ or  $\za,\beta_0,\beta_1,\ldots,\beta_{k}$ is a  $(T,\gamma)$-fan.
 \end{defn}
 

\begin{figure}\begin{center}
\scalebox{0.8}{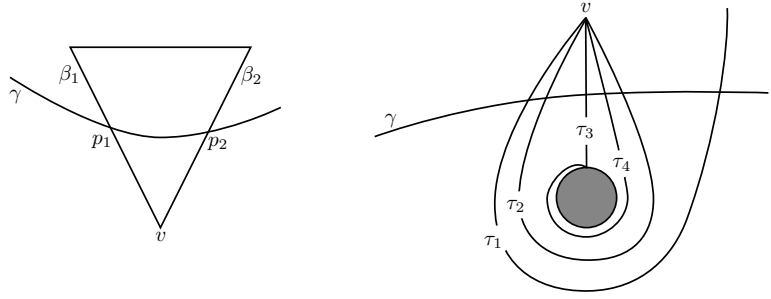}
\caption{Construction of $(T,\zg)$-fans (left). The fan $\tau_1,\tau_2$, $\tau_3$, $ \tau_4$, $\tau_2$ (right) can not be extended to the right, because the configuration $\tau_1,\tau_2,\tau_3,\tau_4,\tau_2,\tau_1$ does not satisfy condition (3) of Definition \ref{def fan}}.\label{figfan}\label{figv}\end{center}
\end{figure}

 Note that if $\tau\in T $ then $\calg_{\tau}$ has exactly one perfect matching $P$ and $x(P)=x_{\tau}$ and $y(P)=1.$ 

\section{Smoothing and snake graphs}\label{sect smoothing}\label{sect 5}

Let $(S,M,T)$ be a triangulated surface. Let $\gamma_1, \gamma_2$ be two arcs in general position in $(S,M)$ which are not in the triangulation $T$ and such that $\gamma_1$ crosses $\gamma_2$ at a point $p \in S \backslash \partial S.$

Fix an orientation on $\gamma_1$ and $\gamma_2$ and let $p_0,p_1,\ldots,p_{d+1}$ (respectively $p'_0$, $ p'_1$, $\ldots,$ $p'_{d'+1}$) be the crossing points of $\gamma_1$ (respectively $  \gamma_2$) with $T$ in order and including the starting point and the endpoint of $\gamma_1$ (respectively $\gamma_2$). Thus $\gamma_1$ runs from $p_0$ to $p_{d+1}$ and $\gamma_2$ from $p'_0$ to $p'_{d'+1}.$ Let $\Delta_j$ (respectively $\Delta '_{j'}$) be the triangle in $T$ that contains the segment of $\gamma_1$ (respectively $\gamma_2$) between the points $p_j$ and $p_{j+1},$ see Figure \ref{delta}. Finally, let $\tau_{i_j} \in T$ (respectively $\tau'_{i'_{j'}}$) denote the arc of the triangulation passing through $p_j$ (respectively $p'_{j'}$). Thus $\tau_{i_j} $ and $\tau_{i_{j+1}}$ are two sides of the triangle $\Delta_j.$ Denote the third side of $\Delta_j$ by $\tau_{[j]}.$ Similarly the third side of $\Delta'_{j'}$ is denoted by $\tau'_{[j']}.$
\begin{figure}
\begin{center}
 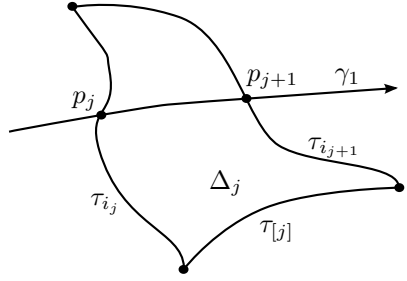
 \caption{The triangle $\Delta_j$}
 \label{delta}
 \end{center}
\end{figure}
Without loss of generality, we may assume that the crossing point of $\gamma_1$ and $\gamma_2$ lies in the interior of some triangle in $T$ and, since $\gamma_1$ and $\gamma_2$ both run through $p,$ this triangle is $\Delta_j$ for some $j,$ and $\Delta'_{j'}$ for some $j'.$

If $\Delta_j$ is the first or the last triangle that $\gamma_1$ meets then we assume without loss of generality that $\Delta_j$ is the first triangle, that is, $\Delta_j=\Delta_1.$ Similarly, if $\Delta'_{j'}$ is the first or the last triangle that $\gamma_2$ meets, then we assume without loss of generality that $\Delta'_{j'}=\Delta'_1.$ In all other cases, there is at least one side of $\Delta_j$ which is crossed both by $\gamma_1$ and $\gamma_2,$ more precisely, we have $\tau_{i_j}=\tau'_{i'_{j'}}$ or $\tau_{i_j}=\tau'_{i'_{j'+1}}$ or $\tau_{i_{j+1}}=\tau'_{i'_{j'}}$ or $\tau_{i_{j+1}}=\tau'_{i'_{j'+1}}.$ We can assume, by changing orientation if necessary, that $\tau_{i_j}=\tau'_{i'_{j'}}$ or $\tau_{i_{j+1}}=\tau'_{i'_{j'+1}}.$ Altogether, there are  five cases which are illustrated in Figure \ref{fivecrossings}.

\begin{figure}
\begin{center}
\scalebox{0.8}{ 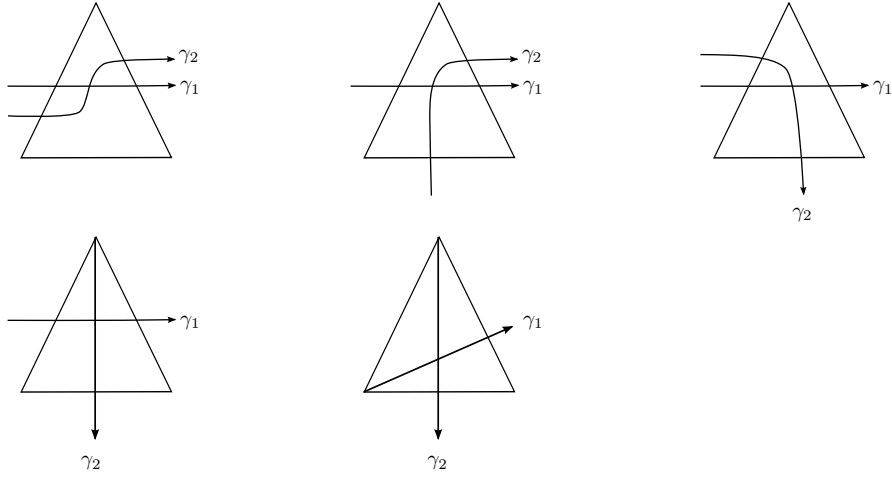}
 \caption{Five possible crossings in the triangle $\Delta_j$. In the top  left picture, $\tau_{i_j}=\tau'_{i'_{j'}}$ and $\tau_{i_{j+1}}=\tau'_{i'_{j'+1}}$,  in the top middle picture,  $\tau_{i_{j+1}}=\tau'_{i'_{j'+1}}$ and,  in the top right picture,
 $\tau_{i_j}=\tau'_{i'_{j'}}$.
}
 \label{fivecrossings}
 \end{center}
\end{figure}

\begin{defn}
 In the last two cases  in Figure \ref{fivecrossings}, we say that $\gamma_1$ and $\gamma_2$ {\it cross at $p$ with an empty overlap}  and in the first three cases $\gamma_1$ and $\gamma_2$ cross at $p$ with a {\it non-empty overlap.}
\end{defn}

Suppose now that $\gamma_1$ and $\gamma_2$ cross at $p$ with a non-empty local overlap. 
We want to define the local overlap of $\gamma_1 $ and $\gamma_2$ at $p$ to be the maximal sequence of arcs in the triangluation which are crossed by both $\gamma_1$ and $\gamma_2$ directly before and after passing through $p$.
Let

\[ s=
\left \{
\begin{array}{ccc}
  j+1 &   & \mbox{ if } \tau_{i_j} \neq \tau'_{i'_{j'}},  \\ \\
  j-k &   &  \mbox{ if } \tau_{i_j} = \tau'_{i'_{j'}},
\end{array}
\right.
\]
where $k \geq 1$ is the largest integer such that $\tau_{i_{j-1}}=\tau'_{i'_{j'-1}}, \dots, \tau_{i_{j-k}}=\tau'_{i'_{j'-k}}$
and let 

\[ t=
\left \{
\begin{array}{ccc}
  j &   & \mbox{ if } \tau_{i_{j+1}} \neq \tau'_{i'_{j'+1}},  \\ \\
  j+\ell &   &  \mbox{ if } \tau_{i_{j+1}} = \tau'_{i'_{j'+1}},
\end{array}
\right.
\]
where $\ell \geq 1$ is the largest integer such that $\tau_{i_{j+1}}=\tau'_{i'_{j'+1}}, \ldots, \tau_{i_{j+\ell}}=\tau'_{i'_{j'+\ell}}.$

\begin{defn}
 We call the sequence $(\tau_{i_s}, \tau_{i_{s+1}}, \ldots, \tau_{i_t})=(\tau'_{i'_{s'}}, \tau'_{i'_{s'+1}}, \ldots, \tau'_{i'_{t'}})$ the {\it local overlap} of $\gamma_1$ and $\gamma_2$ at $p.$
\end{defn}

\subsection{Crossing arcs and crossing snake graphs }

\begin{thm}\label{lem 53}
 Let $\gamma_1$ and $\gamma_2$ be two arcs and $\calg_1$ and $\calg_2$ their corresponding snake graphs. Then $\gamma_1$ and $\gamma_2$ cross with a non-empty local overlap $(\tau_{i_s}, \ldots, \tau_{i_t})=(\tau'_{i'_{s'}}, \ldots, \tau'_{i'_{t'}})$ if and only if $\calg_1$ and $\calg_2$ cross in $\calg_1[s,t] \cong \calg_2[s',t'].$
\end{thm}

\begin{proof}
 Choose a sign function $f$ on the overlap of the snake graphs $\calg_1$ and $\calg_2$ and let $f_1$ and $f_2$ be the induced  sign functions on $\calg_1$ and $\calg_2,$ respectively. As usual, let $e_1, \ldots, e_{d-1}$ (respectively $e'_1, \ldots, e'_{d'-1}$) be the interior edges of $\calg_1$ (respectively $\calg_2$). Recall that the edge $e_j$ (respectively $e'_{j'}$) corresponds to the arc $\tau_{[j]}$ (respectively $\tau'_{[j']}$) of the triangle $\Delta_j$  (respectively $\Delta'_{j'}$) in the triangulation $T.$
 
 A zigzag in the snake graph corresponds to a fan in the triangulation, and a straight subgraph of the snake graph corresponds to a zigzag on the triangulation. Therefore two edges $e_j, e_k$ (respectively $e'_{j'}, e'_{k'}$) of the snake graph have the same sign with respect to $f_1$ (respectively $f_2$) if and only if the corresponding arcs $\tau_{[j]}, \tau_{[k]}$ (respectively $\tau'_{[j']}, \tau'_{[k']}$) lie on the same side of $\gamma_1$ (respectively $\gamma_2$), see Figure \ref{figfansnake}.

 \begin{figure}
\begin{center}
\scalebox{1}{ 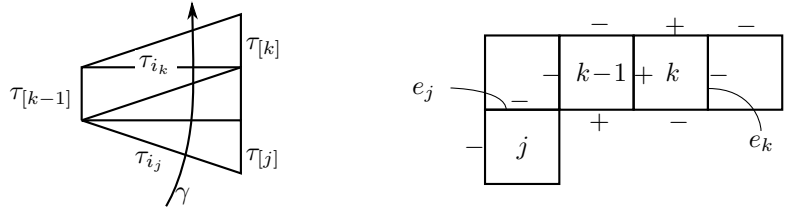}
 \caption{A part of a triangulation on the left and the corresponding snake graph on the right. The edges $e_j$ and $e_k$ have the same sign because the arcs $\tau_{[j]}$ and $\tau_{[k]}$ lie on the same side of $\gamma$.}
 \label{figfansnake}
 \end{center}
\end{figure}

Suppose first that $s>1$ and $t<d.$ Then we have one of the two situations shown in Figure \ref{figlem53}.

 \begin{figure}
\begin{center}
\scalebox{0.8}{ 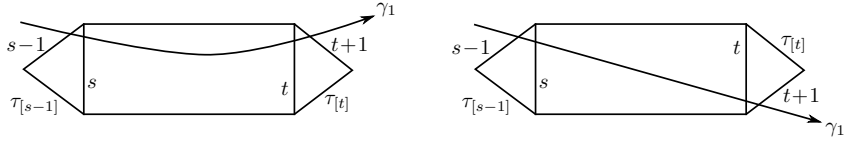}
 \caption{Proof of Theorem \ref{lem 53}; $f_1(e_{s-1})=f_1(e_t)$ on the left, and $f_1(e_{s-1})=-f_1(e_t)$ on the right.}
 \label{figlem53}
 \end{center}
\end{figure}
 
 In both cases $\gamma_2$ crosses $\tau_{i_s}, \ldots, \tau_{i_t}$ but does not cross $\tau_{i_{s-1}}$ nor $\tau_{i_{t+1}}.$ Thus in the first case $\gamma_2$ cannot cross $\gamma_1$ in the overlap. In the second case, $\gamma_2$ must cross $\gamma_1$ in the overlap, since the arcs $\tau_{[s-1]}$ and $\tau_{[t]}$ as well as their endpoints lie on opposite sites of $\gamma_1$ relative to the overlap. Therefore $\gamma_1$ and $\gamma_2$ cross if and only if $f_1(e_{s-1})=-f_1(e_t)$. This implies that the snake graphs $\calg_1$ and $\calg_2$ cross, by Definition  \ref{crossing}.
 
Suppose now that $s=1, \  t <d, \ s' >1$ and $t'=d'.$ Then we have one of the two situations shown in Figure \ref{figlem53b}.

 \begin{figure}
\begin{center}
\scalebox{0.8}{ 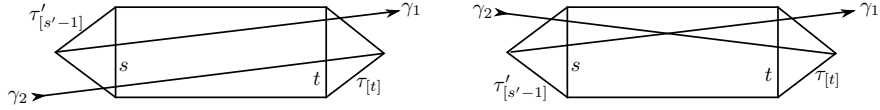}
 \caption{Proof of Theorem \ref{lem 53}}
 \label{figlem53b}
 \end{center}
\end{figure}

In the first case, $\gamma_1$ and $\gamma_2$ do not cross in the overlap and $f_1(e_t)=-f_2(e'_{s'-1})$ because the arcs $\tau_{[t]}$ and $\tau'_{[s'-1]}$ lie on opposite sides of the arcs $\gamma_1$ and $\gamma_2.$ In the second case $\gamma_1$ and $\gamma_2$  cross in the overlap and $f_1(e_t)=f_2(e'_{s'-1}).$ Thus by Definition  \ref{crossing}, $\gamma_1$ and $\gamma_2$ cross if and only if $\calg_1$ and $\calg_2$ cross.

By symmetry, this proves the statement.
\end{proof}

If $\gamma$ is an arc which is not in the triangulation, then the segment of $\gamma$ from its starting point to its first crossing with the triangulation is called the \emph{initial segment} of $\gamma.$

\begin{thm}\label{lem 55}
 Let $\gamma_1$ and $\gamma_2$ be two arcs and $\calg_1$ and $\calg_2$ their associated snake graphs. Suppose that the first tile of $\calg_2$ is given by 
$$\begin{tikzpicture}[scale=.8]
\draw (0,0)--(1,0)--(1,1)--(0,1)--(0,0); 
\node at (.5,.5) {$k$};
\node[scale=.8] at (-.3,.5) {$\delta$};
\node[scale=.8] at (.5,-.3) {$\delta'$};
\end{tikzpicture}
 $$
 Then $\gamma_1$ crosses the initial segment of $\gamma_2$ if and only if one of the following hold.
 
\begin{itemize}
 \item The first or last tile of $\calg_1$ is $G_{\delta},$ or
 \item the first or last tile of $\calg_1$ is $G_{\delta'},$ or
 \item $\calg_1$ contains one of the following  snake graphs $$
\begin{tikzpicture}[scale=.8]
\draw (0,0)--(2,0)--(2,1)--(0,1)--(0,0);
\draw (1,0)--(1,1); 
\node[scale=.9] at (.5,.5) {$\delta$};
\node[scale=.9] at (1.6,.5) {$\delta'$};
\node[scale=.7] at (1.17,.5) {$k$} ;
\end{tikzpicture}
\quad \quad
\begin{tikzpicture}[scale=.8]
\draw (0,0)--(2,0)--(2,1)--(0,1)--(0,0);
\draw (1,0)--(1,1); 
\node[scale=.9] at (.5,.5) {$\delta'$};
\node[scale=.9] at (1.6,.5) {$\delta$};
\node[scale=.7] at (1.17,.5) {$k$} ;
\end{tikzpicture}
$$
\end{itemize}
 
\end{thm}

\begin{proof}
 The configuration of the first tile of $\calg_2$ translates into the following picture.
\[
\begin{tikzpicture}[scale=2]
 \draw (0,0)--(1,-.5)--(1,.5)--(0,0);
 \draw (0,0)--(2.4,.4);
 \path (0,0) edge node [pos=.7, fill=white,outer sep=1mm,scale=.7]{$\gamma_1$} (2.4,.4);
 \node[scale=.8] at (.5,-.4) {$\delta'$};
 \node[scale=.8] at (.5,.4) {$\delta$};
 \node[scale=.6] at (0,0) {$\bullet$};
\node[scale=.6] at (1,-.5) {$\bullet$};
\node[scale=.6] at (1,.5) {$\bullet$};
\node[scale=.8] at (-.2,0) {$x$};
\node[scale=.8] at (1.2,-.5) {$y$};
\node[scale=.8] at (1.2,.5) {$z$};
\end{tikzpicture}
\]
Thus $\gamma_1$ crosses the initial segment of $\gamma_2$ if and only if $\gamma_1$ crosses $\delta$ and ends at $y$ or $\gamma_1$ crosses $\delta'$ and ends at $z$ or $\gamma_1$ crosses $\delta$ and $\delta'.$ This translates to the three cases in the statement. 
\end{proof}

\subsection{Smoothing arcs and resolving snake graphs}

\begin{thm} \label{smoothing1} Let $\gamma_1$ and $\gamma_2$ be two arcs which cross with a non-empty local overlap, and let $\calg_1$ and $\calg_2$ be the corresponding snake graphs with overlap $\calg$. Then the  snake graphs of the four arcs obtained by smoothing the crossing of $\gamma_1$ and $\gamma_2$ in the overlap are given by the resolution $\re 12$ of the crossing of the snake graphs $\calg_1$ and $\calg_2$ at the overlap $\calg.$
 
\end{thm}

\begin{remark}
 We do not assume that $\gamma_1$ and $\gamma_2$ cross only once. If the arcs cross multiple times the theorem can be used to resolve any of the crossings.
\end{remark}

\begin{remark} In \cite{MSW2}, the authors considered also smoothing of generalized arcs (allowing self-crossings) and essential loops. The description of the resolutions of the corresponding snake graphs will be given in a forthcoming paper.

\end{remark}

\begin{proof}
 As usual, let $\ta st = \taa {s'}{t'}$ denote the local overlap of $\gamma_1$ and $\gamma_2$ at the crossing under consideration. Then the four arcs obtained by smoothing the crossing are represented by  
\begin{align*}
 \gamma_3=\gamma_{1,1} \cdot \gamma_{2,2} \hspace{.8in} \gamma_5=\gamma_{1,1} \cdot  \overline \gamma_{2,1} 
\\
 \gamma_4=\gamma_{2,1} \cdot \gamma_{1,2} \hspace{.8in} \gamma_6=\overline \gamma_{2,2} \cdot \gamma_{1,2}
\end{align*}  
 \begin{figure}
\begin{center}
\scalebox{0.8}{ 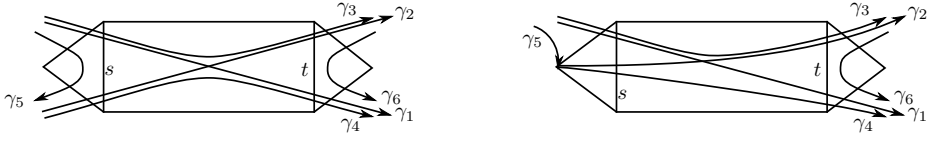}
 \caption{Proof of Theorem \ref{smoothing1}}
 \label{smoothing}
 \end{center}
\end{figure}
(see Figure \ref{smoothing}) 
where $\alpha \cdot \beta$ denotes the concatenation of the paths $\alpha$ and $\beta$ and $\overline \alpha$ denotes the path $\alpha$ with the opposite orientation, $\gamma_{i,1}$ denotes the segment of the arc $\gamma_i$ from its starting point to the crossing point $p$ and $\gamma_{i,2}$ the segment from $p$ to the terminal point, for $i=1,2.$ Recall that arcs are defined up to homotopy, so for example the arc $\gamma_5$ does not cross $\tau_{i_s}.$

Suppose first that $\tp.$ From the construction it is obvious that the snake graphs $\calg_i$ of the arcs $\gamma_i,$ for $i=3,4,5,6,$ are given exactly by the four snake graphs of $\re 12$ in Definition \ref{resolution}.

Now suppose that we do not have $\tp.$ Observe that that we cannot have $s=1$ and $s'=1$ simultaneously because $\gamma_1$ and $\gamma_2$ cross in the overlap, and, similarly, we cannot have $t=d$ and $t'=d'$ simultaneously. Therefore we have to consider the cases where one of the two arcs starts or ends right before or right after the overlap. Since these cases are all symmetric, we may assume without loss of generality that $s'=1$ and $s>1, t<d$ and $t'<d'.$ This case is illustrated on the right side of Figure \ref{smoothing}. In this case, the snake graphs $\calg_3, \ \calg_4$ and $\calg_6$ are the same as in the first case, but $\calg_5$ now is of the form $\calg_5=\calg_1[1,k],$ where $1 \leq k <s-1$ and $k$ is the largest integer such that $\tau_{i_k}$ is not in the maximal $(T,\gamma_1)$-fan ending at $\tau_{i_{s-1}}.$

Indeed, this follows because every arc in this fan crosses $\gamma_1$ but not $\gamma_5,$ since each arc in the fan ends at the endpoint of $\gamma_5.$ Since fans in the triangulation correspond to zigzag subgraphs of the snake graph, we see from the definition of sign functions that any sign function $f_1$ on $\calg_1$ satisfies $f_1(e_{s-1})=f_1(e_\ell)$ for each $k <\ell \leq s-1$ and $f_1(e_{s-1})=-f_1(e_k).$ Therefore $\calg_5$ is exactly as in the definition of $\re 12.$
\end{proof}

So far, we have considered two arcs which cross with a non-empty local overlap. Now we study two arcs which cross with an empty local overlap.

\begin{thm} \label{smoothing2} Let $\gamma_1$ and $\gamma_2$ be two arcs which cross in a triangle $\Delta$ with an empty local overlap, and let $\calg_1$ and $\calg_2$ be the corresponding snake graphs. Assume $\Delta=\Delta'_0$ is the first triangle $\gamma_2$ meets. Then the snake graphs of the four arcs obtained by smoothing the crossing of $\gamma_1$ and $\gamma_2$ in $\Delta$ are given by the resolution $\gt s{\delta_3}12$ of the grafting of $\calg_2$ on $\calg_1$ in $G_s,$ where $0 \leq s \leq d$ is such that $\Delta=\Delta_s$ and if $s=0$ or $s=d$ then $\delta_3$ is the unique side of $\Delta$ that is not crossed by neither $\gamma_1$ nor $\gamma_2.$
\end{thm}

\begin{proof}
 As before, let $\gamma_{i,1}$ denote the segment of $\gamma_i$ from the starting point until the crossing point, and let $\gamma_{i,2}$ denote the segment   from the crossing point to the terminal point. Then the four arcs obtained by smoothing are represented by  
\begin{align*}
 \gamma_3=\gamma_{1,1} \cdot \gamma_{2,2} \hspace{.8in} \gamma_5=\gamma_{1,1} \cdot  \overline \gamma_{2,1} 
\\
 \gamma_4=\gamma_{2,1} \cdot \gamma_{1,2} \hspace{.8in} \gamma_6=\overline \gamma_{2,2} \cdot \gamma_{1,2}
\end{align*}
where $\alpha \cdot \beta$ denotes the concatenation of the paths $\alpha$ and $\beta$ and $\overline \alpha$ is the path $\alpha$ with the opposite orientation.

Suppose first that $\Delta=\Delta_s$ with $0 <s<d.$ Thus $\gamma_1$ crosses two sides $\tau_{i_s}$ and $\tau_{i_{s+1}}$ of $\Delta$ and $\gamma_2$ crosses the third side $\tau_{i'_1}$ of $\Delta,$ see the left picture in Figure \ref{smoothing3}.

 \begin{figure}
\begin{center}
\scalebox{1}{ 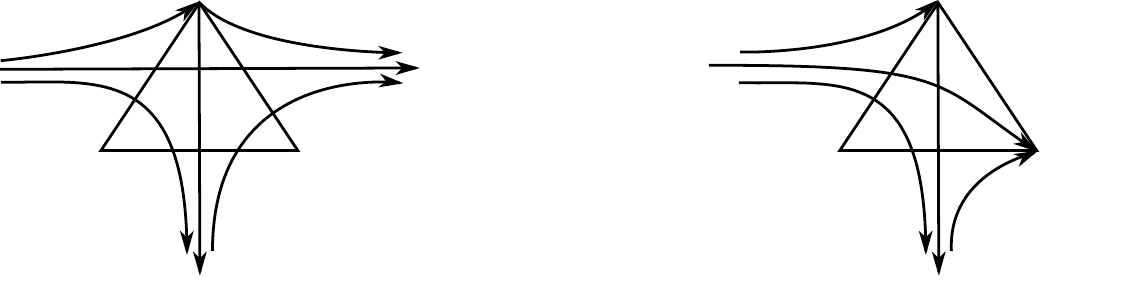}
 \caption{Proof of Theorem \ref{smoothing2}}
 \label{smoothing3}
 \end{center}
\end{figure}
 
Let $\calf$ denote the maximal $(T,\gamma_1)$-fan containing the arcs $\tau_{i_s}$ and $\tau_{i_{s+1}}$ and determined by the crossing. Then the snake graphs $\calg_i$ of the arcs $\gamma_i,$ for $i=3,4,5,6$ are given as follows.  
\begin{flalign*}
\calg_3&=\calg_1[1,s] \cup \calg_2  \hspace{.4in}  \parbox[t]{.55\textwidth}{ glued along the edge of $G_s$ and $G'_1$ labeled $\tau_{i_{s+1}}$}; \\
\calg_4&=
\begin{cases}
\calg_1[k_4,d] &  \parbox[t]{.55\textwidth}{ where $k_4 >s+1$  is the least integer such that $\tau_{i_{k_4}} \not \in \calf,$ if such a $k_4$ exists;}\\
\{\gamma_4\}&  \parbox[t]{.55\textwidth}{ otherwise;}
\end{cases} \\
\calg_5&=
\begin{cases}
\calg_1[1,k_5] &  \parbox[t]{.55\textwidth}{ where $1 \leq k_5 <s$  is the largest integer such that $\tau_{i_{k_5}} \not \in \calf,$ if such a $k_5$ exists;}\\
\{\gamma_5\} &  \parbox[t]{.55\textwidth}{ otherwise;}
\end{cases} \\
\calg_6&=\ocalg_2 \cup \calg_1[s+1,d]  \hspace{.4in}  \parbox[t]{.55\textwidth}{ glued along the edge of $G_{s+1}$ and $G'_1$ labeled $\tau_{i_{s}}$};
\end{flalign*}
Note that if the $k_4$ in the description of $\calg_4$ does not exist, then $\gamma_4$ itself must be in the fan $\mathcal{F}$, hence $\gamma_4\in T$ and $\calg_4=\{\gamma_4\}$. Similarly,  if the $k_5$  in the description of $\calg_5$ does not exist, then $\calg_5=\{\gamma_5\}$. 
Thus these are exactly the snake graphs in $\gt s{\delta_3}12.$

Now suppose that $\Delta=\Delta_d$ is the last triangle $\gamma_1$ meets, see the right side of Figure \ref{smoothing3}. Then  
\begin{align*}
 \gamma_3=\gamma_{1,1} \cdot \gamma_{2,2} \hspace{.8in} \gamma_5=\gamma_{1,1} \cdot  \overline \gamma_{2,1} 
\\
 \gamma_4\in T \hspace{.8in} \gamma_6=\overline \gamma_{2,2} \cdot \gamma_{1,2}
\end{align*}
and $\gamma_4$ is the unique side of the triangle $\Delta$ that is not crossed by $\gamma_1$ or $\gamma_2.$ Let $\calf_1$ be the maximal $(T,\gamma_1)$-fan  containing $\tau_{i_d}$ and $\gamma_4,$ and let $\calf_2$ be the maximal $(T,\gamma_2)$-fan  containing $\tau'_{i'_1}$ and $\gamma_4.$ Then the snake graphs of the arcs $\gamma_i$ for $i=$ 3, 4, 5,  6 are the following.

\begin{flalign*}
\calg_3&=\calg_1 \cup \calg_2  \hspace{.4in}  \parbox[t]{.55\textwidth}{ glued along the edge $G_d$ and $G'_1$ labeled $\gamma_{4}$}; \\
\calg_4&= \{\gamma_4\} \\
\calg_5&=
\begin{cases}
\calg_1[1,k_5] &  \parbox[t]{.55\textwidth}{ where $1 \leq k_5 <d$  is the largest integer such that $\tau_{i_{k_5}} \not \in \calf_1$ if such a $k_5$ exists;}\\
\{\gamma_5\} &  \parbox[t]{.55\textwidth}{ otherwise;}
\end{cases}\\
\calg_6&=
\begin{cases}
\ocalg_1[1,k_6] &  \parbox[t]{.55\textwidth}{ where $1 \leq k_5 <d$ and $k_6$ is the largest integer such that $\tau'_{i'_{k_6}} \not \in \calf_2$ if such a $k_6$ exists;}\\
\{\gamma_6\} &  \parbox[t]{.55\textwidth}{ otherwise;}
\end{cases}
\end{flalign*}

 Since fans in the triangulation correspond to zigzag subgraphs in the snake graph, we see from the definition of sign functions that these snake graphs are exactly the ones in the definition of $\gt s{\delta_3}12$ with $\delta_3=\gamma_4.$
 
 The case $\Delta=\Delta_0$ is similar to the case $\Delta=\Delta_d$ and therefore left to the reader.
\end{proof}


\section{Products of cluster variables}\label{sect products}\label{sect 6}

Let $\cala= \cala(S,M,T)$ be the cluster algebra associated to the surface $(S,M)$ with principal coefficients in the initial seed $\sum_T = (\bx_T, \by_T, Q_T)$ at the triangulation $T= \{ \tau_1, \tau_2, \ldots, \tau_n \}$ where
\begin{align*}
 \bx =(x_i | i=1, \ldots, n, \tau_i \in T) \\
  \by =(y_i | i=1, \ldots, n, \tau_i \in T)
\end{align*}


\subsection{Non-empty overlaps}

If $\calg$ is a snake graph associated to an arc $\gamma$ in a triangulated surface $(S,M,T)$ then each tile of $\calg$ corresponds to a quadrilateral in the triangulation $T,$ and we denote by $\tau_{i(G)}\in T$ the diagonal of that quadrilateral. With this notation we define
\begin{align*}
 x(\calg) = \prod_{ G \mbox{ tile in } \calg} x_{i(G)}\\
 y(\calg) = \prod_{ G \mbox{ tile in } \calg} y_{i(G)}
\end{align*}
If $\calg=\{\tau\}$ consists of a single edge, we let $x(\calg)=1$ and $y(\calg)=1$.

Let $\gamma_1$ and $\gamma_2$ be two arcs which cross with a non-empty local overlap. Let $x_{\gamma_1}$ and $x_{\gamma_2}$ be the corresponding cluster variables and $\calg_1$ and $\calg_2$ the snake graphs with corresponding overlap $\calg.$ 
Recall that $\re 12$ consists of two pairs $(\calg_3,\calg_4)$ and $(\calg_5,\calg_6)$ of snake graphs.

Define the {\em closure} $\tcalg$ {\em  of the overlap} $\calg$  to be the union of all tiles in $\calg_1 \cup \calg_2$ which are not in $\calg_5 \cup \calg_6.$ Let $\re 12$ be the resolution of the crossing of $\calg_1$ and $\calg_2$ at $\calg$, and 
define the {\em Laurent polynomial of the resolution} by

\begin{equation*}
 \call (\re 12) = \call (\calg_3 \sqcup \calg_4) + y( \tcalg) \call (\calg_5 \sqcup \calg_6)
\end{equation*}

where 

\begin{equation*}
 \call (\calg_k \sqcup \calg_\ell) = \frac{1}{x(\calg_k)x(\calg_\ell)} \sum_{P \in \ma k\ell} x(P) y(P)
\end{equation*}

\begin{thm} \label{laurent}
 Let $\gamma_1$ and $\gamma_2$ be two arcs which cross with a non-empty local overlap and let $\calg_1$ and $\calg_2$ be the corresponding snake graphs with local overlap $\calg.$ Then
 
\begin{equation*}
 \call(\calg_1 \sqcup \calg_2) = \call (\re 12)
\end{equation*}
\end{thm}

\begin{proof}
 First we note that
\begin{equation} \label{eq1}
  x(\calg_1 \sqcup \calg_2) = x(\calg_3 \sqcup \calg_4) = x(\calg_5 \sqcup \calg_6) x(\tcalg) x(G),
\end{equation}
where the first identity holds because $\calg_1 \sqcup \calg_2$ and $\calg_3 \sqcup \calg_4$ have the same set of tiles, and the second identity holds because $\tcalg$ consists of the tiles of $(\calg_1 \cup \calg_2) \backslash (\calg_5 \cup \calg_6)$ and the overlap $\calg$ consists of the tiles that appear in both $\calg_1$ and $\calg_2.$

Using the equation (\ref{eq1}) on the definition of $\la 12$ as well as the bijection $\varphi=(\varphi_{34},\varphi_{56}):\ma 12 \to \match\re 12$ of Theorem \ref{bijections}, we obtain
\begin{align*}
 \la 12 = &\displaystyle \frac{1}{x(\calg_3 \sqcup \calg_4)} \sum_{\varphi_{34}(P) \in \ma 34} x(P)y(P)\\&+ \displaystyle \frac{1}{x(\calg)x(\tcalg)x(\calg_5 \sqcup \calg_6)} \sum_{\varphi_{56}(P) \in \ma 56} x(P)y(P).
\end{align*}

On the other hand,

\begin{equation*}
 \call(\re 12) = \la 34 + y(\tcalg) \la 56
\end{equation*}

and therefore it suffices to show the following lemma.
\end{proof}

\begin{lem} \label{lem 72} Let $P\in \Match(\calg_1\sqcup\calg_2)$.
\begin{itemize}
 \item[(a)] $x(\varphi_{34}(P))\ y(\varphi_{34}(P))=x(P) y(P)$ if $\varphi(P) \in \ma 34$.
 \item[(b)] $x(\varphi_{56}(P))\ y(\varphi_{56}(P))=x(P)y(P) / x(\tcalg) x(\calg) y(\tcalg) $ if $\varphi(P) \in \ma 56$.
\end{itemize}
\end{lem}

\begin{proof}
 Let $P \in \ma 12$ and let $P_i$ be its restriction to $\calg_i,$  $ i=1,2.$
 
(a) Since the mapping $P \to \varphi_{34}(P)$ preserves the weight of the edges on the matching, we have $x(\varphi_{34}(P))=x(P).$ The identity $y(\varphi_{34}(P))=y(P)$ follows from an inspection of the eight cases in Figure~\ref{mu} and the four cases of Figure~\ref{figrho}.

(b)  By definition
\begin{equation} \label{eq2}
 \varphi_{56}(P) = \restr{P}{\calg_5 \sqcup \calg_6} \backslash \{  \mbox{ glueing edges }   \}
\end{equation}
where the glueing edge for $\calg_5$ has weight $x_{i_s}$ if $s>1$ and $s'>1,$ and in all other cases $\calg_5$ has no glueing edge, and the glueing edge for $\calg_6$ has weight $x_{i_t}$ if $t<d$ and $t' < d',$ and in all other cases $\calg_6$ has no glueing edge. Since the glueing edges are in the interior of $\calg_5$ and $\calg_6,$ removing them from the matching does not change the $y$-monomials. Thus
\begin{equation*}
 y(\varphi_{56}(P)) = y(P ) \big / y(\tcalg)
\end{equation*}

It remains to study the $x$-monomials. First note that if $P_1$ or $P_2$ contains an interior edge of $\tcalg$ then $\varphi(P)$ would be in $\ma 34.$ Thus in our situation, we must have that on $\tcalg$ both matchings $P_1$ and $P_2$ consist of boundary edges of $\tcalg.$ Moreover, since $\varphi(P) \in \ma 56,$ it follows that $P_1$ and $P_2$ must be as in Figure \ref{mu2} and Figure~\ref{figrho2}
for all tiles in the overlap. In particular,  $P_1$ and $P_2$ do not share an edge on $\tcalg.$
 
Therefore equation (\ref{eq2}) implies that if $s>1, s'>1, t<d$ and $t' <d'$ then
\begin{align} \label{eq3}
 x(\varphi_{56}(P)) = \frac{x(\restr{P}{\calg_5 \sqcup \calg_6})}{x_{i_s}x_{i_t}}= \frac{x(P)}{x_{i_s}x_{i_t}}\frac{x_{i_{s-1}}x_{i'_{s'-1}}x_{i_{t+1}} x_{i'_{t'+1}}}{\displaystyle \prod_{\tau \in {\partial \calg}} x_{\tau}}
\end{align}
where the product runs over all boundary edges of $\calg.$ Indeed, in this case $\widetilde \calg =\calg$ and  the four variables $x_{i_{s-1}}, x_{i'_{s'-1}}, x_{i_{t+1}}$ and $x_{i'_{t'+1}}$ appear in the numerator, because these are the weights of the four edges which connect $\calg$ with $\calg_1 \backslash \calg$ and with $\calg_2 \backslash \calg.$ These edges are boundary edges in $\calg$ but interior edges in $\calg_1,$ respectively $\calg_2.$

To finish the proof in the case where $\tp$ it suffices to show that 
\begin{align}\label{eq4} 
 x(\calg) x(\tcalg) = \frac{x_{i_s}x_{i_t}}{x_{i_{s-1}} x_{i'_{s'-1}} x_{i_{t+1}} x_{i'_{t'+t1}}} \prodd x_{\tau}
\end{align}

We proceed by induction on $\ell=t-s+1$ the number of tiles in $\calg.$ If $\ell=1,$ then $\calg$ consists of a single tile with boundary weights $x_{i_{s-1}},\ x_{i'_{s'-1}},\ x_{i_{t+1}},$ $x_{i'_{t'+1}}.$ Moreover, $s=t$, and $\calg= \tcalg$, since $\tp$. This proves equation (\ref{eq4}) in this case.

Suppose now that $\ell >1.$ Then $x(\widetilde \calg)=x(\calg) = x_{i_s} x_{i_{s+1}} \ldots x_{i_t}.$ Moreover, except for the four extremal edges with weights $x_{i_{s-1}},\ x_{i'_{s'-1}},\ x_{i_{t+1}},\ x_{i'_{t'+1}},$ every boundary edge of $\calg$ has weight $x_{i_s}, x_{i_{s+1}}, \ldots, x_{i_{t-1}}$ or $x_{i_t},$ each of the weights $x_{i_{s+1}}, \ldots, x_{i_{t-1}}$ appears exactly twice and the weights $x_{i_s}$ and $x_{i_t}$ appears exactly once in the boundary of $\calg$. Thus
\begin{align*}
 \prodd x_{\tau} = {x_{i_{s+1}}^2} \cdots {x_{i_{t-1}}^2} x_{i_s} x_{i_t}x_{i_{s-1}}x_{i'_{s'-1}}x_{i_{t+1}}x_{i'_{t'+1}}
\end{align*}
 which shows (\ref{eq4}) in this case.

Now suppose $s'=1.$ Then $s>1,$ because $\gamma$ and $\gamma'$ cross in the overlap. Suppose first that $t<d$ and $t'<d'.$ Then it follows from the definition of $\varphi$ that the matching $P=(P_1,P_2)$ must be of the form shown in Figure \ref{mu2}, but with the tile $G'_{s'-1}$ removed. Only the third case in Figure \ref{mu2} defines a matching after removing $G'_{s'-1},$ thus $P=(P_1,P_2)$ is of the form 
shown in Figure \ref{figlem72}.

 \begin{figure}
\begin{center}
\scalebox{0.8}{ 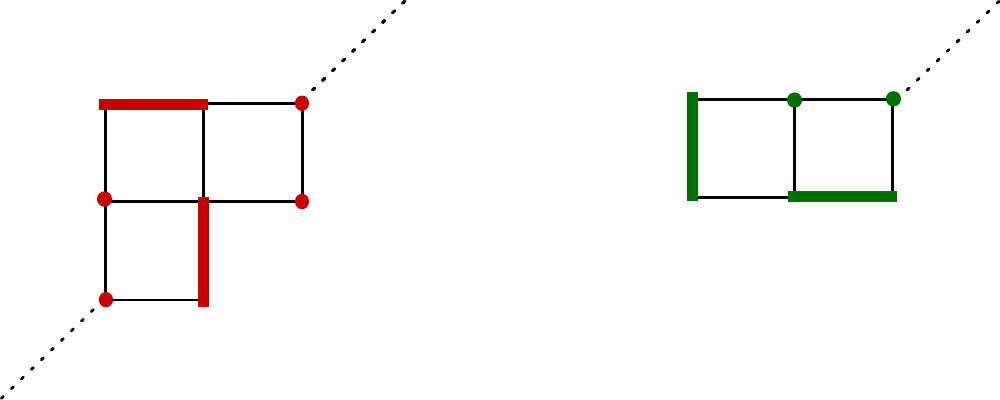}
 \caption{Proof of Lemma \ref{lem 72}}
 \label{figlem72}
 \end{center}
\end{figure}

Let $k < s-1$ be the largest integer such that the sign function of $\calg_1$ agrees on the edges $e_k$ and $e_{s-1}$ (if such a $k$ exists) as in the definition of $\calg_5.$ Then we have a situation similar to the  example shown in Figure \ref{figlem72b}, where $k=s-6.$ The red edges are in $P_1,$ and the green edges are in $P_2.$
 \begin{figure}
\begin{center}
\scalebox{0.8}{ 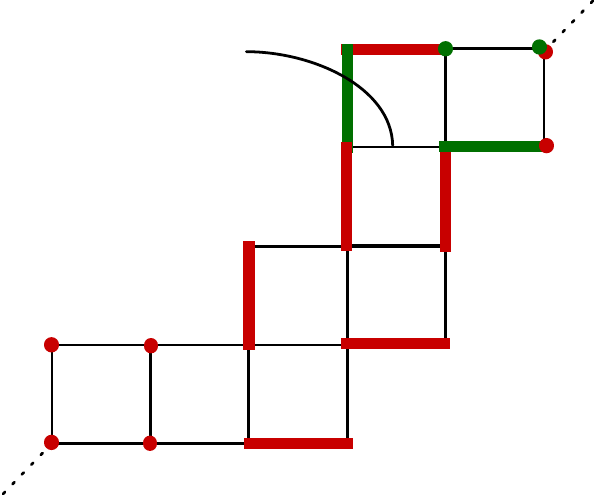}
 \caption{Proof of Lemma \ref{lem 72}}
 \label{figlem72b}
 \end{center}
\end{figure}
The tile $G_k$ is matched by $\varphi_{56}(P).$ Therefore
\begin{equation} \label{eq5}
 x(\varphi_{56}(P)) = \frac{{x(P)}\,x_{[s-1]} x_{i_{s-1}} x_{i_{t+1}} x_{i'_{t'+1}} }{{x_{i_t}} \left(   \displaystyle \prod_{j=k+1}^s x_{i_j}   \right)  \left(  \prodd x_{\tau}  \right) }.
\end{equation}
Note  that $x_{i_{s-1}}$ appears in both products in the denominator and once in the numerator. To finish the proof in this case, it suffices to show that
\begin{equation} \label{eq6}
 x(\calg) x(\tcalg) = \frac{x_{i_t}}{x_{[s-1]} x_{i_{s-1}} x_{i_{t+1}} x_{i'_{t'+1}}} \left(   \displaystyle \prod_{j=k+1}^s x_{i_j}   \right)  \left(  \prodd x_{\tau}  \right).
\end{equation}
Again we proceed by induction on $\ell=t-s+1.$ If $\ell=1$ then $\calg$ consists of a single tile and the statement follows simply by computing
\[\begin{array} {rclrcl}
 \prodd x_{\tau} & =& x_{[s-1]} x_{i_{s-1}} x_{i_{t+1}} x_{i'_{t'+1}} &\\
 x(\calg) & =& x_{i_s} =x_{i_t} & x(\tcalg) &=& \displaystyle \prod_{j=k+1}^s x_{i_j}.
\end{array}\]

Suppose now that $\ell >1.$ Then
\begin{align*}
 x(\calg)x(\tcalg)= \left(\pprod {j=s}{t}{x_{i_j}} \right)\left(\pprod {j=k+1}{t}{x_{i_j}}\right) = \left(\pprod {j=k+1}{s}{x_{i_j}}\right) x_{i_s} \left(\pprod {j=s+1}{t}{x_{i_j}^2}\right)
\end{align*}

On the other hand, except for the four extremal edges with weights $x_{[s-1]},$ $x_{i_{s-1}},\ x_{i_{t+1}},$ and $ x_{i'_{t'+1}},$ every boundary edge of $\calg$ has weight $x_{i_s}, x_{i_{s+1}}, \ldots,$ $x_{i_{t-1}}$ or $x_{i_t},$ the weights $x_{i_{s+1}}, \ldots, x_{i_{t-1}}$ appear exactly twice and $x_{i_s}$ and $x_{i_t}$ exactly once. Thus
\begin{align*}
 \prodd x_{\tau} = \left(  \pprod {j=s+1}{t-1}{x_{i_j}^2}  \right)  x_{i_s} x_{i_t} x_{[s-1]} x_{i_{s-1}} x_{i_{t+1}} x_{i'_{t'+1}}
\end{align*}
and this shows equation (\ref{eq6}).
The cases where $s=1, t=d$ or $t'=d'$ are similar.
\end{proof}

\subsection{Empty overlaps}

Now let $\gamma_1$ and $\gamma_2$ be two arcs which cross in a triangle $\Delta$ with an empty overlap. We may assume without loss of generality that $\Delta$ is the first triangle $\gamma_2$ meets. Let $x_{\gamma_1}$ and $x_{\gamma_2}$ be the corresponding cluster variables and $\calg_1$ and $\calg_2$ be their associated snake graphs, respectively. We know from Theorem \ref{smoothing2} that the snake graphs of the arcs obtained by smoothing the crossing of $\gamma_1$ and $\gamma_2$ are given by the resolution $\gt s{\e_3}12$ of the grafting of $\calg_2$ on $\calg_1$ in $G_s,$ where $s$ is such that $\Delta=\Delta_s$ and, if $s=0,$ then $\e_3$ is the unique side of $\Delta$ which is not crossed neither by $\gamma_1$ nor $\gamma_2.$

The edge of $G_s$ which is the glueing edge for the grafting is called the {\em grafting edge}. We say that the grafting edge is {\em minimal} if it belongs to the minimal matching on $\calg_1.$

Recall that $\gt s {\e_3}12$ is a pair $(\calg_3 \sqcup \calg_4), (\calg_5 \sqcup \calg_6).$ Let
\begin{align*}
 \call (\gt s{\e_3}12) = y_{34} \call (\calg_3 \sqcup \calg_4) + y_{56} \call (\calg_5 \sqcup \calg_6)
\end{align*}
where

\begin{align} \label{skeincoeff}
\left\{\begin{array}{llll}
 y_{34}=1& \textup{and} &y_{56} = \pprod {j=k_5+1}{s-1}{y_{i_j}} & \mbox{ if the grafting edge is minimal; } \\
  y_{34} = \pprod {j=s+1}{k_4-1}{y_{i_j}} &\textup{and} &y_{56}=1 & \mbox{ otherwise. } 
\end{array}\right.
\end{align}

\begin{thm} \label{laurent2}
 With the notation above, we have
\begin{align*}
 \la 12 = \call (\gt s{\e_3}12)
\end{align*}
\end{thm}

\begin{proof} The cases $s=d$ and $s=0$ are symmetric and therefore we consider only the case $0 < s \leq d.$ By definition,

\begin{align*}
 \la 12 = \frac{1}{x(\calg_1 \sqcup \calg_2)} \displaystyle \sum_{P \in \ma 12} x(P) y(P)
\end{align*}
 and from the construction of $\gt s{\e_3}12$ in Definition \ref{grafting}, we see
\begin{align} \label{graft1}
 x(\calg_1 \sqcup \calg_2) =
\begin{cases}
  x(\calg_3 \sqcup \calg_4) \pprod {j=s+1}{k_4-1}{x_{i_j}} & \mbox { if } 0 < s< d;\\
  x(\calg_3 \sqcup \calg_4)  & \mbox { if }  s=d;
\end{cases}
\end{align}
{ where $k_4 > s+1$ is the least integer such that $f_1(e_s)=-f_1(e_{k_4-1})$ if such a $k_4$ exists; and $k_4=d+1$ otherwise,}
 and
\begin{align} \label{graft2}
 x(\calg_1 \sqcup \calg_2) =
\begin{cases}
  x(\calg_5 \sqcup \calg_6) \pprod {j=k_5+1}{s}{x_{i_j}} & \mbox { if } 0 < s< d;\\
  x(\calg_5 \sqcup \calg_6) \left(\pprod {j=k_5+1}{s}{x_{i_j}}\right)\left(  \pprod {j=1}{k_6-1}{x_{i'_j}}\right) & \mbox { if }  s=d;
\end{cases}
\end{align}{ where $k_5 < s$ is the largest integer such that $f_1(e_{k_5})=-f_1(e_{s})$ if such a $k_5$ exists; and $k_4=0$ otherwise.}
 
{\it Case 1.} Suppose first that $0 < s < d.$ Using the bijection $\varphi = (\varphi_{34} , \varphi_{56}): \ma 12 \to \match (\gt s{e_3}12) $ of Theorem \ref{bijections}, we obtain  
\begin{align*}
 &\la 12 & = \displaystyle \frac{1}{x(\calg_3 \sqcup \calg_4) \left(   \pprod {j=s+1}{k_4-1}{x_{i_j}} \right) } \displaystyle \sum_{\varphi_{34}(P) \in \ma 34 } x(P) y(P)\\
 & &+ \displaystyle \frac{1}{x(\calg_5 \sqcup \calg_6) \left(   \pprod {j=k_5+1}{s}{x_{i_j}} \right) } \displaystyle \sum_{\varphi_{56}(P) \in \ma 56 } x(P) y(P)
\end{align*}
 
Therefore, in case 1,  the statement follows from the following lemma.
\end{proof}
\begin{lem}\label{lem 75}
\begin{itemize}
 \item[(a)] If $\varphi(P) \in \ma 34$ then 
\begin{align*}
 x(P)y(P) =x(\varphi_{34}(P)) y(\varphi_{34}(P)) \left ( \pprod {j=s+1}{k_4-1}{x_{i_j}} \right ) y_{34} 
\end{align*}
\item [(b)] If $\varphi (P) \in \ma 56$ then
\begin{align*}
   x(P)y(P)=x(\varphi_{56}(P)) y(\varphi_{56}(P)) \left ( \pprod {j=k_5+1}{s}{x_{i_j}} \right )y_{56} 
\end{align*}
\end{itemize}
\end{lem}
\begin{proof}
(a) In the operation $\sigma_{s,3}$ in the definition of $\varphi,$ we lose an edge of the matching (the grafting edge) with weight $x_{i_{s+1}}$, see Figure~\ref{sigma}. The map $\varphi_{34}$ is given by the first three cases of $\sigma$ in Figure~\ref{sigma}. In each of these cases, the north east vertex of the tile $G_s$ in $\calg_1$ is matched in $P$ by an edge of $G_s.$ The subgraph $\calg_1[s, k_4-1]$ is a zigzag subgraph of $\calg_1,$ thus, on $\calg_1[s+1,k_4-1]$, the matching $P$ is uniquely determined by the fact that the north east vertex of $G_s$ is matched in $G_s.$ Indeed, the matching is as in the left picture of Figure \ref{figlem75}.
 \begin{figure}
\begin{center}
\scalebox{0.8}{ 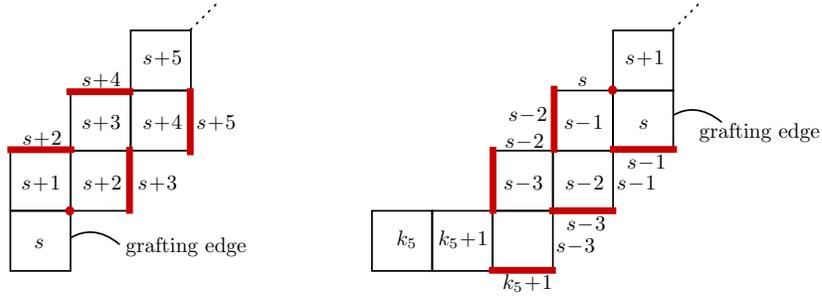}
 \caption{Proof of Lemma \ref{lem 75}: part (a) on the left; part (b) on the right.}
 \label{figlem75}
 \end{center}
\end{figure}
Furthermore, the weight of $P$ on $\calg_1(s+1,k_4-1)$ is equal to $\pprod {j=s+2}{k_4-1}{x_{i_j}}.$ Therefore
\begin{align*}
 x(P) \big / x(\varphi_{34}(P)) = \left (  \pprod {j=s+2}{k_4-1}{x_{i_j}} \right ) x_{i_{s+1}} 
\end{align*}
where the $x_{i_{s+1}}$ comes from the operation $\sigma_{s,3}.$
Moreover, on $\calg_1(s+1,k_4-1),$ the matching $P$ is the minimal matching, if the grafting edge is minimal, and $P$ is the maximal matching, otherwise. Thus
\begin{align*}
 y(P) \big / y(\varphi_{34}(P)) = \left\{
\begin{array}{cl}
  1 & \mbox{ if the grafting edge is minimal, }\\ \\
\pprod {j=s+1}{k_4-1}{y_{i_j}} & \mbox{ otherwise.}
\end{array}\right.
\end{align*}
which is exactly the definition of $y_{34}.$ This completes the proof of part (a).

(b) In the operation $\sigma_{s,5}$ in the definition of $\varphi,$ which is given by the last three cases in Figure~\ref{sigma}, we lose an edge (the grafting edge) with weight $x_{i_s}.$ In each of the three cases in the definition of $\sigma_{s,5},$ the northwest vertex of $G_s$ in $\calg_1$ is matched in $P$ by an edge of $G_{s+1}.$ Again, since the subgraph $\calg_1[k_5+1,s+1]$ is a zigzag subgraph, then on $\calg(k_5+1,s)$ the matching $P$ is uniquely determined. Indeed, the matching is as in the right picture of Figure \ref{figlem75}.

Furthermore, the weight of $P$ on $\calg(k_5+1,s)$ is equal to $\left ( \pprod {j=k_5+1}{s-1}{x_{i_j}} \right ) x_{i_s} .$ Moreover, on $\calg(k_5+1,s),$ the matching $P$ is the maximal matching if the grafting edge is minimal, and the minimal matching otherwise. Thus
\begin{align*}
y(P) \big / y(\varphi_{56}(P)) = 
\begin{cases}
 \pprod {j=k_5+1}{s-1}{y_{i_j}} & \mbox{ if the grafting edge is minimal}\\
 \quad1 & \mbox{ otherwise}
\end{cases}
\end{align*}
which is exactly the definition of $y_{56}.$ Moreover, comparing this situation to the two cases for $y(P) \big / y(\varphi_{34} (P))$ above, we see that 
\[\begin{array}{lcl}
 y_{34}=1 &  \Longleftrightarrow & y_{56} = \pprod {j=k_5+1}{s-1}{y_{i_j}} \\
 y_{34} = \pprod {j=s+1}{k_4-1}{y_{i_j}} &  \Longleftrightarrow & y_{56}=1
\end{array}\]
which completes the proof of part (b).

{\it Case 2.} Now suppose that $s=d.$ Using the bijection $\varphi=(\varphi_{34}, \varphi_{56}): \ma 12 \to \match (\gt s{e_3}12)$ of Theorem \ref{bijections} and equations (\ref{graft1}) and (\ref{graft2}), we obtain 
\begin{align*}
 \la 12 &= \frac{1}{x(\calg_3 \sqcup \calg_4)} \displaystyle \sum_{\varphi_{34}(P) \in \ma 34} x(P) y(P)\\
&+ \frac{1}{x(\calg_5 \sqcup \calg_6) \left(  \pprod {j=k_5+1}{s}{x_{i_j}}  \right) \left(  \pprod {j=1}{k_6-1}{x_{i'_j}}  \right)} \displaystyle \sum_{\varphi_{56}(P) \in \ma 56} x(P) y(P)
\end{align*}
Note that $\calg_4 = \{\delta_3\} $ and $x(\calg_4)=1$ in this case. Therefore the statement follows from the following lemma. 
\end{proof}
 
\begin{lem}\label{lem 76}
\begin{itemize}
\item[(a)] If $\varphi(P) \in \ma 34$ then
\begin{align*} x(P)y(P) =  
x(\varphi_{34}(P))\, y(\varphi_{34}(P))\,y_{34}
\end{align*}
 \item[(b)] If $\varphi(P) \in \ma 56$ then
\begin{align*}x(P)y(P) =
x(\varphi_{56}(P))\, y(\varphi_{56}(P))  \left ( \pprod {j=k_5+1}{s}{x_{i_j}} \right ) \left ( \pprod {j=1}{k_6-1}{x_{i'_j}} \right )  y_{56}.  
\end{align*}
\end{itemize}
\end{lem}
\begin{proof}
(a) In the operation $\sigma_{s,3}$ in the definition of $\varphi,$ we loose an edge of the matching with weight $x_{\e_3}$ where $\e_3$ is the grafting edge. Since $\calg_4 = \{ \e_3\}$ in this case, we get $x(\varphi_{34}(P))=x(P).$
 
 Moreover, a direct inspection of the first three cases in Figure \ref{sigma} (with the tile $G_{s+1}$ removed) shows that 
\begin{align*}
 y(P)=y(\varphi_{34}(P)),
\end{align*}
and on the other hand $y_{34}=1$ since $s=d.$
\item[(b)] Since $s=d,$ the operation $\sigma_{s,5}$ does not apply and $\varphi_{56}(P)$ is simply given by restricting $P$ to $\calg_5 \sqcup \calg_6.$ Note that $\calg_5=\calg_1[1,k_5]$ and $\calg_1[k_5+1,s]$ is the zigzag graph shown on the left hand side of Figure \ref{figlem56}.
 \begin{figure}
\begin{center}
\scalebox{0.8}{ 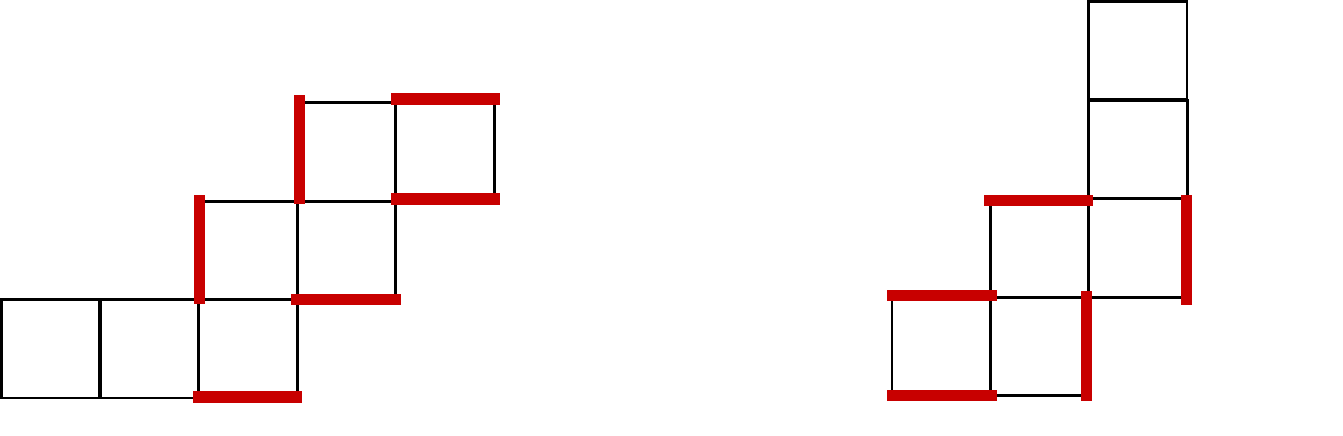}
 \caption{Proof of Lemma \ref{lem 76}}
 \label{figlem56}
 \end{center}
\end{figure}
The matching $P$ is determined on $\calg_1[k_5+1,s]$ by the initial condition given in the fifth case of Figure \ref{sigma}. Similarly, $\calg_6=\calg_2[k_6,d']$ and $\calg_2[1, k_6-1]$ is shown on the right hand side of Figure \ref{figlem56}. Here the matching $P$ is determined by the initial condition as shown in Figure \ref{figlem56}. It follows that 
\begin{align*}
x(P) = x(\varphi_{56}(P))\left ( \pprod {j=k_5+1}{s}{x_{i_j}} \right) \left ( \pprod {j=1}{k_6-1}{x_{i'_j}} \right ) .
\end{align*}
 
\end{proof}

\subsection{Skein relations}

As a corollary we obtain a new proof of the skein relations.

\begin{cor} \label{skein}
 Let $\gamma_1$ and $\gamma_2$ be two arcs which cross and let $(\gamma_3, \gamma_4)$ and $(\gamma_5,\gamma_6)$ be the two pairs of arcs  obtained by smoothing the crossing. 
 \begin{itemize}
\item[(a)] If $\gamma_1$ and $\gamma_2$ have a non-empty local overlap, then
 $$x_{\gamma_1} x_{\gamma_2} =x_{\gamma_3}x_{\gamma_4} + y(\tcalg) x_{\gamma_5} x_{\gamma_6}$$
 where $\tcalg$ is the closure of the overlap $\calg.$
 \item[(b)] If $\gamma_1$ and $\gamma_2$ have an empty local overlap, then
 $$x_{\gamma_1} x_{\gamma_2} =y_{34}x_{\gamma_3}x_{\gamma_4} + y_{56} x_{\gamma_5} x_{\gamma_6}$$
 where $y_{34}$ and $y_{56}$ are as in equation \ref{skeincoeff}.
\end{itemize}
\end{cor}
\begin{proof} 
First note that for any $(\ell,k)=(1,2),(3,4),(5,6)$
we have 
\begin{eqnarray}\label{eq skein 66}
\lefteqn{\quad\quad\call (\calg_\ell) \call (\calg_k)} \\&=& \displaystyle \frac{1}{x(\calg_\ell)y(\calg_\ell)x(\calg_k)y(\calg_k)} \displaystyle \sum_{P_\ell \in \Match (\calg_\ell)} x(P_\ell)y(P_\ell) \displaystyle \sum_{P_k \in \match (\calg_k)} x(P_k)y(P_k)\nonumber\\
 &=& \displaystyle \frac{1}{x(\calg_\ell)y(\calg_\ell)x(\calg_k)y(\calg_k)} \displaystyle \sum_{(P_\ell,P_k) \in \match \calg_\ell \times \calg_k} x(P_\ell)x(P_k)y(P_\ell)y(P_k)\nonumber\\
  &=& \displaystyle \frac{1}{x(\calg_\ell)y(\calg_\ell)x(\calg_k)y(\calg_k)} \displaystyle \sum_{P \in \ma \ell k} x(P) y(P) \nonumber\\
  &=& \la \ell k.\nonumber
\end{eqnarray}

Therefore in case (a) we have
\begin{eqnarray*}
x_{\gamma_1} x_{\gamma_2} &=& \call (\calg_1 \sqcup \calg_2)\\
&=& \call (\res_{\calg} (\calg_1, \calg_2))\\
&=& \call (\calg_3 \sqcup \calg_4) + y(\tcalg) \call (\calg_5 \sqcup \calg_6)\\
&=& x_{\gamma_3}x_{\gamma_4} + y(\tcalg) x_{\gamma_5} x_{\gamma_6}
\end{eqnarray*}
where the first equality and the last equality hold by equation (\ref{eq skein 66}), the second equality by Theorem  \ref{laurent}, and the third equality by definition.

In case (b) we have
\begin{eqnarray*}
x_{\gamma_1} x_{\gamma_2} &=& \call (\calg_1 \sqcup \calg_2)\\
&=& \call (\graft_{\calg} (\calg_1, \calg_2))\\
&=& y_{34}\call (\calg_3 \sqcup \calg_4) + y_{56} \call (\calg_5 \sqcup \calg_6)\\
&=& y_{34}x_{\gamma_3}x_{\gamma_4} + y_{56} x_{\gamma_5} x_{\gamma_6}
\end{eqnarray*}
where the first and the last equality hold by (\ref{eq skein 66}), the second equality   by Theorem \ref{laurent2}, and the third equality   by definition.
\end{proof}


\section{Proof that $\varphi$ is a bijection}\label{sect 7}
In this section, we shall construct the inverse map $  \psi$ of the map $\varphi$ of Theorem~\ref{bijections} 
$$  \psi : \match ( \re 12 )\longrightarrow \match (\calg_1 \sqcup \calg_2).$$

We start with the following two lemmas. Recall that the operation $\rho$ is defined in Figure \ref{figrho}.

\begin{lem} \label{rho} If $\calg$ is a snake graph  and $ P, P'$ are two perfect matchings of $\calg$ such that  the operation $\rho$ does not apply at any of the tiles of $\calg,$ then
\[\begin{array}{rcl}
 P \cup  P'  & = & \{ \mbox{all boundary edges of } \calg \}  \  ;\\
     P \cap  P' &  = & \hspace{50pt}\emptyset  \hspace{60pt} .
\end{array}\]
\end{lem}

\begin{proof}
 If $ P$ or $ P'$ contains an interior edge then one of the first 3 cases in Figure~\ref {figrho} applies; if $ P$ and $ P'$ have a common boundary edge, then case 4 applies.
\end{proof}


\begin{lem}  \label {rho2}
 Let $\calg$ be a snake graph with sign function $f$ and let $ P$ be a perfect matching of $\calg$ which consists of boundary edges only. Let $\NE$ be the set of all north and east edges of the boundary and $\SW$ the set of all south and west edges of the boundary. Then 
\begin{align*}
& f(a)=f(b)  &   \mbox{ if $a,b$ are both in $ P \cap \NE$ or both in $ P \cap \SW$ }  \\
\\
  & f(c)=-f(b)  &   \mbox{ if one of $a,b$ is in $ P \cap \NE$ and the other in $ P \cap \SW $.}    
\end{align*}
\end{lem}

\begin{proof}
We proceed by induction on the number of tiles, $d$, in $\calg.$ If $d=1$ then the matching consists of either the east and west edge or the north and south edge, and in both cases the result follows directly from the definition of sign functions. Suppose $d>1,$ and let $\calg'=\calg[1,d-1].$ We may assume without loss of generality that the tile $G_d$ lies east of the tile $G_{d-1}.$ If $ P$ contains the east edge of $G_d$ then $ P[1,d-1]$ is a perfect matching of $\calg',$ the east edge of $G_d$ has the same sign as the north edge of $G_{d-1}$ and the result follows by induction. Otherwise, $ P$ contains the north and the south edge of $G_d$ and $ P'= P[1,d-1]\cup \{ \mbox{ east edge of $G_{d-1}$ } \}$ is a perfect matching of $\calg'$ consisting of boundary edges. By induction, the east edge of $G_{d-1}$ has the same sign as all north and all east edges in $ P'.$ The result follows now since the north edge of $G_d$ has the same sign as the east edge of $G_{d-1}$ and the opposite sign of the south of $G_d.$  
\end{proof}

Now we want to define 
$$  \psi : \match ( \re 12 )\longrightarrow \match (\calg_1 \sqcup \calg_2).$$

Recall that $\re 12$ is a pair $(\calg_3 \sqcup \calg_4, \calg_5 \sqcup \calg_6)$ where $\calg_3$ and $\calg_4$ overlap in $\calg.$ Let $u, v, u', v'$ be such that $\calg \cong \calg_3 [u,v] \cong \calg_4 [u',v']$ is the overlap. Let $ P_i \in \match \calg_i$ for $i=3,4,5,6.$
We treat the cases $u\ne v$ and $u=v$ separately.

\subsection{Definition of $\psi$ in the case $u \neq v$}\label{sect 71}

Suppose $u \neq 1$ and $u' \neq 1.$ Then we define $  \psi ( P_3, P_4)$ as follows.
\begin{enumerate}
\item \label{a} If the pair $( P_3, P_4)$ on $(\calg_3[u-1,u+1], \calg_4[u'-1,u'+1])$ is one of the eight configurations in Figure \ref{mu} then let $  \psi ( P_3, P_4)$ be
\begin{align*}
 ( P_3[1,u-1) \cup \mu_{u,1} \cup  P_4 (u'+1,d_4] , \hspace{80pt} \\ P_4[1,u'-1) \cup \mu_{u,2} \cup  P_3 (u+1,d_3] )
\end{align*}
\item \label{b} If (\ref{a}) does not apply, let $j$ be the least integer such that $1<j \leq v-u-2$ and the local configuration of $( P_3, P_4)$ on $(\calg_3[u+j,u+j+1], \calg_4[u'+j,u'+j+1])$ is one of the four shown in Figure \ref{figrho}, if such a $j$ exists, then let $  \psi ( P_3, P_4)$ be
\begin{align*}
( P_3[1,u+j-1) \cup \rho_{j,1} \cup  P_4 (u'+j+2,d_4] , \hspace{80pt} \\ P_4[1,u'+j-1) \cup \rho_{j,2} \cup  P_3 (u+j+2,d_3] )
\end{align*}
 \item \label{c} If (\ref{a}) and (\ref{b}) do not apply, Lemma \ref{psi} below implies that $\mu$ can be applied to the pair $( P_3, P_4)$ on $(\calg_3[v-1,v+1],\calg_4[v'-1,v'+1]),$ and we let $  \psi ( P_3, P_4)$ be
 \begin{align*}
 ( P_3[1,v-1) \cup \mu_{v,1} \cup  P_4 (v'+1,d_4] ,  P_4[1,v'-1) \cup \mu_{v,2} \cup  P_3 (v+1,d_3] )
\end{align*}
\end{enumerate}
If $u=1$ or $u'=1$ then we define $  \psi ( P_3, P_4)$ using only step (\ref{b}) with $1 \leq j \leq v-u-1$ and step (\ref{c}).

\begin{lem} \label {psi} 
If (\ref{a}) and (\ref{b}) do not apply and $u \neq 1, \ u' \neq 1$ then the local configuration of $(P_3,P_4)$ on $(\calg_3[v-1,v+1], \calg_4[v'-1, v'+1])$ is one of the eight cases on the left in Figure \ref{mu}, relabeling $s-1,s,s+1$ by $v+1,v,v-1$, respectively, and rotating by 180$^{\,\circ}$.
\end{lem}

\begin{proof}
Fix a sign function $f$ on the overlap $\calg$ and denote by $f_i$ the induced sign function on $\calg_i,$ for $i=1,2,3,4.$ Since $\calg_1$ and $\calg_2$ cross in $\calg,$ one of the conditions (i) or (ii) in Definition  \ref{crossing} is satisfied. Since $u \neq 1$ and $u' \neq 1$, we must have condition (i), and, because of symmetry, we may assume   that $f_1(e_{s-1}) =-f_1(e_t)=\varepsilon.$ It follows from the definition of $\calg_3$ and $\calg_4$ that $f_3(e_{u-1})=f_1(e_{s-1})=-f_4(e'_{u'-1})$ and $f_3(e_{v})=-f_1(e_{t})=-f_4(e'_{v'}).$ In particular, we have 
 
\begin{equation} \label{*}
 f_3(e_{u-1})=f_3(e_{v})=-f_4(e'_{v'})=\varepsilon
\end{equation}
 
Since (\ref{a}) does not apply, the local configuration of $( P_3, P_4)$ on $(\calg_3[u-1,u+1], \calg_4[u'-1,u'+1])$ is one of the three cases in Figure \ref{mu2}, relabeling $s-1,s,s+1$ by $u-1,u,u+1$, respectively, where we assume without loss of generality that $\calg_3[u-1,u+1]$ is the zigzag graph and $\calg_4[u'-1,u'+1]$ is the straight graph. In particular, the north edge of $G_u$ is contained in $ P_3,$ and $f_3$ equals $-\varepsilon$ on this edge, and the south edge of $G'_{u'+1}$ is contained in $P_4$, and $f_3$ equals $- \varepsilon$ on this edge.
 
Since (\ref{b}) does not apply, Lemma \ref {rho} implies that $ P_3[u+1,v-1]$ and $ P_4[u'+1,v'-1]$ consist of boundary edges and are disjoint. It then follows from Lemma~\ref{rho2}   that for all edges $a \in P_3[u+1,v-1].$

\begin{equation*}
 f(a) =
\left \{
\begin{array}{rcl}
  \varepsilon &   & \mbox{ if $a \in  P_3$ is a south or a west edge; }   \\
  -\varepsilon&   & \mbox{ if $a \in  P_3$ is a north or a east edge; }  
\end{array}
\right.
\end{equation*}
and for all edges $a \in P_4[u'+1,v'-1]$
\begin{equation*}
 f(a) =
\left \{
\begin{array}{rcl}
  \varepsilon &   & \mbox{ if $a \in  P_4$ is a north or a east edge; }   \\
  -\varepsilon&   & \mbox{ if $a \in  P_4$ is a south or a west edge. }  
\end{array}
\right.
\end{equation*}

Now consider the local configuration at $G_v.$ Suppose first that $\calg_3[v-1,v+1]$ is a zigzag graph.   Without loss of generality assume that $G_v$ lies east of $G_{v-1}$, as in the following picture.

\begin{center}
\scalebox{1}{ 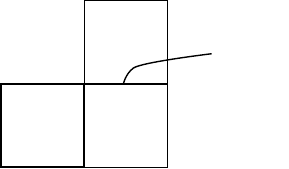}
\end{center}
 We know from equation (\ref{*}) that $f_3(e_v)=\varepsilon.$ This implies that $f_3$ equals $-\varepsilon$ on the north edge of $G_{v-1},$ and therefore the north edge of $G_{v-1}$ is an element of $ P_3.$ 
Similarly, if the south edge of $G_{v-1}$ is a boundary edge, then it is in $P_3$ since it has sign $\varepsilon$; if the south edge of $G_{v-1}$ is not a boundary edge, then $G_{v-2}$ is south of $G_{v-1}$ and the east edge of $G_{v-2}$ has sign $-\varepsilon$, hence is in $P_3$.  It follows that the southwest corner of $G_v$ is matched by an edge not in $G_v$. Therefore
the east edge of $G_v$ must be contained in $ P_3.$ This implies that the local configuration of $( P_3, P_4)$ on $(\calg_3[v-1,v+1], \calg_4[v'-1,v'+1])$ is one of the eight configurations in Figure~\ref{mu} after rotating by $180^{\circ}$ and relabeling $v=s, \ v-1=s+1, \ v+1=s-1.$
 This proves the statement in this case.
 
Now assume that $\calg_3[v-1,v+1]$ is straight. Then $\calg_4[v'-1,v'+1]$ is zigzag, and we know from equation (\ref{*}) that $f_4(e'_{v'})=-\varepsilon.$ Then an argument similar to the one above shows that the east edge of $G'_{v'}$ belongs to $ P_4,$ and again we conclude that the local configuration is one of the eight in Figure \ref{mu}.
\end{proof}

Our next step is to define $  \psi ( P_5, P_6)$.
If $1 <u,\ v' <d_4,\ 1 <u'$ and $v<d_3$ then let $  \psi( P_5, P_6)$ be
\begin{align*}
 (\tau_{u-1,1} ( P_5) \cup \eta_1 \cup \tau_{d_3-v,2}( P_6), \tau_{u-1,2}( P_5 ) \cup \eta_2 \cup \tau_{d_3-v,1}( P_6))
\end{align*}
where $\tau$ is given by Figure \ref{tau}, and the pair $(\eta_1,\eta_2)$  is the unique completion of the matching on $\calg_1\sqcup \calg_2$ given by Lemma \ref{eta} below.

\begin{figure}\begin{center}
  \scalebox{0.70}{\Large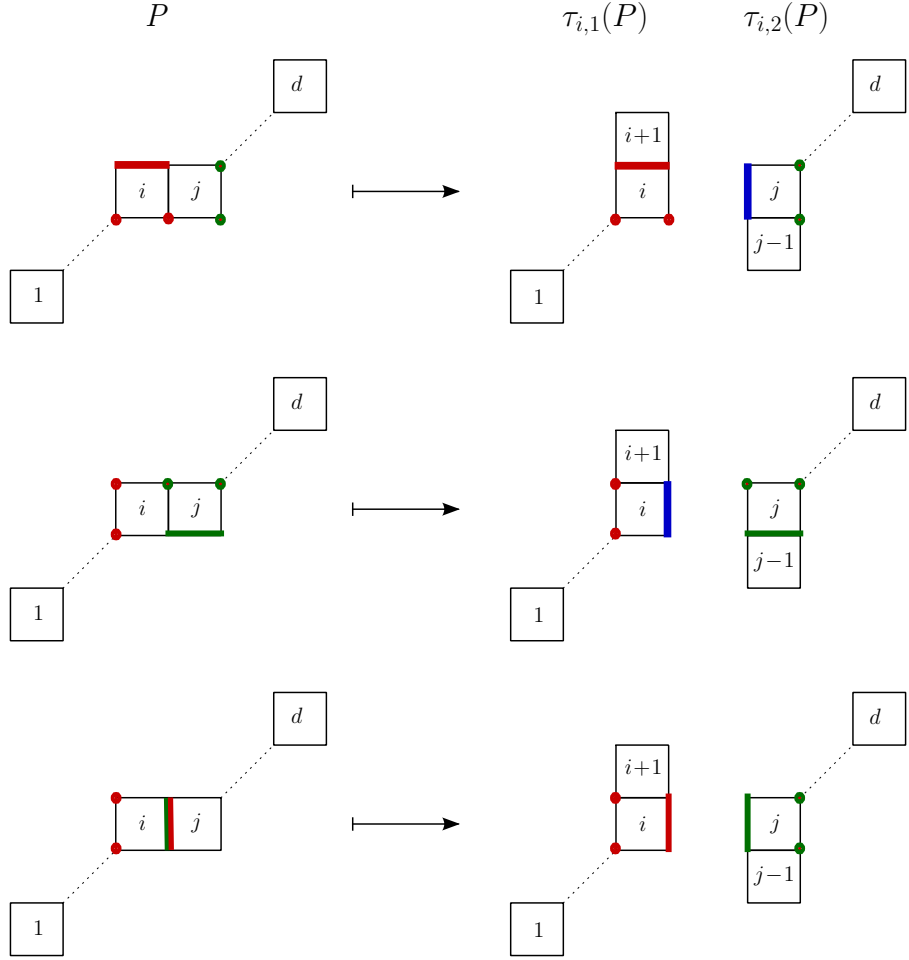}
 \caption{The operation $\tau$. The new edges are blue.}
 \label{tau}\end{center}
\end{figure}

If $u=1, v'=d_4, 1=u'$ or $v=d_3$, respectively, then remove the term $\tau_{u-1,1}( P_5)$, $\tau_{d_2-v,2}( P_6)$, $\tau_{u-1,2}( P_5)$ or $\tau_{d_3-v,1} ( P_6),$ respectively, from the definition above.


\begin{lem} \label{eta} 
 In the situation above, there exists a unique $\eta_1$ and a unique $\eta_2$ which  consist of boundary edges of $\calg$ and which are complementary perfect matchings on the overlap $\calg.$ 
\end{lem}

\begin{proof}
The uniqueness of $\eta_1$ and $\eta_2$ follows from the simple fact that if one chooses an edge $a$ in the first tile of a snake graph then there is a unique way to complete it to a perfect matching consisting only of boundary edges, except possibly for the first edge $a.$ Indeed, for example if $u \neq 1,$ the matching of the tile $G_{u-1}$ is given by $\tau_{u-1,1}( P_5)$, and there is one and only one choice to complete a matching on the tile $G_u$ using boundary edges only.

Assume $u\ne 1$ and $u'\ne 1$. There exists a unique way $\eta_1$ to extend $\tau_{u-1,1}(P_5)$ into $\calg$ using only boundary edges. We need to check that $\eta_1$ is compatible with $\tau_{d_3-v,2}(P_6)$. If $d_3=v$, there is nothing to show, so
 suppose  $ v<d_3$. Recall that we assume $u \neq v$. Since $\calg_3$ and $\calg_4$ overlap in $\calg$, we have   $f_{3}(e_{u-1})=-f_4(e'_{u'-1})$, and since $\calg_3,\calg_4$ do not cross in $\calg$, we have $f_3(e_{u-1})= f_3(e_v),$  $f_4(e'_{u'-1})=f_4(e'_{v'})$, by Definition \ref{crossing}. Then, using the definition of $\calg_3,\calg_4$, we have
 
\begin{align} \label{star}
 f_1(e_{s-1})=f_3(e_{u-1})= f_3(e_v)=f_2(e'_{t'})=\varepsilon \nonumber\\
- f_3(e_{u-1})=f_4(e'_{u'-1})=-\varepsilon\nonumber  \\
f_1(e_t) = f_4(e'_{v'})=f_4(e'_{u'-1})=- \varepsilon
\end{align}
where $\varepsilon= \pm.$ Then we have the {following local configuration} on $\calg_1$ and $\calg_2,$ respectively.
\begin{center}
\scalebox{0.8}{ 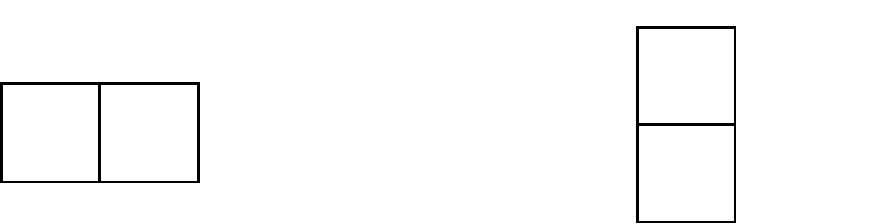}
\end{center}
Since $G_{s-1}$ is matched by $\tau_{u-1,1}(P_5),$ the north edge of $G_s$ is not in $P_1.$ Using Lemma \ref{rho2}, it follows that
 
 \begin{equation*}
 f_1(a) = 
\left \{
\begin{array}{rcl}
  \varepsilon &   & \mbox{ if $a \in  P_1 \cap \calg$ is a south or a west edge } \ ;   \\
  -\varepsilon&   & \mbox{ if $a \in  P_1 \cap \calg$ is a north or a east edge }  \ \,  .
\end{array}
\right.
\end{equation*}
In particular, equation (\ref{star}) implies that the east and west edges of $G_t$ are not in $\eta_1$ if $G_{t+1}$ is north of $G_t$ in $\calg_1$, see the left picture in the figure below.
\begin{center}
 \scalebox{0.8}{ 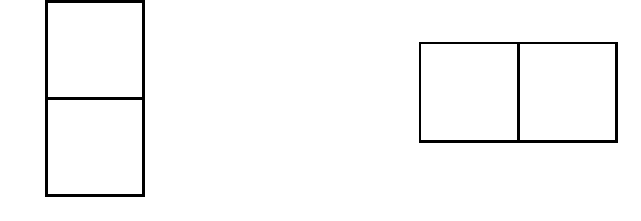}
\end{center}
Moreover if $G_{t-1}$ is west of $G_t$ then the north edge of $G_{t-1}$ is not in $\eta_1$, since its sign is $\varepsilon$. Similarly, if $G_{t+1}$ is east of $G_t$, see the right picture in the figure above, 
then 
the north and south edges of $G_t$ are not in $\eta_1$. 
Moreover, if $G_{t-1}$ is south of $G_t$ then the east edge of $G_{t-1}$ is not in $\eta_1$, since its sign is $\varepsilon$.
This shows that the completion $\eta_1$ is compatible with $\tau_{d_3-v,2}(P_6).$ 

Similarly,  since $G'_{s'-1}$ is matched by $\tau_{u-1,2}(P_5),$ the west edge of $G'_{s'}$ is not in $P_2.$ Therefore
 \begin{equation*}
 f_2(b) =
\left \{
\begin{array}{rcl}
  \varepsilon &   & \mbox{ if $a \in  P_2 \cap \calg$ is a north or a east edge }  \  \, ; \\
  -\varepsilon&   & \mbox{ if $a \in  P_2 \cap \calg$ is a south or a west edge }  \ .
\end{array}
\right.
\end{equation*}
Therefore $\eta_1$ and $\eta_2$ are complementary on the overlap $\calg.$ Again one can show that $\eta_2$ is compatible with $\tau_{d_3-v,1}(P_6)$ using the sign conditions. The cases $u=1$ or $v=1$ are similar. Note that if $u=1$ and $v'=d_4$, the matching $P_1$ is defined by $\eta_1$ only and, in this case, the complementarity condition is necessary for the uniqueness of the pair $(\eta_1,\eta_2)$.
\end{proof}

\subsection{Definition of $\psi$ in the case $u =v$} If $u=v$ then $u'=v'.$ Let $ \nu^{-1}$ be the inverse of the map given in the Figure  \ref{nu}. If $u>1,\ u'>1,\  d_3 \geq u+1,\  d_4 \geq u'+1$ then define $  \psi ( P_3,  P_4)$ as 

\begin{equation*}
 ( P_3 [1,u-1) \cup \nu^{-1}_{u,1} \cup  P_4(u'+1,d_4], \  P_4[1,u'-1) \cup \nu^{-1}_{u,2} \cup  P_3(u+1,d_3] )
\end{equation*}
 
 If $u=1,\ u'=1,\ d_3=u+1,$ or $d_4=u'+1$, respectively, then remove the term $ P_3[1,u-1), \ P_4[1,u'-1),\  P_3(u+1,d_3],$ or $ P_4(u'+1,d_4],$ respectively, from the definition above, and use the appropriate restrictions of the maps $\nu^{-1}.$
 
Now we define $  \psi ( P_5, P_6)$ where $u=v.$ If $u>1,u'>1,v<d_3,v'<d_4$ then let $\psi(P_5,P_6)$ be

\begin{equation*}
 ( \tau_{u-1,1}(P_5) \cup \tau_{d_3-u,2}(P_6), \tau_{u-1,2}(P_5)\cup \tau_{d_3-u,1}(P_6) )
\end{equation*} 
 
 Note that both $G_{u-1}, G_{u}=G'_{u'}, G'_{u'+1}$ and $G'_{u'-1},G_{u}=G'_{u'},G_{u+1}$ form a straight piece.
  
Otherwise, let $\psi(P_5,P_6)$ be

$$
\begin{cases}
 (\eta_1 \cup \tau_{d_3-u,2}(P_6), \ P_5 \cup \eta_2 \cup \tau_{d_3-u,1}(P_6)) & \mbox{ if $u=1$}\\
 (\tau_{u-1,1}(P_5) \cup \tau_{d_3-u,2}  (P_6), \ \tau_{u-1,2}(P_5) \cup \eta_2 \cup P_6) & \mbox{ if $u'=d_4$}\\
 (P_5 \cup \eta_1\cup \tau_{d_3-u,2}(P_6), \ \eta_2 \cup \tau_{d_3-u,1}(P_6)) & \mbox{ if $u'=1$}\\
 ( \tau_{u-1,1}(P_5) \cup \eta_1 \cup P_6, \  \tau_{u-1,2}(P_5) \cup \eta_2) & \mbox{ if $u=d_3$}\\
 (\eta_1 \cup P_6, \ P_5  \cup \eta_2) & \mbox{ if $u=1$ and $u'=d_4$}\\
 (P_5 \cup \eta_1 ,  \eta_2 \cup P_6) & \mbox{ if $u=d_3$ and $u'=1$}.
\end{cases}
$$
where $\eta_i$ is the unique completion using only boundary edges of $\calg_i, \ i=1,2$. 

\subsection{The map $\psi $ is well-defined}
Lemma \ref {psi} implies that  $\psi (P_3,P_4)$ is a perfect matching of $\calg_1 \sqcup \calg_2$,  if $u\ne v$, and, for $u=v$, this follows directly from the definition. Therefore to show that $\psi $ is well-defined it only remains to prove the following Lemma.
\begin{lem}\label{lem 34}
 $\psi (P_5,P_6)$ is a perfect matching of $\calg_1 \sqcup \calg_2.$  
\end{lem}

\begin{proof} If $u\ne v$, this follows from Lemma \ref{eta}.
Suppose now  that $u=v.$ If $u>1,u'>1, v<d_3,$ and $v' <d_4,$ we only need to check that $\psi(P_5,P_6)$ is well-defined on the tiles $G_s$ and $G'_{s'}$ of $\calg_1$ and $\calg_2,$ respectively, which form the overlap. The local configuration of $\calg_1[s-1,s+1] $ and  $\calg_2[s'-1,s'+1] $ is as follows.
\[
\begin{tikzpicture}
 \node at (0,0) {\begin{tikzpicture}
 \draw (0,0) -- (1,0) -- (1,3)-- (0,3)--(0,0);
 \draw (0,1)--(1,1) (0,2)--(1,2);
 \node at (.5,.5) {$s-1$};
 \node at (.5,1.5) {$s$};
 \node at (.5,2.5) {$s+1$};
 \node at (0,1) {$\bullet$};
  \node at (1,1) {$\bullet$};
   \node at (0,2) {$\bullet$};
    \node at (1,2) {$\bullet$};
\end{tikzpicture}};
\node at (3,-.5){\begin{tikzpicture}
 \draw (5,1) -- (8,1) -- (8,2)-- (5,2)--(5,1);
 \draw (6,1)--(6,2) (7,1)--(7,2);
 \node at (5.5,1.5) {$s'-1$};
 \node at (6.5,1.5) {$s'$};
 \node at (7.5,1.5) {$s'+1$};
\end{tikzpicture}
};
\end{tikzpicture}
\]
The south vertices of $G_s$ are matched by edges of $G_{u-1}$ and the north vertices are matched by the edges of $G_{u'+1}.$ Hence, this yields a matching on $\calg_1.$ A similar argument works for $\calg_2$.

Now suppose $u=1.$ Notice that we still assume $u=v.$ Then by  definition of $\psi$ in this case, we have
\begin{equation*}
 \psi(P_5,P_6) =(\eta_1 \cup \tau_{d_3-u,2}(P_6), P_5 \cup \eta_2\cup \tau_{d_3-v,1}(P_6)).
\end{equation*}
Observe that $\eta_1 \cup \tau_{d_3-u,2}(P_6)$ yields a matching on $\calg_1$ by completion with $\eta_1.$ Now we need to show that $P_5 \cup \eta_2\cup\tau_{d_3-v,1}(P_6)$ is a matching of $\calg_2.$

Since we have $u=v=1,$ it follows that $s=t=1$ and $s'=t'$ in Definition  \ref{resolution}, and therefore $\calg_5= \ocalg_2[k',1]$ where $k' < s'-1$ is the largest integer such that $f_2(e'_{k'})=f_2(e'_{s'-1}).$ It follows that the subgraph $\calg_2[k'+1,s'-1]$ is a zigzag subgraph of $\calg_2$ and thus the graph $\calg_2$ has the  shape, shown in Figure \ref{figlem34}. Note that $\calg_2[s'-1,s'+1]$ is a straight subgraph.

\begin{figure}
  \scalebox{0.80}{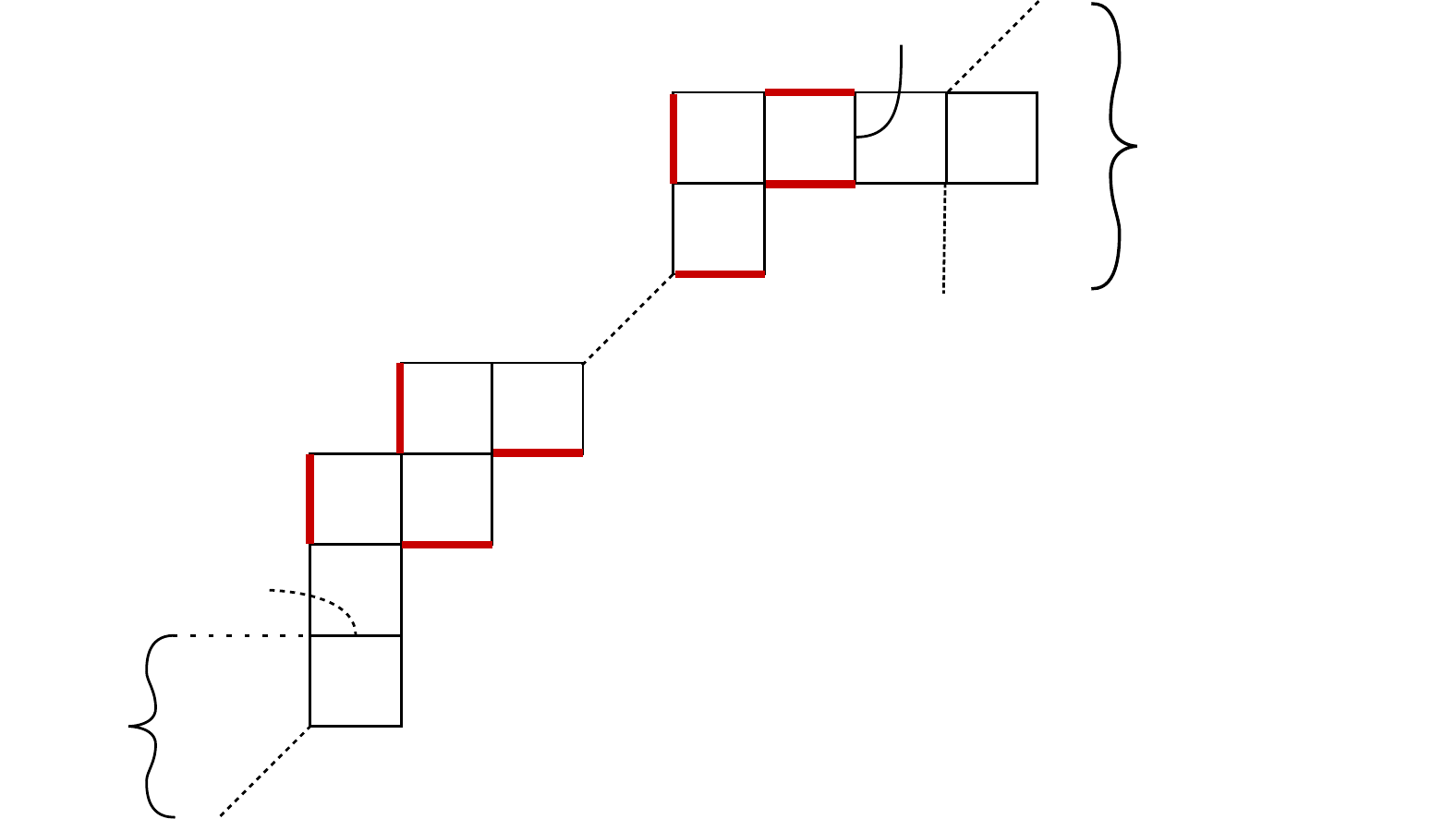}
 \caption{Proof of Lemma \ref{lem 34}}
 \label{figlem34}
\end{figure}

Moreover $\calg_2[1,k']$ is matched by $P_5$ and the tile $G'_{s'}$ is partially matched by $\tau_{d_3-v,1}(P_6).$ On the remaining zigzag graph $\calg_2[k'+1,s'-1]$ there is a unique completion $\eta_2$ consisting of all west and all south edges on the boundary and the north edge of $G'_{s'-1}$, see Figure \ref{figlem34}.

The other cases, $u=d_3, \ u'=1,$ or $u'=d_4$ are similar.
\end{proof}


%
%
%
%
%
%
%
%


\subsection{The map $\psi$ is the inverse of the map $\varphi$}
It follows immediately from the construction that $\varphi\comp \psi (P_3,P_4)=(P_3,P_4)$ for all $(P_3,P_4) \in \match (\calg_3 \sqcup \calg_4)$ and that
$\psi\circ \varphi(P_1,P_2)= (P_1,P_2)$ for all $(P_1,P_2)\in\match(\calg_1\sqcup\calg_2)$ such that $\varphi (P_1,P_2) \in \match (\calg_3 \sqcup \calg_4).$

Now let $(P_1,P_2) \in \match (\calg_1 \sqcup \calg_2)$ such that $\varphi(P_1,P_2) \in \match (\calg_5 \sqcup \calg_6).$ This means that $\varphi(P_1,P_2)$ is defined by case (iv) of the definition of $\varphi.$ In particular, the operation $\rho$ does not apply to the pair $(P_1,P_2)$ and hence Lemma \ref{rho} implies that $(P_1,P_2)$ consists of boundary edges on $\calg$ and are complementary.  Let us assume first that $s\ne 1,s' \neq 1, t \neq d$ and $t' \neq d',$ and that $s\ne t$. Then
\begin{align*}
 \varphi(P_1,P_2) = (P_1[1,s-1] \sqcup P_2[1,s'-1] \backslash \{ a_5 \}, P_2[t'+1,d'] \sqcup P_1[t+1,d] \backslash \{ a_6 \} )
\end{align*}
where $a_5$ (respectively $a_6$) is the glueing edge in the definition of $\calg_5$ (respectively $\calg_6$). Since $s \neq t$ we have $u \neq v$ and we must use the definition of $\psi$ given in section~\ref{sect 71}.  Note that $\tau_{u-1,1} (P_1[1,s-1]  \sqcup P_2[1,s'-1] \backslash \{ a_5 \})$ is exactly $P_1[1,s-1]$ and that $\tau_{d_3-v,2} (P_2[t'+1,d']  \sqcup P_1[t+1,d] \backslash \{ a_6 \})$ is exactly $P_1[t+1,d],$ since $d_6-(d_3-v)=d-t.$ Since the completions $\eta_1,\eta_2$ are unique, it follows from Lemma \ref{eta} that the first component of $\psi (\varphi(P_1,P_2))$ is equal to $P_1.$ By a similar argument one can show that the second component of $\psi(\varphi(P_1,P_2))$ is $P_2.$ If $t=s$ then the overlap consists of a single tile, and the result follows from Lemma \ref{lem 34}.

In the cases $s = 1, \ s' = 1,\  t = d, $ or $t' = d',$ a similar argument shows the same result. Note that the edges lost by restricting to $\calg_5 \sqcup \calg_6$ in the definition of $\varphi$ are recovered in the completions $\eta_1$ and $\eta_2$ in the definition of $\psi.$ This shows that $\psi \varphi$ is the identity.

It remains to show that $\varphi \psi (P_5,P_6)=(P_5,P_6)$ for all $(P_5,P_6) \in \match (\calg_5 \sqcup \calg_6).$

We shall need the following lemma.

\begin{lem}  \label {image psi}
 Let $(P_5, P_6) \in \match (\calg_5 \sqcup \calg_6)$ and let $(P_1,P_2) = \psi(P_5,P_6).$ Then
 
\begin{itemize}
 \item[(a)] \label{(a)} $P_1$ and $P_2$ do not contain interior edges of the overlap $\calg.$
 \item[(b)] \label{(b)} $P_1$ and $P_2$ do not have an edge of $\calg$ in common.
\end{itemize}
\end{lem}

\begin{proof}
 (a) This holds
  because, by definition, $\eta_i$ contains only boundary edges of $\calg.$
 
 (b) Using (a) and the fact that $\calg$ is a snake graph, it suffices to show that $P_1$ and $P_2$ do not have an edge of the first tile of $\calg$ in common. Recall that the first tile of $\calg$ is denoted by $G_s$ in $\calg_1$ and $G'_{s'}$ in $\calg_2$ and that $P_1$ is determined on $G_s$ by $\tau_{u-1,2} (P_5).$ We consider the first case of Figure \ref{tau}. The following two subcases are illustrated in Figure \ref{tau 1}. On the left hand side, $\calg_1[s-1,s+1]$ is a zigzag graph and $\calg_2[s'-1,s'+1]$ is straight. By definition of $\tau,$ $P_1$ contains the south edge of $G_s$  and therefore it also contains the north edge of $G_s$. Again by definition of $\tau$, $P_2$ contains the north edge of $G'_{s'-1}$. To 
 see this one needs to flip the rightmost picture in the first row of Figure \ref{tau} and relabel $j$ with $s'-1$ and $d$ with $1$. 
 Hence $P_2$ also contains the south edge $G'_{s'+1}.$ In particular, this shows (b) in this case.

\begin{figure}
\begin{center}
  \scalebox{0.8}{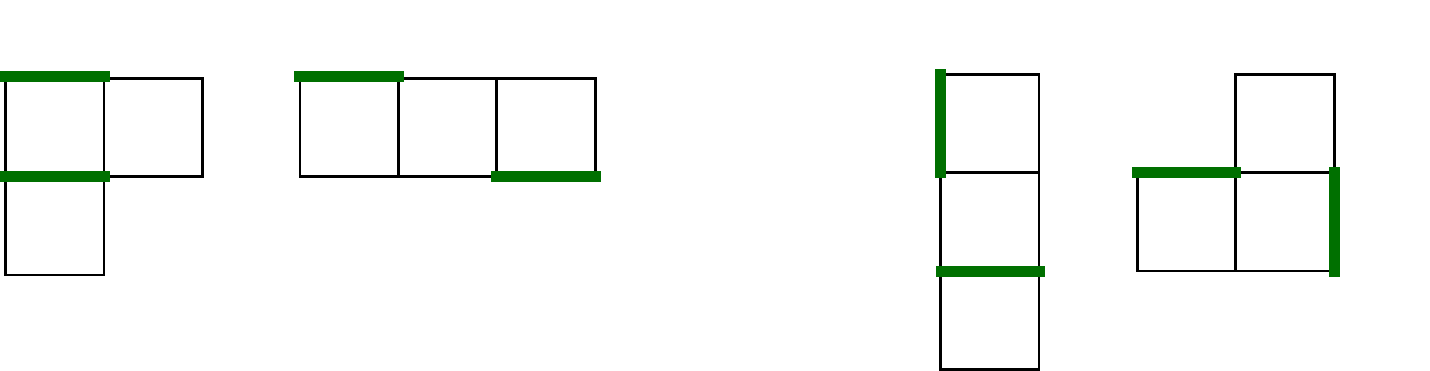}
\caption{Proof of Lemma \ref{image psi}}
\label{tau 1}
\end{center}
\end{figure}

On the right hand side of Figure \ref{tau 1}, $\calg_1[s-1,s+1]$ is straight and $\calg_2[s'-1,s'+1]$ is zigzag and again $P_1$ contains the south edge of $G_s$ by definition of $\tau$ but now it contains the west edge of $G_{s+1}$, whereas $P_2$ contains the north edge of $G'_{s'-1}$ and the east edge of $G'_{s'}.$ This shows (b) in this case.

In the remaining two cases of Figure \ref{tau}, the lemma can be shown by a similar argument.
\end{proof}

We return to proving that $\varphi\psi(P_5,P_6)=(P_5,P_6)$.
Using Lemma \ref{image psi}, we show first that $(P_1,P_2) =\psi(P_5,P_6)$ does not satisfy conditions (i),(ii),(iii) of the definition of the map $\varphi.$   Suppose first that $u>1,\ u'>1,\ v'>d_4$ and $v < d_3$, and $u\ne v$.

Lemma \ref{image psi} part (a) excludes cases 4 and 7 of Figure \ref{mu} and part (b) excludes cases 2, 3, and 6. In case 5 of the figure, the matching $P_1$ contains the south edge of $G_s$  and therefore the matching $P_5$ must be as in the first case of Figure \ref{tau}. But this implies that $P_2$ contains the north edge of $G'_{s'-1},$ a contradiction. Finally, in cases 1 and 8 of Figure \ref{mu} the matching $P_1$ contains the west edge of the tile $G_s$ which is not possible by the definition of $\tau.$ This shows that $(P_1,P_2)$ does not satisfy condition (i) of the definition of $\varphi.$

To show that  $(P_1,P_2)$ does not satisfy condition (ii), it suffices to observe that part (a) of Lemma \ref{image psi} excludes the first three cases of Figure \ref{figrho}, and part (b)  excludes the fourth case.

To show that $(P_1,P_2)$ does not satisfy condition (iii), we need to exclude all cases of Figure \ref{mu} but relabeling $s-1=t+1,\ s=t,\ s+1=t-1,\ s'-1=t'+1,\ s'+1=t'-1$ and rotating each graph by $180^{\circ}.$ Again, part (a) of Lemma \ref{image psi} excludes cases 4 and 7, part (b) excludes cases 2, 3 and 6. In case 5, $P_1$ contains the common edge of $G_t$ and $G_{t+1},$ which implies that the matching $P_6$ must be as in the second case of Figure \ref{tau}. But then $P_2$ does not contain the common edge of $G'_{t'}$ and $G'_{t'+1},$ a contradiction. Finally, in cases 1 and 8 of Figure \ref{mu}, $P_1$ would contain the east edge of $G_t$ which is not possible by the definition of $\tau.$

We have shown that $(P_1,P_2)=\psi(P_5,P_6)$ does not satisfy any of the conditions (i)-(iii) in the definition of $\varphi,$ hence condition (iv) applies and we have $\varphi \psi (P_5,P_6)=(P_5,P_6).$

If $u=v$ then the operation $\nu$ does not apply to $\psi(P_5,P_6)$, because, by definition of $\tau$, the vertices on the tiles $G_s$ and $G'_{s'}$ are matched by edges of the tiles $G_{u-1}, G_{u+1}$, $G'_{u'-1}$, $G'_{u'+1}$, but each of the pictures on the right hand side of Figure \ref{nu} contains at least one edge of $G_s$ or $G'_{s'}$.

In the case $u=1, u'=1, v=d_3$ and $v'=d_4$ the result follows by a similar argument. This shows that $\varphi \psi$ is the identity and thus both $\varphi$ and $\psi$ are bijections; which completes the proof of part (1) of Theorem \ref{bijections}.

Next we define the inverse map for part (2)
\begin{align*}
 \psi: \match (\graft_{s,\delta_3} (\calg_1,\calg_2)) \longrightarrow \match (\calg_1 \sqcup \calg_2)
\end{align*}

Recall that $\graft_{s,\delta_3} (\calg_1,\calg_2)$ is a pair $(\calg_3 \sqcup \calg_4, \ \calg_5 \sqcup \calg_6).$ Let $P_3, P_4, P_5, P_6$ be perfect matchings of $\calg_3, \calg_4, \calg_5, \calg_6,$ respectively.

Let $\psi(P_3,P_4)$ be
\begin{align*}
 &(P_3[1,s-1] \cup (\sigma^{-1}_{s,3})_1 \cup \eta_1 \cup P_4[1,d_4], \ \eta_2 \cup P_3(s+1,d_3] ) &\mbox{ if } s \neq d
\\
 &(P_3[1,s-1] \cup (\sigma^{-1}_{s,3})_1 \cup \eta_1, \ \eta_2 \cup P_3(s+1,d_3] ) & \mbox{ if } s = d
\end{align*}
where $\eta_1$ is the unique completion using only boundary edges of $\calg_1$ and $\eta_2$ is the unique completion using only boundary edges of the first tile of $\calg_2.$

Let $\psi(P_5, P_6)$ be
\begin{align*}
& (P_5 \cup \eta_1 \cup (\sigma^{-1}_{s,5})_1  \cup P_6[d'+2,d_6], \ P_6[1,d') \cup \eta_2 ) &\mbox{ if } s \neq d,\\
 & (P_5 \cup \eta_1, \ P_6 \cup \eta_2) & \mbox{ if } s =d,
\end{align*}
where $\eta_1$ is the unique completion using only boundary edges of $\calg_1$ and $\eta_2$ the unique completion using only boundary edges of the first tile of $\calg_2.$

Since the completions in the definition of $\psi$ are unique it follows immediately from the construction that $\varphi \psi$ is the identity and $\psi \varphi$ is the identity. This completes the proof of Theorem~\ref{bijections}.

{}

\end{document}

%% file: tile.pdf_tex
\begingroup%
  \makeatletter%
  \providecommand\color[2][]{%
    \errmessage{(Inkscape) Color is used for the text in Inkscape, but the package 'color.sty' is not loaded}%
    \renewcommand\color[2][]{}%
  }%
  \providecommand\transparent[1]{%
    \errmessage{(Inkscape) Transparency is used (non-zero) for the text in Inkscape, but the package 'transparent.sty' is not loaded}%
    \renewcommand\transparent[1]{}%
  }%
  \providecommand\rotatebox[2]{#2}%
  \ifx\svgwidth\undefined%
    \setlength{\unitlength}{68.59006958bp}%
    \ifx\svgscale\undefined%
      \relax%
    \else%
      \setlength{\unitlength}{\unitlength * \real{\svgscale}}%
    \fi%
  \else%
    \setlength{\unitlength}{\svgwidth}%
  \fi%
  \global\let\svgwidth\undefined%
  \global\let\svgscale\undefined%
  \makeatother%
  \begin{picture}(1,0.70549664)%
    \put(0,0){\includegraphics[width=\unitlength]{tile.pdf}}%
    \put(0.49750979,0.31061177){\color[rgb]{0,0,0}\makebox(0,0)[lb]{\smash{$G$}}}%
    \put(0.00429408,0.31029854){\color[rgb]{0,0,0}\makebox(0,0)[lb]{\smash{West}}}%
    \put(0.79324631,0.31021311){\color[rgb]{0,0,0}\makebox(0,0)[lb]{\smash{East}}}%
    \put(0.36461283,0.61868105){\color[rgb]{0,0,0}\makebox(0,0)[lb]{\smash{North}}}%
    \put(0.36019916,0.00759722){\color[rgb]{0,0,0}\makebox(0,0)[lb]{\smash{South}}}%
  \end{picture}%
\endgroup%

%% file: snakegraph.pdf_tex
\begingroup%
  \makeatletter%
  \providecommand\color[2][]{%
    \errmessage{(Inkscape) Color is used for the text in Inkscape, but the package 'color.sty' is not loaded}%
    \renewcommand\color[2][]{}%
  }%
  \providecommand\transparent[1]{%
    \errmessage{(Inkscape) Transparency is used (non-zero) for the text in Inkscape, but the package 'transparent.sty' is not loaded}%
    \renewcommand\transparent[1]{}%
  }%
  \providecommand\rotatebox[2]{#2}%
  \ifx\svgwidth\undefined%
    \setlength{\unitlength}{305.28312835bp}%
    \ifx\svgscale\undefined%
      \relax%
    \else%
      \setlength{\unitlength}{\unitlength * \real{\svgscale}}%
    \fi%
  \else%
    \setlength{\unitlength}{\svgwidth}%
  \fi%
  \global\let\svgwidth\undefined%
  \global\let\svgscale\undefined%
  \makeatother%
  \begin{picture}(1,0.40633397)%
    \put(0,0){\includegraphics[width=\unitlength]{snakegraph.pdf}}%
    \put(0.03257134,0.05898449){\color[rgb]{0,0,0}\makebox(0,0)[lb]{\smash{$ G_1$}}}%
    \put(0.1238278,0.05898449){\color[rgb]{0,0,0}\makebox(0,0)[lb]{\smash{$G_2$}}}%
    \put(0.1238278,0.1466429){\color[rgb]{0,0,0}\makebox(0,0)[lb]{\smash{$G_3$}}}%
    \put(0.1238278,0.24005812){\color[rgb]{0,0,0}\makebox(0,0)[lb]{\smash{$G_4$}}}%
    \put(0.21816646,0.24005812){\color[rgb]{0,0,0}\makebox(0,0)[lb]{\smash{$G_5$}}}%
    \put(0.21816646,0.33398742){\color[rgb]{0,0,0}\makebox(0,0)[lb]{\smash{$G_6$}}}%
    \put(0.30654445,0.33151837){\color[rgb]{0,0,0}\makebox(0,0)[lb]{\smash{$G_7$}}}%
    \put(0.40088315,0.33398742){\color[rgb]{0,0,0}\makebox(0,0)[lb]{\smash{$G_8$}}}%
    \put(0.12920883,0.11006521){\color[rgb]{0,0,0}\makebox(0,0)[lb]{\smash{$e_2$}}}%
    \put(0.12765513,0.20060384){\color[rgb]{0,0,0}\makebox(0,0)[lb]{\smash{$e_3$}}}%
    \put(0.18145749,0.24521087){\color[rgb]{0,0,0}\makebox(0,0)[lb]{\smash{$e_4$}}}%
    \put(0.21717613,0.29084621){\color[rgb]{0,0,0}\makebox(0,0)[lb]{\smash{$e_5$}}}%
    \put(0.26942255,0.33491963){\color[rgb]{0,0,0}\makebox(0,0)[lb]{\smash{$e_6$}}}%
    \put(0.36246002,0.33428128){\color[rgb]{0,0,0}\makebox(0,0)[lb]{\smash{$e_7$}}}%
    \put(0.08216134,0.06354145){\color[rgb]{0,0,0}\makebox(0,0)[lb]{\smash{$e_1$}}}%
  \end{picture}%
\endgroup%

%% file: overlap.pdf_tex
\begingroup%
  \makeatletter%
  \providecommand\color[2][]{%
    \errmessage{(Inkscape) Color is used for the text in Inkscape, but the package 'color.sty' is not loaded}%
    \renewcommand\color[2][]{}%
  }%
  \providecommand\transparent[1]{%
    \errmessage{(Inkscape) Transparency is used (non-zero) for the text in Inkscape, but the package 'transparent.sty' is not loaded}%
    \renewcommand\transparent[1]{}%
  }%
  \providecommand\rotatebox[2]{#2}%
  \ifx\svgwidth\undefined%
    \setlength{\unitlength}{312.6bp}%
    \ifx\svgscale\undefined%
      \relax%
    \else%
      \setlength{\unitlength}{\unitlength * \real{\svgscale}}%
    \fi%
  \else%
    \setlength{\unitlength}{\svgwidth}%
  \fi%
  \global\let\svgwidth\undefined%
  \global\let\svgscale\undefined%
  \makeatother%
  \begin{picture}(1,0.36540307)%
    \put(0,0){\includegraphics[width=\unitlength]{overlap.pdf}}%
  \end{picture}%
\endgroup%

%% file: overlapb.pdf_tex
\begingroup%
  \makeatletter%
  \providecommand\color[2][]{%
    \errmessage{(Inkscape) Color is used for the text in Inkscape, but the package 'color.sty' is not loaded}%
    \renewcommand\color[2][]{}%
  }%
  \providecommand\transparent[1]{%
    \errmessage{(Inkscape) Transparency is used (non-zero) for the text in Inkscape, but the package 'transparent.sty' is not loaded}%
    \renewcommand\transparent[1]{}%
  }%
  \providecommand\rotatebox[2]{#2}%
  \ifx\svgwidth\undefined%
    \setlength{\unitlength}{156.7bp}%
    \ifx\svgscale\undefined%
      \relax%
    \else%
      \setlength{\unitlength}{\unitlength * \real{\svgscale}}%
    \fi%
  \else%
    \setlength{\unitlength}{\svgwidth}%
  \fi%
  \global\let\svgwidth\undefined%
  \global\let\svgscale\undefined%
  \makeatother%
  \begin{picture}(1,0.54770262)%
    \put(0,0){\includegraphics[width=\unitlength]{overlapb.pdf}}%
  \end{picture}%
\endgroup%

%% file: figres.pdf_tex
\begingroup%
  \makeatletter%
  \providecommand\color[2][]{%
    \errmessage{(Inkscape) Color is used for the text in Inkscape, but the package 'color.sty' is not loaded}%
    \renewcommand\color[2][]{}%
  }%
  \providecommand\transparent[1]{%
    \errmessage{(Inkscape) Transparency is used (non-zero) for the text in Inkscape, but the package 'transparent.sty' is not loaded}%
    \renewcommand\transparent[1]{}%
  }%
  \providecommand\rotatebox[2]{#2}%
  \ifx\svgwidth\undefined%
    \setlength{\unitlength}{478.54848633bp}%
    \ifx\svgscale\undefined%
      \relax%
    \else%
      \setlength{\unitlength}{\unitlength * \real{\svgscale}}%
    \fi%
  \else%
    \setlength{\unitlength}{\svgwidth}%
  \fi%
  \global\let\svgwidth\undefined%
  \global\let\svgscale\undefined%
  \makeatother%
  \begin{picture}(1,0.62198504)%
    \put(0,0){\includegraphics[width=\unitlength]{figres.pdf}}%
    \put(0.04946672,0.59554788){\color[rgb]{0,0,0}\makebox(0,0)[lb]{\smash{$\mathcal{G}_1$}}}%
    \put(0.21222099,0.59554788){\color[rgb]{0,0,0}\makebox(0,0)[lb]{\smash{$\mathcal{G}_2$}}}%
    \put(0.42385508,0.59554788){\color[rgb]{0,0,0}\makebox(0,0)[lb]{\smash{$\mathcal{G}_3$}}}%
    \put(0.56688607,0.59554788){\color[rgb]{0,0,0}\makebox(0,0)[lb]{\smash{$\mathcal{G}_4$}}}%
    \put(0.77702572,0.59554788){\color[rgb]{0,0,0}\makebox(0,0)[lb]{\smash{$\mathcal{G}_5$}}}%
    \put(0.88089471,0.59554788){\color[rgb]{0,0,0}\makebox(0,0)[lb]{\smash{$\mathcal{G}_6$}}}%
    \put(0.09343927,0.03819421){\color[rgb]{0,0,0}\makebox(0,0)[lb]{\smash{$e_0'$}}}%
    \put(0.77550185,0.03819421){\color[rgb]{0,0,0}\makebox(0,0)[lb]{\smash{$e_0'$}}}%
  \end{picture}%
\endgroup%

%% file: figgraft.pdf_tex
\begingroup%
  \makeatletter%
  \providecommand\color[2][]{%
    \errmessage{(Inkscape) Color is used for the text in Inkscape, but the package 'color.sty' is not loaded}%
    \renewcommand\color[2][]{}%
  }%
  \providecommand\transparent[1]{%
    \errmessage{(Inkscape) Transparency is used (non-zero) for the text in Inkscape, but the package 'transparent.sty' is not loaded}%
    \renewcommand\transparent[1]{}%
  }%
  \providecommand\rotatebox[2]{#2}%
  \ifx\svgwidth\undefined%
    \setlength{\unitlength}{312.66984863bp}%
    \ifx\svgscale\undefined%
      \relax%
    \else%
      \setlength{\unitlength}{\unitlength * \real{\svgscale}}%
    \fi%
  \else%
    \setlength{\unitlength}{\svgwidth}%
  \fi%
  \global\let\svgwidth\undefined%
  \global\let\svgscale\undefined%
  \makeatother%
  \begin{picture}(1,0.22787615)%
    \put(0,0){\includegraphics[width=\unitlength]{figgraft.pdf}}%
    \put(0.92214271,0.0802294){\color[rgb]{0,0,0}\makebox(0,0)[lb]{\smash{$G_1'$}}}%
    \put(0.92579783,0.01434845){\color[rgb]{0,0,0}\makebox(0,0)[lb]{\smash{$\e_3'$}}}%
    \put(0.86247545,0.0802294){\color[rgb]{0,0,0}\makebox(0,0)[lb]{\smash{$\e_5'$}}}%
    \put(0.7268886,0.0802294){\color[rgb]{0,0,0}\makebox(0,0)[lb]{\smash{$G_{s+1}$}}}%
    \put(0.74740947,0.01458317){\color[rgb]{0,0,0}\makebox(0,0)[lb]{\smash{$\e_5$}}}%
    \put(0.64924712,0.0802294){\color[rgb]{0,0,0}\makebox(0,0)[lb]{\smash{$G_s$}}}%
    \put(0.65436434,0.14611037){\color[rgb]{0,0,0}\makebox(0,0)[lb]{\smash{$\e_3$}}}%
    \put(0.06168075,0.08105505){\color[rgb]{0,0,0}\makebox(0,0)[lb]{\smash{$G_s$}}}%
    \put(0.04885712,0.17160831){\color[rgb]{0,0,0}\makebox(0,0)[lb]{\smash{$G_{s+1}$}}}%
    \put(0.13167701,0.08105505){\color[rgb]{0,0,0}\makebox(0,0)[lb]{\smash{$\e_3$}}}%
    \put(-0.00006996,0.17084109){\color[rgb]{0,0,0}\makebox(0,0)[lb]{\smash{$\e_5$}}}%
    \put(0.2484092,0.01443661){\color[rgb]{0,0,0}\makebox(0,0)[lb]{\smash{$\e_5'$}}}%
    \put(0.18197352,0.08105505){\color[rgb]{0,0,0}\makebox(0,0)[lb]{\smash{$\e_3'$}}}%
    \put(0.24228363,0.08105505){\color[rgb]{0,0,0}\makebox(0,0)[lb]{\smash{$G_1'$}}}%
  \end{picture}%
\endgroup%

%% file: figgrafting.pdf_tex
\begingroup%
  \makeatletter%
  \providecommand\color[2][]{%
    \errmessage{(Inkscape) Color is used for the text in Inkscape, but the package 'color.sty' is not loaded}%
    \renewcommand\color[2][]{}%
  }%
  \providecommand\transparent[1]{%
    \errmessage{(Inkscape) Transparency is used (non-zero) for the text in Inkscape, but the package 'transparent.sty' is not loaded}%
    \renewcommand\transparent[1]{}%
  }%
  \providecommand\rotatebox[2]{#2}%
  \ifx\svgwidth\undefined%
    \setlength{\unitlength}{444.94848633bp}%
    \ifx\svgscale\undefined%
      \relax%
    \else%
      \setlength{\unitlength}{\unitlength * \real{\svgscale}}%
    \fi%
  \else%
    \setlength{\unitlength}{\svgwidth}%
  \fi%
  \global\let\svgwidth\undefined%
  \global\let\svgscale\undefined%
  \makeatother%
  \begin{picture}(1,0.51320142)%
    \put(0,0){\includegraphics[width=\unitlength]{figgrafting.pdf}}%
    \put(0.05320217,0.48476787){\color[rgb]{0,0,0}\makebox(0,0)[lb]{\smash{$\mathcal{G}_1$}}}%
    \put(0.22824672,0.48476787){\color[rgb]{0,0,0}\makebox(0,0)[lb]{\smash{$\mathcal{G}_2$}}}%
    \put(0.41630708,0.48476787){\color[rgb]{0,0,0}\makebox(0,0)[lb]{\smash{$\mathcal{G}_3$}}}%
    \put(0.58092673,0.48476787){\color[rgb]{0,0,0}\makebox(0,0)[lb]{\smash{$\mathcal{G}_4$}}}%
    \put(0.70265321,0.48476787){\color[rgb]{0,0,0}\makebox(0,0)[lb]{\smash{$\mathcal{G}_5$}}}%
    \put(0.87190055,0.48476787){\color[rgb]{0,0,0}\makebox(0,0)[lb]{\smash{$\mathcal{G}_6$}}}%
    \put(0.13505749,0.13758873){\color[rgb]{0,0,0}\makebox(0,0)[lb]{\smash{$\delta_3$}}}%
    \put(0.06589638,0.36046395){\color[rgb]{0,0,0}\makebox(0,0)[lb]{\smash{$\delta_3$}}}%
    \put(0.5766602,0.0119961){\color[rgb]{0,0,0}\makebox(0,0)[lb]{\smash{$\delta_3$}}}%
    \put(0.23003752,0.01330611){\color[rgb]{0,0,0}\makebox(0,0)[lb]{\smash{$\delta_3'$}}}%
    \put(0.22741745,0.26745182){\color[rgb]{0,0,0}\makebox(0,0)[lb]{\smash{$\delta_3'$}}}%
    \put(0.19859682,0.29889254){\color[rgb]{0,0,0}\makebox(0,0)[lb]{\smash{$\delta_5'$}}}%
    \put(0.10388725,0.29889254){\color[rgb]{0,0,0}\makebox(0,0)[lb]{\smash{$\delta_5$}}}%
    \put(0.54723494,0.28300203){\color[rgb]{0,0,0}\makebox(0,0)[lb]{\smash{$k_4=s+2$}}}%
    \put(0.68910562,0.28166536){\color[rgb]{0,0,0}\makebox(0,0)[lb]{\smash{$k_5=1$}}}%
    \put(0.70708523,0.00477935){\color[rgb]{0,0,0}\makebox(0,0)[lb]{\smash{$k_5=4$}}}%
    \put(0.87249764,0.00477935){\color[rgb]{0,0,0}\makebox(0,0)[lb]{\smash{$k_6=3$}}}%
    \put(0.07296716,0.33033327){\color[rgb]{0,0,0}\makebox(0,0)[lb]{\smash{$s$}}}%
    \put(0.13846864,0.10893828){\color[rgb]{0,0,0}\makebox(0,0)[lb]{\smash{$s$}}}%
  \end{picture}%
\endgroup%

%% file: mu.pdf_tex
\begingroup%
  \makeatletter%
  \providecommand\color[2][]{%
    \errmessage{(Inkscape) Color is used for the text in Inkscape, but the package 'color.sty' is not loaded}%
    \renewcommand\color[2][]{}%
  }%
  \providecommand\transparent[1]{%
    \errmessage{(Inkscape) Transparency is used (non-zero) for the text in Inkscape, but the package 'transparent.sty' is not loaded}%
    \renewcommand\transparent[1]{}%
  }%
  \providecommand\rotatebox[2]{#2}%
  \ifx\svgwidth\undefined%
    \setlength{\unitlength}{435.90464675bp}%
    \ifx\svgscale\undefined%
      \relax%
    \else%
      \setlength{\unitlength}{\unitlength * \real{\svgscale}}%
    \fi%
  \else%
    \setlength{\unitlength}{\svgwidth}%
  \fi%
  \global\let\svgwidth\undefined%
  \global\let\svgscale\undefined%
  \makeatother%
  \begin{picture}(1,1.28234474)%
    \put(0,0){\includegraphics[width=\unitlength]{mu.pdf}}%
    \put(0.03784336,1.23637387){\color[rgb]{0,0,0}\makebox(0,0)[lb]{\smash{$s$}}}%
    \put(0.08389649,1.2365387){\color[rgb]{0,0,0}\makebox(0,0)[lb]{\smash{$s\!+\!1$}}}%
    \put(0.02064079,1.16990937){\color[rgb]{0,0,0}\makebox(0,0)[lb]{\smash{$s\!-\!1$}}}%
    \put(0.21192963,1.23637387){\color[rgb]{0,0,0}\makebox(0,0)[lb]{\smash{$s'\!-\!1$}}}%
    \put(0.29493107,1.23637387){\color[rgb]{0,0,0}\makebox(0,0)[lb]{\smash{$s'$}}}%
    \put(0.34288053,1.23637387){\color[rgb]{0,0,0}\makebox(0,0)[lb]{\smash{$s'\!+\!1$}}}%
    \put(0.62512777,1.23637387){\color[rgb]{0,0,0}\makebox(0,0)[lb]{\smash{$s$}}}%
    \put(0.67118187,1.2365387){\color[rgb]{0,0,0}\makebox(0,0)[lb]{\smash{$s\!+\!1$}}}%
    \put(0.6079252,1.16990937){\color[rgb]{0,0,0}\makebox(0,0)[lb]{\smash{$s\!-\!1$}}}%
    \put(0.79921403,1.23637387){\color[rgb]{0,0,0}\makebox(0,0)[lb]{\smash{$s'\!-\!1$}}}%
    \put(0.88221547,1.23637387){\color[rgb]{0,0,0}\makebox(0,0)[lb]{\smash{$s'$}}}%
    \put(0.93016499,1.23637387){\color[rgb]{0,0,0}\makebox(0,0)[lb]{\smash{$s'\!+\!1$}}}%
    \put(0.03784336,1.07365585){\color[rgb]{0,0,0}\makebox(0,0)[lb]{\smash{$s$}}}%
    \put(0.08389649,1.07382082){\color[rgb]{0,0,0}\makebox(0,0)[lb]{\smash{$s\!+\!1$}}}%
    \put(0.02064079,1.00719133){\color[rgb]{0,0,0}\makebox(0,0)[lb]{\smash{$s\!-\!1$}}}%
    \put(0.21192963,1.07365585){\color[rgb]{0,0,0}\makebox(0,0)[lb]{\smash{$s'\!-\!1$}}}%
    \put(0.29493107,1.07365585){\color[rgb]{0,0,0}\makebox(0,0)[lb]{\smash{$s'$}}}%
    \put(0.34288053,1.07365585){\color[rgb]{0,0,0}\makebox(0,0)[lb]{\smash{$s'\!+\!1$}}}%
    \put(0.62512777,1.07365585){\color[rgb]{0,0,0}\makebox(0,0)[lb]{\smash{$s$}}}%
    \put(0.67118187,1.07382082){\color[rgb]{0,0,0}\makebox(0,0)[lb]{\smash{$s\!+\!1$}}}%
    \put(0.6079252,1.00719133){\color[rgb]{0,0,0}\makebox(0,0)[lb]{\smash{$s\!-\!1$}}}%
    \put(0.79921403,1.07365585){\color[rgb]{0,0,0}\makebox(0,0)[lb]{\smash{$s'\!-\!1$}}}%
    \put(0.88221547,1.07365585){\color[rgb]{0,0,0}\makebox(0,0)[lb]{\smash{$s'$}}}%
    \put(0.93016499,1.07365585){\color[rgb]{0,0,0}\makebox(0,0)[lb]{\smash{$s'\!+\!1$}}}%
    \put(0.03784336,0.91215264){\color[rgb]{0,0,0}\makebox(0,0)[lb]{\smash{$s$}}}%
    \put(0.08389649,0.91231818){\color[rgb]{0,0,0}\makebox(0,0)[lb]{\smash{$s\!+\!1$}}}%
    \put(0.02064079,0.84568812){\color[rgb]{0,0,0}\makebox(0,0)[lb]{\smash{$s\!-\!1$}}}%
    \put(0.21192963,0.91215264){\color[rgb]{0,0,0}\makebox(0,0)[lb]{\smash{$s'\!-\!1$}}}%
    \put(0.29493107,0.91215264){\color[rgb]{0,0,0}\makebox(0,0)[lb]{\smash{$s'$}}}%
    \put(0.34288053,0.91215264){\color[rgb]{0,0,0}\makebox(0,0)[lb]{\smash{$s'\!+\!1$}}}%
    \put(0.62512777,0.91215264){\color[rgb]{0,0,0}\makebox(0,0)[lb]{\smash{$s$}}}%
    \put(0.67118187,0.91231818){\color[rgb]{0,0,0}\makebox(0,0)[lb]{\smash{$s\!+\!1$}}}%
    \put(0.6079252,0.84568812){\color[rgb]{0,0,0}\makebox(0,0)[lb]{\smash{$s\!-\!1$}}}%
    \put(0.79921403,0.91215264){\color[rgb]{0,0,0}\makebox(0,0)[lb]{\smash{$s'\!-\!1$}}}%
    \put(0.88221547,0.91215264){\color[rgb]{0,0,0}\makebox(0,0)[lb]{\smash{$s'$}}}%
    \put(0.93016499,0.91215264){\color[rgb]{0,0,0}\makebox(0,0)[lb]{\smash{$s'\!+\!1$}}}%
    \put(0.03784336,0.74869689){\color[rgb]{0,0,0}\makebox(0,0)[lb]{\smash{$s$}}}%
    \put(0.08389649,0.74886207){\color[rgb]{0,0,0}\makebox(0,0)[lb]{\smash{$s\!+\!1$}}}%
    \put(0.02064079,0.68223237){\color[rgb]{0,0,0}\makebox(0,0)[lb]{\smash{$s\!-\!1$}}}%
    \put(0.21192963,0.74869689){\color[rgb]{0,0,0}\makebox(0,0)[lb]{\smash{$s'\!-\!1$}}}%
    \put(0.29493107,0.74869689){\color[rgb]{0,0,0}\makebox(0,0)[lb]{\smash{$s'$}}}%
    \put(0.34288053,0.74869689){\color[rgb]{0,0,0}\makebox(0,0)[lb]{\smash{$s'\!+\!1$}}}%
    \put(0.62512777,0.74869689){\color[rgb]{0,0,0}\makebox(0,0)[lb]{\smash{$s$}}}%
    \put(0.67118187,0.74886207){\color[rgb]{0,0,0}\makebox(0,0)[lb]{\smash{$s\!+\!1$}}}%
    \put(0.6079252,0.68223237){\color[rgb]{0,0,0}\makebox(0,0)[lb]{\smash{$s\!-\!1$}}}%
    \put(0.79921403,0.74869689){\color[rgb]{0,0,0}\makebox(0,0)[lb]{\smash{$s'\!-\!1$}}}%
    \put(0.88221547,0.74869689){\color[rgb]{0,0,0}\makebox(0,0)[lb]{\smash{$s'$}}}%
    \put(0.93016499,0.74869689){\color[rgb]{0,0,0}\makebox(0,0)[lb]{\smash{$s'\!+\!1$}}}%
    \put(0.03784336,0.58719367){\color[rgb]{0,0,0}\makebox(0,0)[lb]{\smash{$s$}}}%
    \put(0.08389649,0.58735813){\color[rgb]{0,0,0}\makebox(0,0)[lb]{\smash{$s\!+\!1$}}}%
    \put(0.02064079,0.52072916){\color[rgb]{0,0,0}\makebox(0,0)[lb]{\smash{$s\!-\!1$}}}%
    \put(0.21192963,0.58719367){\color[rgb]{0,0,0}\makebox(0,0)[lb]{\smash{$s'\!-\!1$}}}%
    \put(0.29493107,0.58719367){\color[rgb]{0,0,0}\makebox(0,0)[lb]{\smash{$s'$}}}%
    \put(0.34288053,0.58719367){\color[rgb]{0,0,0}\makebox(0,0)[lb]{\smash{$s'\!+\!1$}}}%
    \put(0.62512777,0.58719367){\color[rgb]{0,0,0}\makebox(0,0)[lb]{\smash{$s$}}}%
    \put(0.67118187,0.58735813){\color[rgb]{0,0,0}\makebox(0,0)[lb]{\smash{$s\!+\!1$}}}%
    \put(0.6079252,0.52072916){\color[rgb]{0,0,0}\makebox(0,0)[lb]{\smash{$s\!-\!1$}}}%
    \put(0.79921403,0.58719367){\color[rgb]{0,0,0}\makebox(0,0)[lb]{\smash{$s'\!-\!1$}}}%
    \put(0.88221547,0.58719367){\color[rgb]{0,0,0}\makebox(0,0)[lb]{\smash{$s'$}}}%
    \put(0.93016499,0.58719367){\color[rgb]{0,0,0}\makebox(0,0)[lb]{\smash{$s'\!+\!1$}}}%
    \put(0.03784336,0.42569046){\color[rgb]{0,0,0}\makebox(0,0)[lb]{\smash{$s$}}}%
    \put(0.08389649,0.42585419){\color[rgb]{0,0,0}\makebox(0,0)[lb]{\smash{$s\!+\!1$}}}%
    \put(0.02064079,0.35922595){\color[rgb]{0,0,0}\makebox(0,0)[lb]{\smash{$s\!-\!1$}}}%
    \put(0.21192963,0.42569046){\color[rgb]{0,0,0}\makebox(0,0)[lb]{\smash{$s'\!-\!1$}}}%
    \put(0.29493107,0.42569046){\color[rgb]{0,0,0}\makebox(0,0)[lb]{\smash{$s'$}}}%
    \put(0.34288053,0.42569046){\color[rgb]{0,0,0}\makebox(0,0)[lb]{\smash{$s'\!+\!1$}}}%
    \put(0.62512777,0.42569046){\color[rgb]{0,0,0}\makebox(0,0)[lb]{\smash{$s$}}}%
    \put(0.67118187,0.42585419){\color[rgb]{0,0,0}\makebox(0,0)[lb]{\smash{$s\!+\!1$}}}%
    \put(0.6079252,0.35922595){\color[rgb]{0,0,0}\makebox(0,0)[lb]{\smash{$s\!-\!1$}}}%
    \put(0.79921403,0.42569046){\color[rgb]{0,0,0}\makebox(0,0)[lb]{\smash{$s'\!-\!1$}}}%
    \put(0.88221547,0.42569046){\color[rgb]{0,0,0}\makebox(0,0)[lb]{\smash{$s'$}}}%
    \put(0.93016499,0.42569046){\color[rgb]{0,0,0}\makebox(0,0)[lb]{\smash{$s'\!+\!1$}}}%
    \put(0.03784336,0.26051667){\color[rgb]{0,0,0}\makebox(0,0)[lb]{\smash{$s$}}}%
    \put(0.08389649,0.2606797){\color[rgb]{0,0,0}\makebox(0,0)[lb]{\smash{$s\!+\!1$}}}%
    \put(0.02064079,0.19405221){\color[rgb]{0,0,0}\makebox(0,0)[lb]{\smash{$s\!-\!1$}}}%
    \put(0.21192963,0.26051667){\color[rgb]{0,0,0}\makebox(0,0)[lb]{\smash{$s'\!-\!1$}}}%
    \put(0.29493107,0.26051667){\color[rgb]{0,0,0}\makebox(0,0)[lb]{\smash{$s'$}}}%
    \put(0.34288053,0.26051667){\color[rgb]{0,0,0}\makebox(0,0)[lb]{\smash{$s'\!+\!1$}}}%
    \put(0.62512777,0.26051667){\color[rgb]{0,0,0}\makebox(0,0)[lb]{\smash{$s$}}}%
    \put(0.67118187,0.2606797){\color[rgb]{0,0,0}\makebox(0,0)[lb]{\smash{$s\!+\!1$}}}%
    \put(0.6079252,0.19405221){\color[rgb]{0,0,0}\makebox(0,0)[lb]{\smash{$s\!-\!1$}}}%
    \put(0.79921403,0.26051667){\color[rgb]{0,0,0}\makebox(0,0)[lb]{\smash{$s'\!-\!1$}}}%
    \put(0.88221547,0.26051667){\color[rgb]{0,0,0}\makebox(0,0)[lb]{\smash{$s'$}}}%
    \put(0.93016499,0.26051667){\color[rgb]{0,0,0}\makebox(0,0)[lb]{\smash{$s'\!+\!1$}}}%
    \put(0.03784336,0.09901346){\color[rgb]{0,0,0}\makebox(0,0)[lb]{\smash{$s$}}}%
    \put(0.08389649,0.09917576){\color[rgb]{0,0,0}\makebox(0,0)[lb]{\smash{$s\!+\!1$}}}%
    \put(0.02064079,0.032549){\color[rgb]{0,0,0}\makebox(0,0)[lb]{\smash{$s\!-\!1$}}}%
    \put(0.21192963,0.09901346){\color[rgb]{0,0,0}\makebox(0,0)[lb]{\smash{$s'\!-\!1$}}}%
    \put(0.29493107,0.09901346){\color[rgb]{0,0,0}\makebox(0,0)[lb]{\smash{$s'$}}}%
    \put(0.34288053,0.09901346){\color[rgb]{0,0,0}\makebox(0,0)[lb]{\smash{$s'\!+\!1$}}}%
    \put(0.62512777,0.09901346){\color[rgb]{0,0,0}\makebox(0,0)[lb]{\smash{$s$}}}%
    \put(0.67118187,0.09917576){\color[rgb]{0,0,0}\makebox(0,0)[lb]{\smash{$s\!+\!1$}}}%
    \put(0.6079252,0.032549){\color[rgb]{0,0,0}\makebox(0,0)[lb]{\smash{$s\!-\!1$}}}%
    \put(0.79921403,0.09901346){\color[rgb]{0,0,0}\makebox(0,0)[lb]{\smash{$s'\!-\!1$}}}%
    \put(0.88221547,0.09901346){\color[rgb]{0,0,0}\makebox(0,0)[lb]{\smash{$s'$}}}%
    \put(0.93016499,0.09901346){\color[rgb]{0,0,0}\makebox(0,0)[lb]{\smash{$s'\!+\!1$}}}%
  \end{picture}%
\endgroup%

%% file: mu2.pdf_tex
\begingroup%
  \makeatletter%
  \providecommand\color[2][]{%
    \errmessage{(Inkscape) Color is used for the text in Inkscape, but the package 'color.sty' is not loaded}%
    \renewcommand\color[2][]{}%
  }%
  \providecommand\transparent[1]{%
    \errmessage{(Inkscape) Transparency is used (non-zero) for the text in Inkscape, but the package 'transparent.sty' is not loaded}%
    \renewcommand\transparent[1]{}%
  }%
  \providecommand\rotatebox[2]{#2}%
  \ifx\svgwidth\undefined%
    \setlength{\unitlength}{338.28564745bp}%
    \ifx\svgscale\undefined%
      \relax%
    \else%
      \setlength{\unitlength}{\unitlength * \real{\svgscale}}%
    \fi%
  \else%
    \setlength{\unitlength}{\svgwidth}%
  \fi%
  \global\let\svgwidth\undefined%
  \global\let\svgscale\undefined%
  \makeatother%
  \begin{picture}(1,0.66453051)%
    \put(0,0){\includegraphics[width=\unitlength]{mu2.pdf}}%
    \put(0.04238415,0.54817851){\color[rgb]{0,0,0}\makebox(0,0)[lb]{\smash{$s$}}}%
    \put(0.10779182,0.54817851){\color[rgb]{0,0,0}\makebox(0,0)[lb]{\smash{$s\!+\!1$}}}%
    \put(0.02439289,0.46448199){\color[rgb]{0,0,0}\makebox(0,0)[lb]{\smash{$s\!-\!1$}}}%
    \put(0.26991112,0.54817851){\color[rgb]{0,0,0}\makebox(0,0)[lb]{\smash{$s'\!-\!1$}}}%
    \put(0.37151983,0.54817851){\color[rgb]{0,0,0}\makebox(0,0)[lb]{\smash{$s'$}}}%
    \put(0.44021585,0.54817851){\color[rgb]{0,0,0}\makebox(0,0)[lb]{\smash{$s'\!+\!1$}}}%
    \put(0.86112745,0.54817851){\color[rgb]{0,0,0}\makebox(0,0)[lb]{\smash{$s'\!-\!1$}}}%
    \put(0.77723327,0.54817851){\color[rgb]{0,0,0}\makebox(0,0)[lb]{\smash{$s\!-\!1$}}}%
    \put(0.83151027,0.6463737){\color[rgb]{0,0,0}\makebox(0,0)[lb]{\smash{$\mathcal{G}_5$}}}%
    \put(0.04238415,0.34007034){\color[rgb]{0,0,0}\makebox(0,0)[lb]{\smash{$s$}}}%
    \put(0.10779182,0.34007034){\color[rgb]{0,0,0}\makebox(0,0)[lb]{\smash{$s\!+\!1$}}}%
    \put(0.02439289,0.25637383){\color[rgb]{0,0,0}\makebox(0,0)[lb]{\smash{$s\!-\!1$}}}%
    \put(0.26991112,0.34007034){\color[rgb]{0,0,0}\makebox(0,0)[lb]{\smash{$s'\!-\!1$}}}%
    \put(0.37151983,0.34007034){\color[rgb]{0,0,0}\makebox(0,0)[lb]{\smash{$s'$}}}%
    \put(0.44021585,0.34007034){\color[rgb]{0,0,0}\makebox(0,0)[lb]{\smash{$s'\!+\!1$}}}%
    \put(0.86112745,0.34007034){\color[rgb]{0,0,0}\makebox(0,0)[lb]{\smash{$s'\!-\!1$}}}%
    \put(0.77723327,0.34007034){\color[rgb]{0,0,0}\makebox(0,0)[lb]{\smash{$s\!-\!1$}}}%
    \put(0.04238415,0.12723246){\color[rgb]{0,0,0}\makebox(0,0)[lb]{\smash{$s$}}}%
    \put(0.10779182,0.12723246){\color[rgb]{0,0,0}\makebox(0,0)[lb]{\smash{$s\!+\!1$}}}%
    \put(0.02439289,0.04353595){\color[rgb]{0,0,0}\makebox(0,0)[lb]{\smash{$s\!-\!1$}}}%
    \put(0.26991112,0.12723246){\color[rgb]{0,0,0}\makebox(0,0)[lb]{\smash{$s'\!-\!1$}}}%
    \put(0.37151983,0.12723246){\color[rgb]{0,0,0}\makebox(0,0)[lb]{\smash{$s'$}}}%
    \put(0.44021585,0.12723246){\color[rgb]{0,0,0}\makebox(0,0)[lb]{\smash{$s'\!+\!1$}}}%
    \put(0.86112745,0.12723246){\color[rgb]{0,0,0}\makebox(0,0)[lb]{\smash{$s'\!-\!1$}}}%
    \put(0.77723327,0.12723246){\color[rgb]{0,0,0}\makebox(0,0)[lb]{\smash{$s\!-\!1$}}}%
    \put(0.03512517,0.6463737){\color[rgb]{0,0,0}\makebox(0,0)[lb]{\smash{$\mathcal{G}_1$}}}%
    \put(0.33758508,0.6463737){\color[rgb]{0,0,0}\makebox(0,0)[lb]{\smash{$\mathcal{G}_2$}}}%
  \end{picture}%
\endgroup%

%% file: rho.pdf_tex
\begingroup%
  \makeatletter%
  \providecommand\color[2][]{%
    \errmessage{(Inkscape) Color is used for the text in Inkscape, but the package 'color.sty' is not loaded}%
    \renewcommand\color[2][]{}%
  }%
  \providecommand\transparent[1]{%
    \errmessage{(Inkscape) Transparency is used (non-zero) for the text in Inkscape, but the package 'transparent.sty' is not loaded}%
    \renewcommand\transparent[1]{}%
  }%
  \providecommand\rotatebox[2]{#2}%
  \ifx\svgwidth\undefined%
    \setlength{\unitlength}{388.864016bp}%
    \ifx\svgscale\undefined%
      \relax%
    \else%
      \setlength{\unitlength}{\unitlength * \real{\svgscale}}%
    \fi%
  \else%
    \setlength{\unitlength}{\svgwidth}%
  \fi%
  \global\let\svgwidth\undefined%
  \global\let\svgscale\undefined%
  \makeatother%
  \begin{picture}(1,0.41870005)%
    \put(0,0){\includegraphics[width=\unitlength]{rho.pdf}}%
    \put(0.041413,0.36760511){\color[rgb]{0,0,0}\makebox(0,0)[lb]{\smash{$i$}}}%
    \put(0.25948409,0.36760511){\color[rgb]{0,0,0}\makebox(0,0)[lb]{\smash{$i$}}}%
    \put(0.65859536,0.36760511){\color[rgb]{0,0,0}\makebox(0,0)[lb]{\smash{$i$}}}%
    \put(0.87666646,0.36760511){\color[rgb]{0,0,0}\makebox(0,0)[lb]{\smash{$i$}}}%
    \put(0.09761336,0.36760511){\color[rgb]{0,0,0}\makebox(0,0)[lb]{\smash{$i+1$}}}%
    \put(0.31660413,0.36760511){\color[rgb]{0,0,0}\makebox(0,0)[lb]{\smash{$i+1$}}}%
    \put(0.31660413,0.36760511){\color[rgb]{0,0,0}\makebox(0,0)[lb]{\smash{$i+1$}}}%
    \put(0.71571537,0.36760511){\color[rgb]{0,0,0}\makebox(0,0)[lb]{\smash{$i+1$}}}%
    \put(0.93378646,0.36760511){\color[rgb]{0,0,0}\makebox(0,0)[lb]{\smash{$i+1$}}}%
    \put(0.041413,0.25651228){\color[rgb]{0,0,0}\makebox(0,0)[lb]{\smash{$i$}}}%
    \put(0.25948409,0.25651228){\color[rgb]{0,0,0}\makebox(0,0)[lb]{\smash{$i$}}}%
    \put(0.65859536,0.25651228){\color[rgb]{0,0,0}\makebox(0,0)[lb]{\smash{$i$}}}%
    \put(0.87666646,0.25651228){\color[rgb]{0,0,0}\makebox(0,0)[lb]{\smash{$i$}}}%
    \put(0.09761336,0.25651228){\color[rgb]{0,0,0}\makebox(0,0)[lb]{\smash{$i+1$}}}%
    \put(0.31660413,0.25651228){\color[rgb]{0,0,0}\makebox(0,0)[lb]{\smash{$i+1$}}}%
    \put(0.31660413,0.25651228){\color[rgb]{0,0,0}\makebox(0,0)[lb]{\smash{$i+1$}}}%
    \put(0.71571537,0.25651228){\color[rgb]{0,0,0}\makebox(0,0)[lb]{\smash{$i+1$}}}%
    \put(0.93378646,0.25651228){\color[rgb]{0,0,0}\makebox(0,0)[lb]{\smash{$i+1$}}}%
    \put(0.041413,0.14953401){\color[rgb]{0,0,0}\makebox(0,0)[lb]{\smash{$i$}}}%
    \put(0.25948409,0.14953401){\color[rgb]{0,0,0}\makebox(0,0)[lb]{\smash{$i$}}}%
    \put(0.65859536,0.14953401){\color[rgb]{0,0,0}\makebox(0,0)[lb]{\smash{$i$}}}%
    \put(0.87666646,0.14953401){\color[rgb]{0,0,0}\makebox(0,0)[lb]{\smash{$i$}}}%
    \put(0.09761336,0.14953401){\color[rgb]{0,0,0}\makebox(0,0)[lb]{\smash{$i+1$}}}%
    \put(0.31660413,0.14953401){\color[rgb]{0,0,0}\makebox(0,0)[lb]{\smash{$i+1$}}}%
    \put(0.31660413,0.14953401){\color[rgb]{0,0,0}\makebox(0,0)[lb]{\smash{$i+1$}}}%
    \put(0.71571537,0.14953401){\color[rgb]{0,0,0}\makebox(0,0)[lb]{\smash{$i+1$}}}%
    \put(0.93378646,0.14953401){\color[rgb]{0,0,0}\makebox(0,0)[lb]{\smash{$i+1$}}}%
    \put(0.041413,0.03844119){\color[rgb]{0,0,0}\makebox(0,0)[lb]{\smash{$i$}}}%
    \put(0.25948409,0.03844119){\color[rgb]{0,0,0}\makebox(0,0)[lb]{\smash{$i$}}}%
    \put(0.65859536,0.03844119){\color[rgb]{0,0,0}\makebox(0,0)[lb]{\smash{$i$}}}%
    \put(0.87666646,0.03844119){\color[rgb]{0,0,0}\makebox(0,0)[lb]{\smash{$i$}}}%
    \put(0.09761336,0.03844119){\color[rgb]{0,0,0}\makebox(0,0)[lb]{\smash{$i+1$}}}%
    \put(0.31660413,0.03844119){\color[rgb]{0,0,0}\makebox(0,0)[lb]{\smash{$i+1$}}}%
    \put(0.31660413,0.03844119){\color[rgb]{0,0,0}\makebox(0,0)[lb]{\smash{$i+1$}}}%
    \put(0.71571537,0.03844119){\color[rgb]{0,0,0}\makebox(0,0)[lb]{\smash{$i+1$}}}%
    \put(0.93378646,0.03844119){\color[rgb]{0,0,0}\makebox(0,0)[lb]{\smash{$i+1$}}}%
  \end{picture}%
\endgroup%

%% file: rho2.pdf_tex
\begingroup%
  \makeatletter%
  \providecommand\color[2][]{%
    \errmessage{(Inkscape) Color is used for the text in Inkscape, but the package 'color.sty' is not loaded}%
    \renewcommand\color[2][]{}%
  }%
  \providecommand\transparent[1]{%
    \errmessage{(Inkscape) Transparency is used (non-zero) for the text in Inkscape, but the package 'transparent.sty' is not loaded}%
    \renewcommand\transparent[1]{}%
  }%
  \providecommand\rotatebox[2]{#2}%
  \ifx\svgwidth\undefined%
    \setlength{\unitlength}{389.050136bp}%
    \ifx\svgscale\undefined%
      \relax%
    \else%
      \setlength{\unitlength}{\unitlength * \real{\svgscale}}%
    \fi%
  \else%
    \setlength{\unitlength}{\svgwidth}%
  \fi%
  \global\let\svgwidth\undefined%
  \global\let\svgscale\undefined%
  \makeatother%
  \begin{picture}(1,0.09137358)%
    \put(0,0){\includegraphics[width=\unitlength]{rho2.pdf}}%
    \put(0.04139319,0.04030308){\color[rgb]{0,0,0}\makebox(0,0)[lb]{\smash{$i$}}}%
    \put(0.25935996,0.04030308){\color[rgb]{0,0,0}\makebox(0,0)[lb]{\smash{$i$}}}%
    \put(0.09756666,0.04030308){\color[rgb]{0,0,0}\makebox(0,0)[lb]{\smash{$i+1$}}}%
    \put(0.31645267,0.04030308){\color[rgb]{0,0,0}\makebox(0,0)[lb]{\smash{$i+1$}}}%
    \put(0.31645267,0.04030308){\color[rgb]{0,0,0}\makebox(0,0)[lb]{\smash{$i+1$}}}%
    \put(0.66239288,0.04030308){\color[rgb]{0,0,0}\makebox(0,0)[lb]{\smash{$i$}}}%
    \put(0.88035964,0.04030308){\color[rgb]{0,0,0}\makebox(0,0)[lb]{\smash{$i$}}}%
    \put(0.71856628,0.04030308){\color[rgb]{0,0,0}\makebox(0,0)[lb]{\smash{$i+1$}}}%
    \put(0.93745232,0.04030308){\color[rgb]{0,0,0}\makebox(0,0)[lb]{\smash{$i+1$}}}%
    \put(0.93745232,0.04030308){\color[rgb]{0,0,0}\makebox(0,0)[lb]{\smash{$i+1$}}}%
  \end{picture}%
\endgroup%

%% file: nu.pdf_tex
\begingroup%
  \makeatletter%
  \providecommand\color[2][]{%
    \errmessage{(Inkscape) Color is used for the text in Inkscape, but the package 'color.sty' is not loaded}%
    \renewcommand\color[2][]{}%
  }%
  \providecommand\transparent[1]{%
    \errmessage{(Inkscape) Transparency is used (non-zero) for the text in Inkscape, but the package 'transparent.sty' is not loaded}%
    \renewcommand\transparent[1]{}%
  }%
  \providecommand\rotatebox[2]{#2}%
  \ifx\svgwidth\undefined%
    \setlength{\unitlength}{410.58503459bp}%
    \ifx\svgscale\undefined%
      \relax%
    \else%
      \setlength{\unitlength}{\unitlength * \real{\svgscale}}%
    \fi%
  \else%
    \setlength{\unitlength}{\svgwidth}%
  \fi%
  \global\let\svgwidth\undefined%
  \global\let\svgscale\undefined%
  \makeatother%
  \begin{picture}(1,1.72281412)%
    \put(0,0){\includegraphics[width=\unitlength]{nu.pdf}}%
    \put(0.02459899,1.55556096){\color[rgb]{0,0,0}\makebox(0,0)[lb]{\smash{$s\!-\!1$}}}%
    \put(0.10679877,1.55556096){\color[rgb]{0,0,0}\makebox(0,0)[lb]{\smash{$s$}}}%
    \put(0.16464306,1.55556096){\color[rgb]{0,0,0}\makebox(0,0)[lb]{\smash{$s\!+\!1$}}}%
    \put(0.310776,1.55556096){\color[rgb]{0,0,0}\makebox(0,0)[lb]{\smash{$s'$}}}%
    \put(0.29579625,1.62367471){\color[rgb]{0,0,0}\makebox(0,0)[lb]{\smash{$s'\!+\!1$}}}%
    \put(0.29579625,1.48515285){\color[rgb]{0,0,0}\makebox(0,0)[lb]{\smash{$s'\!-\!1$}}}%
    \put(0.64271768,1.55556096){\color[rgb]{0,0,0}\makebox(0,0)[lb]{\smash{$s\!-\!1$}}}%
    \put(0.86548408,1.62308135){\color[rgb]{0,0,0}\makebox(0,0)[lb]{\smash{$s$}}}%
    \put(0.70886537,1.62254317){\color[rgb]{0,0,0}\makebox(0,0)[lb]{\smash{$s'\!+\!1$}}}%
    \put(0.72384513,1.55556096){\color[rgb]{0,0,0}\makebox(0,0)[lb]{\smash{$s'$}}}%
    \put(0.84807656,1.55556096){\color[rgb]{0,0,0}\makebox(0,0)[lb]{\smash{$s'\!-\!1$}}}%
    \put(0.91958764,1.62254317){\color[rgb]{0,0,0}\makebox(0,0)[lb]{\smash{$s\!+\!1$}}}%
    \put(0.02459899,1.55556096){\color[rgb]{0,0,0}\makebox(0,0)[lb]{\smash{$s\!-\!1$}}}%
    \put(0.10679877,1.55556096){\color[rgb]{0,0,0}\makebox(0,0)[lb]{\smash{$s$}}}%
    \put(0.16464306,1.55556096){\color[rgb]{0,0,0}\makebox(0,0)[lb]{\smash{$s\!+\!1$}}}%
    \put(0.310776,1.55556096){\color[rgb]{0,0,0}\makebox(0,0)[lb]{\smash{$s'$}}}%
    \put(0.29579625,1.62367471){\color[rgb]{0,0,0}\makebox(0,0)[lb]{\smash{$s'\!+\!1$}}}%
    \put(0.29579625,1.48515285){\color[rgb]{0,0,0}\makebox(0,0)[lb]{\smash{$s'\!-\!1$}}}%
    \put(0.64271768,1.55556096){\color[rgb]{0,0,0}\makebox(0,0)[lb]{\smash{$s\!-\!1$}}}%
    \put(0.86548408,1.62308135){\color[rgb]{0,0,0}\makebox(0,0)[lb]{\smash{$s$}}}%
    \put(0.70886537,1.62254317){\color[rgb]{0,0,0}\makebox(0,0)[lb]{\smash{$s'\!+\!1$}}}%
    \put(0.72384513,1.55556096){\color[rgb]{0,0,0}\makebox(0,0)[lb]{\smash{$s'$}}}%
    \put(0.84807656,1.55556096){\color[rgb]{0,0,0}\makebox(0,0)[lb]{\smash{$s'\!-\!1$}}}%
    \put(0.91958764,1.62254317){\color[rgb]{0,0,0}\makebox(0,0)[lb]{\smash{$s\!+\!1$}}}%
    \put(0.02459899,1.31395449){\color[rgb]{0,0,0}\makebox(0,0)[lb]{\smash{$s\!-\!1$}}}%
    \put(0.10679877,1.31395449){\color[rgb]{0,0,0}\makebox(0,0)[lb]{\smash{$s$}}}%
    \put(0.16464306,1.31395449){\color[rgb]{0,0,0}\makebox(0,0)[lb]{\smash{$s\!+\!1$}}}%
    \put(0.310776,1.31395449){\color[rgb]{0,0,0}\makebox(0,0)[lb]{\smash{$s'$}}}%
    \put(0.29579625,1.38206824){\color[rgb]{0,0,0}\makebox(0,0)[lb]{\smash{$s'\!+\!1$}}}%
    \put(0.29579625,1.24354639){\color[rgb]{0,0,0}\makebox(0,0)[lb]{\smash{$s'\!-\!1$}}}%
    \put(0.64271768,1.31395449){\color[rgb]{0,0,0}\makebox(0,0)[lb]{\smash{$s\!-\!1$}}}%
    \put(0.86548408,1.38147489){\color[rgb]{0,0,0}\makebox(0,0)[lb]{\smash{$s$}}}%
    \put(0.70886537,1.3809367){\color[rgb]{0,0,0}\makebox(0,0)[lb]{\smash{$s'\!+\!1$}}}%
    \put(0.72384513,1.31395449){\color[rgb]{0,0,0}\makebox(0,0)[lb]{\smash{$s'$}}}%
    \put(0.84807656,1.31395449){\color[rgb]{0,0,0}\makebox(0,0)[lb]{\smash{$s'\!-\!1$}}}%
    \put(0.91958764,1.3809367){\color[rgb]{0,0,0}\makebox(0,0)[lb]{\smash{$s\!+\!1$}}}%
    \put(0.02459899,1.07234802){\color[rgb]{0,0,0}\makebox(0,0)[lb]{\smash{$s\!-\!1$}}}%
    \put(0.10679877,1.07234802){\color[rgb]{0,0,0}\makebox(0,0)[lb]{\smash{$s$}}}%
    \put(0.16464306,1.07234802){\color[rgb]{0,0,0}\makebox(0,0)[lb]{\smash{$s\!+\!1$}}}%
    \put(0.310776,1.07234802){\color[rgb]{0,0,0}\makebox(0,0)[lb]{\smash{$s'$}}}%
    \put(0.29579625,1.14046177){\color[rgb]{0,0,0}\makebox(0,0)[lb]{\smash{$s'\!+\!1$}}}%
    \put(0.29579625,1.00193992){\color[rgb]{0,0,0}\makebox(0,0)[lb]{\smash{$s'\!-\!1$}}}%
    \put(0.64271768,1.07234802){\color[rgb]{0,0,0}\makebox(0,0)[lb]{\smash{$s\!-\!1$}}}%
    \put(0.86548408,1.13986842){\color[rgb]{0,0,0}\makebox(0,0)[lb]{\smash{$s$}}}%
    \put(0.70886537,1.13933023){\color[rgb]{0,0,0}\makebox(0,0)[lb]{\smash{$s'\!+\!1$}}}%
    \put(0.72384513,1.07234802){\color[rgb]{0,0,0}\makebox(0,0)[lb]{\smash{$s'$}}}%
    \put(0.84807656,1.07234802){\color[rgb]{0,0,0}\makebox(0,0)[lb]{\smash{$s'\!-\!1$}}}%
    \put(0.91958764,1.13933023){\color[rgb]{0,0,0}\makebox(0,0)[lb]{\smash{$s\!+\!1$}}}%
    \put(0.02459899,0.83074156){\color[rgb]{0,0,0}\makebox(0,0)[lb]{\smash{$s\!-\!1$}}}%
    \put(0.10679877,0.83074156){\color[rgb]{0,0,0}\makebox(0,0)[lb]{\smash{$s$}}}%
    \put(0.16464306,0.83074156){\color[rgb]{0,0,0}\makebox(0,0)[lb]{\smash{$s\!+\!1$}}}%
    \put(0.310776,0.83074156){\color[rgb]{0,0,0}\makebox(0,0)[lb]{\smash{$s'$}}}%
    \put(0.29579625,0.8988553){\color[rgb]{0,0,0}\makebox(0,0)[lb]{\smash{$s'\!+\!1$}}}%
    \put(0.29579625,0.76033345){\color[rgb]{0,0,0}\makebox(0,0)[lb]{\smash{$s'\!-\!1$}}}%
    \put(0.64271768,0.83074156){\color[rgb]{0,0,0}\makebox(0,0)[lb]{\smash{$s\!-\!1$}}}%
    \put(0.86548408,0.89826195){\color[rgb]{0,0,0}\makebox(0,0)[lb]{\smash{$s$}}}%
    \put(0.70886537,0.89772376){\color[rgb]{0,0,0}\makebox(0,0)[lb]{\smash{$s'\!+\!1$}}}%
    \put(0.72384513,0.83074156){\color[rgb]{0,0,0}\makebox(0,0)[lb]{\smash{$s'$}}}%
    \put(0.84807656,0.83074156){\color[rgb]{0,0,0}\makebox(0,0)[lb]{\smash{$s'\!-\!1$}}}%
    \put(0.91958764,0.89772376){\color[rgb]{0,0,0}\makebox(0,0)[lb]{\smash{$s\!+\!1$}}}%
    \put(0.02459899,0.58913509){\color[rgb]{0,0,0}\makebox(0,0)[lb]{\smash{$s\!-\!1$}}}%
    \put(0.10679877,0.58913509){\color[rgb]{0,0,0}\makebox(0,0)[lb]{\smash{$s$}}}%
    \put(0.16464306,0.58913509){\color[rgb]{0,0,0}\makebox(0,0)[lb]{\smash{$s\!+\!1$}}}%
    \put(0.310776,0.58913509){\color[rgb]{0,0,0}\makebox(0,0)[lb]{\smash{$s'$}}}%
    \put(0.29579625,0.65724884){\color[rgb]{0,0,0}\makebox(0,0)[lb]{\smash{$s'\!+\!1$}}}%
    \put(0.29579625,0.51872698){\color[rgb]{0,0,0}\makebox(0,0)[lb]{\smash{$s'\!-\!1$}}}%
    \put(0.64271768,0.58913509){\color[rgb]{0,0,0}\makebox(0,0)[lb]{\smash{$s\!-\!1$}}}%
    \put(0.86548408,0.65665548){\color[rgb]{0,0,0}\makebox(0,0)[lb]{\smash{$s$}}}%
    \put(0.70886537,0.6561173){\color[rgb]{0,0,0}\makebox(0,0)[lb]{\smash{$s'\!+\!1$}}}%
    \put(0.72384513,0.58913509){\color[rgb]{0,0,0}\makebox(0,0)[lb]{\smash{$s'$}}}%
    \put(0.84807656,0.58913509){\color[rgb]{0,0,0}\makebox(0,0)[lb]{\smash{$s'\!-\!1$}}}%
    \put(0.91958764,0.6561173){\color[rgb]{0,0,0}\makebox(0,0)[lb]{\smash{$s\!+\!1$}}}%
    \put(0.02459899,0.34752862){\color[rgb]{0,0,0}\makebox(0,0)[lb]{\smash{$s\!-\!1$}}}%
    \put(0.10679877,0.34752862){\color[rgb]{0,0,0}\makebox(0,0)[lb]{\smash{$s$}}}%
    \put(0.16464306,0.34752862){\color[rgb]{0,0,0}\makebox(0,0)[lb]{\smash{$s\!+\!1$}}}%
    \put(0.310776,0.34752862){\color[rgb]{0,0,0}\makebox(0,0)[lb]{\smash{$s'$}}}%
    \put(0.29579625,0.41564237){\color[rgb]{0,0,0}\makebox(0,0)[lb]{\smash{$s'\!+\!1$}}}%
    \put(0.29579625,0.27712051){\color[rgb]{0,0,0}\makebox(0,0)[lb]{\smash{$s'\!-\!1$}}}%
    \put(0.64271768,0.34752862){\color[rgb]{0,0,0}\makebox(0,0)[lb]{\smash{$s\!-\!1$}}}%
    \put(0.86548408,0.41504901){\color[rgb]{0,0,0}\makebox(0,0)[lb]{\smash{$s$}}}%
    \put(0.70886537,0.41451083){\color[rgb]{0,0,0}\makebox(0,0)[lb]{\smash{$s'\!+\!1$}}}%
    \put(0.72384513,0.34752862){\color[rgb]{0,0,0}\makebox(0,0)[lb]{\smash{$s'$}}}%
    \put(0.84807656,0.34752862){\color[rgb]{0,0,0}\makebox(0,0)[lb]{\smash{$s'\!-\!1$}}}%
    \put(0.91958764,0.41451083){\color[rgb]{0,0,0}\makebox(0,0)[lb]{\smash{$s\!+\!1$}}}%
    \put(0.02459899,0.10202527){\color[rgb]{0,0,0}\makebox(0,0)[lb]{\smash{$s\!-\!1$}}}%
    \put(0.10679877,0.10202527){\color[rgb]{0,0,0}\makebox(0,0)[lb]{\smash{$s$}}}%
    \put(0.16464306,0.10202527){\color[rgb]{0,0,0}\makebox(0,0)[lb]{\smash{$s\!+\!1$}}}%
    \put(0.310776,0.10202527){\color[rgb]{0,0,0}\makebox(0,0)[lb]{\smash{$s'$}}}%
    \put(0.29579625,0.17013902){\color[rgb]{0,0,0}\makebox(0,0)[lb]{\smash{$s'\!+\!1$}}}%
    \put(0.29579625,0.03161717){\color[rgb]{0,0,0}\makebox(0,0)[lb]{\smash{$s'\!-\!1$}}}%
    \put(0.64271768,0.10202527){\color[rgb]{0,0,0}\makebox(0,0)[lb]{\smash{$s\!-\!1$}}}%
    \put(0.86548408,0.16954567){\color[rgb]{0,0,0}\makebox(0,0)[lb]{\smash{$s$}}}%
    \put(0.70886537,0.16900748){\color[rgb]{0,0,0}\makebox(0,0)[lb]{\smash{$s'\!+\!1$}}}%
    \put(0.72384513,0.10202527){\color[rgb]{0,0,0}\makebox(0,0)[lb]{\smash{$s'$}}}%
    \put(0.84807656,0.10202527){\color[rgb]{0,0,0}\makebox(0,0)[lb]{\smash{$s'\!-\!1$}}}%
    \put(0.91958764,0.16900748){\color[rgb]{0,0,0}\makebox(0,0)[lb]{\smash{$s\!+\!1$}}}%
    \put(0.06625975,1.70785452){\color[rgb]{0,0,0}\makebox(0,0)[lb]{\smash{$\mathcal{G}_1$}}}%
    \put(0.31223035,1.70785452){\color[rgb]{0,0,0}\makebox(0,0)[lb]{\smash{$\mathcal{G}_2$}}}%
    \put(0.67927279,1.70785452){\color[rgb]{0,0,0}\makebox(0,0)[lb]{\smash{$\mathcal{G}_3$}}}%
    \put(0.86117941,1.70785452){\color[rgb]{0,0,0}\makebox(0,0)[lb]{\smash{$\mathcal{G}_4$}}}%
  \end{picture}%
\endgroup%

%% file: sigma.pdf_tex
\begingroup%
  \makeatletter%
  \providecommand\color[2][]{%
    \errmessage{(Inkscape) Color is used for the text in Inkscape, but the package 'color.sty' is not loaded}%
    \renewcommand\color[2][]{}%
  }%
  \providecommand\transparent[1]{%
    \errmessage{(Inkscape) Transparency is used (non-zero) for the text in Inkscape, but the package 'transparent.sty' is not loaded}%
    \renewcommand\transparent[1]{}%
  }%
  \providecommand\rotatebox[2]{#2}%
  \ifx\svgwidth\undefined%
    \setlength{\unitlength}{422.96554565bp}%
    \ifx\svgscale\undefined%
      \relax%
    \else%
      \setlength{\unitlength}{\unitlength * \real{\svgscale}}%
    \fi%
  \else%
    \setlength{\unitlength}{\svgwidth}%
  \fi%
  \global\let\svgwidth\undefined%
  \global\let\svgscale\undefined%
  \makeatother%
  \begin{picture}(1,1.0597087)%
    \put(0,0){\includegraphics[width=\unitlength]{sigma.pdf}}%
    \put(0.04055441,1.04518698){\color[rgb]{0,0,0}\makebox(0,0)[lb]{\smash{$\mathcal{G}_1$}}}%
    \put(0.16455476,1.04269801){\color[rgb]{0,0,0}\makebox(0,0)[lb]{\smash{$\mathcal{G}_2$}}}%
    \put(0.60233375,1.04334772){\color[rgb]{0,0,0}\makebox(0,0)[lb]{\smash{$\mathcal{G}_3$}}}%
    \put(0.86524279,1.03782995){\color[rgb]{0,0,0}\makebox(0,0)[lb]{\smash{$\mathcal{G}_4$}}}%
    \put(0.047456,0.89252881){\color[rgb]{0,0,0}\makebox(0,0)[lb]{\smash{$s$}}}%
    \put(0.03989037,0.96058125){\color[rgb]{0,0,0}\makebox(0,0)[lb]{\smash{$s\!+\!1$}}}%
    \put(-0.00005172,0.96242051){\color[rgb]{0,0,0}\makebox(0,0)[lb]{\smash{$\delta_5$}}}%
    \put(0.09517231,0.89079386){\color[rgb]{0,0,0}\makebox(0,0)[lb]{\smash{$\delta_3$}}}%
    \put(0.1376838,0.88968475){\color[rgb]{0,0,0}\makebox(0,0)[lb]{\smash{$\delta_3'$}}}%
    \put(0.18397808,0.84102967){\color[rgb]{0,0,0}\makebox(0,0)[lb]{\smash{$\delta_5'$}}}%
    \put(0.58656629,0.9255311){\color[rgb]{0,0,0}\makebox(0,0)[lb]{\smash{$s$}}}%
    \put(0.63595616,0.86008966){\color[rgb]{0,0,0}\makebox(0,0)[lb]{\smash{$\delta_3$}}}%
    \put(0.65277945,0.92383437){\color[rgb]{0,0,0}\makebox(0,0)[lb]{\smash{1}}}%
    \put(0.90791443,0.92369185){\color[rgb]{0,0,0}\makebox(0,0)[lb]{\smash{$s\!+\!1$}}}%
    \put(0.85304966,0.92358754){\color[rgb]{0,0,0}\makebox(0,0)[lb]{\smash{1}}}%
    \put(0.18945761,0.89242452){\color[rgb]{0,0,0}\makebox(0,0)[lb]{\smash{1}}}%
    \put(0.90453589,0.86008966){\color[rgb]{0,0,0}\makebox(0,0)[lb]{\smash{$\delta_5$}}}%
    \put(0.047456,0.72230218){\color[rgb]{0,0,0}\makebox(0,0)[lb]{\smash{$s$}}}%
    \put(0.03989037,0.79035462){\color[rgb]{0,0,0}\makebox(0,0)[lb]{\smash{$s\!+\!1$}}}%
    \put(-0.00005172,0.79219388){\color[rgb]{0,0,0}\makebox(0,0)[lb]{\smash{$\delta_5$}}}%
    \put(0.09517231,0.72056723){\color[rgb]{0,0,0}\makebox(0,0)[lb]{\smash{$\delta_3$}}}%
    \put(0.1376838,0.71945811){\color[rgb]{0,0,0}\makebox(0,0)[lb]{\smash{$\delta_3'$}}}%
    \put(0.18397808,0.67080301){\color[rgb]{0,0,0}\makebox(0,0)[lb]{\smash{$\delta_5'$}}}%
    \put(0.58656629,0.75530448){\color[rgb]{0,0,0}\makebox(0,0)[lb]{\smash{$s$}}}%
    \put(0.63595616,0.689863){\color[rgb]{0,0,0}\makebox(0,0)[lb]{\smash{$\delta_3$}}}%
    \put(0.65277945,0.75360774){\color[rgb]{0,0,0}\makebox(0,0)[lb]{\smash{1}}}%
    \put(0.90791443,0.7534652){\color[rgb]{0,0,0}\makebox(0,0)[lb]{\smash{$s\!+\!1$}}}%
    \put(0.85304966,0.75336093){\color[rgb]{0,0,0}\makebox(0,0)[lb]{\smash{1}}}%
    \put(0.18945761,0.72219788){\color[rgb]{0,0,0}\makebox(0,0)[lb]{\smash{1}}}%
    \put(0.90453589,0.689863){\color[rgb]{0,0,0}\makebox(0,0)[lb]{\smash{$\delta_5$}}}%
    \put(0.047456,0.55585836){\color[rgb]{0,0,0}\makebox(0,0)[lb]{\smash{$s$}}}%
    \put(0.03989037,0.6239108){\color[rgb]{0,0,0}\makebox(0,0)[lb]{\smash{$s\!+\!1$}}}%
    \put(-0.00005172,0.62575006){\color[rgb]{0,0,0}\makebox(0,0)[lb]{\smash{$\delta_5$}}}%
    \put(0.09517231,0.55412341){\color[rgb]{0,0,0}\makebox(0,0)[lb]{\smash{$\delta_3$}}}%
    \put(0.13542243,0.55342547){\color[rgb]{0,0,0}\makebox(0,0)[lb]{\smash{$\delta_3'$}}}%
    \put(0.18397808,0.5043592){\color[rgb]{0,0,0}\makebox(0,0)[lb]{\smash{$\delta_5'$}}}%
    \put(0.58656629,0.58886067){\color[rgb]{0,0,0}\makebox(0,0)[lb]{\smash{$s$}}}%
    \put(0.63595616,0.52341919){\color[rgb]{0,0,0}\makebox(0,0)[lb]{\smash{$\delta_3$}}}%
    \put(0.65277945,0.58716392){\color[rgb]{0,0,0}\makebox(0,0)[lb]{\smash{1}}}%
    \put(0.90791443,0.58702138){\color[rgb]{0,0,0}\makebox(0,0)[lb]{\smash{$s\!+\!1$}}}%
    \put(0.85304966,0.58691711){\color[rgb]{0,0,0}\makebox(0,0)[lb]{\smash{1}}}%
    \put(0.18945761,0.55575406){\color[rgb]{0,0,0}\makebox(0,0)[lb]{\smash{1}}}%
    \put(0.90453589,0.52341919){\color[rgb]{0,0,0}\makebox(0,0)[lb]{\smash{$\delta_5$}}}%
    \put(0.047456,0.38563173){\color[rgb]{0,0,0}\makebox(0,0)[lb]{\smash{$s$}}}%
    \put(0.03989037,0.45368417){\color[rgb]{0,0,0}\makebox(0,0)[lb]{\smash{$s\!+\!1$}}}%
    \put(-0.00005172,0.45552343){\color[rgb]{0,0,0}\makebox(0,0)[lb]{\smash{$\delta_5$}}}%
    \put(0.09517231,0.38389678){\color[rgb]{0,0,0}\makebox(0,0)[lb]{\smash{$\delta_3$}}}%
    \put(0.1376838,0.38278767){\color[rgb]{0,0,0}\makebox(0,0)[lb]{\smash{$\delta_3'$}}}%
    \put(0.18397808,0.33413256){\color[rgb]{0,0,0}\makebox(0,0)[lb]{\smash{$\delta_5'$}}}%
    \put(0.58656629,0.41863404){\color[rgb]{0,0,0}\makebox(0,0)[lb]{\smash{$s$}}}%
    \put(0.63595616,0.35319255){\color[rgb]{0,0,0}\makebox(0,0)[lb]{\smash{$\delta_3$}}}%
    \put(0.65277945,0.41693729){\color[rgb]{0,0,0}\makebox(0,0)[lb]{\smash{1}}}%
    \put(0.90791443,0.41679475){\color[rgb]{0,0,0}\makebox(0,0)[lb]{\smash{$s\!+\!1$}}}%
    \put(0.85304966,0.41669048){\color[rgb]{0,0,0}\makebox(0,0)[lb]{\smash{1}}}%
    \put(0.18945761,0.38552743){\color[rgb]{0,0,0}\makebox(0,0)[lb]{\smash{1}}}%
    \put(0.90453589,0.35319255){\color[rgb]{0,0,0}\makebox(0,0)[lb]{\smash{$\delta_5$}}}%
    \put(0.047456,0.22297073){\color[rgb]{0,0,0}\makebox(0,0)[lb]{\smash{$s$}}}%
    \put(0.03989037,0.29102317){\color[rgb]{0,0,0}\makebox(0,0)[lb]{\smash{$s\!+\!1$}}}%
    \put(-0.00005172,0.29286242){\color[rgb]{0,0,0}\makebox(0,0)[lb]{\smash{$\delta_5$}}}%
    \put(0.09517231,0.22123577){\color[rgb]{0,0,0}\makebox(0,0)[lb]{\smash{$\delta_3$}}}%
    \put(0.1376838,0.22012666){\color[rgb]{0,0,0}\makebox(0,0)[lb]{\smash{$\delta_3'$}}}%
    \put(0.18397808,0.17147156){\color[rgb]{0,0,0}\makebox(0,0)[lb]{\smash{$\delta_5'$}}}%
    \put(0.58656629,0.25597303){\color[rgb]{0,0,0}\makebox(0,0)[lb]{\smash{$s$}}}%
    \put(0.63595616,0.19053155){\color[rgb]{0,0,0}\makebox(0,0)[lb]{\smash{$\delta_3$}}}%
    \put(0.65277945,0.25427629){\color[rgb]{0,0,0}\makebox(0,0)[lb]{\smash{1}}}%
    \put(0.90791443,0.25413375){\color[rgb]{0,0,0}\makebox(0,0)[lb]{\smash{$s\!+\!1$}}}%
    \put(0.85304966,0.25402947){\color[rgb]{0,0,0}\makebox(0,0)[lb]{\smash{1}}}%
    \put(0.18945761,0.22286643){\color[rgb]{0,0,0}\makebox(0,0)[lb]{\smash{1}}}%
    \put(0.90453589,0.19053155){\color[rgb]{0,0,0}\makebox(0,0)[lb]{\smash{$\delta_5$}}}%
    \put(0.047456,0.05652691){\color[rgb]{0,0,0}\makebox(0,0)[lb]{\smash{$s$}}}%
    \put(0.03989037,0.12457935){\color[rgb]{0,0,0}\makebox(0,0)[lb]{\smash{$s\!+\!1$}}}%
    \put(-0.00005172,0.12641858){\color[rgb]{0,0,0}\makebox(0,0)[lb]{\smash{$\delta_5$}}}%
    \put(0.09517231,0.05479193){\color[rgb]{0,0,0}\makebox(0,0)[lb]{\smash{$\delta_3$}}}%
    \put(0.1376838,0.05368287){\color[rgb]{0,0,0}\makebox(0,0)[lb]{\smash{$\delta_3'$}}}%
    \put(0.18397808,0.00502774){\color[rgb]{0,0,0}\makebox(0,0)[lb]{\smash{$\delta_5'$}}}%
    \put(0.58656629,0.08952922){\color[rgb]{0,0,0}\makebox(0,0)[lb]{\smash{$s$}}}%
    \put(0.63595616,0.02408773){\color[rgb]{0,0,0}\makebox(0,0)[lb]{\smash{$\delta_3$}}}%
    \put(0.65277945,0.08783244){\color[rgb]{0,0,0}\makebox(0,0)[lb]{\smash{1}}}%
    \put(0.90791443,0.08768987){\color[rgb]{0,0,0}\makebox(0,0)[lb]{\smash{$s\!+\!1$}}}%
    \put(0.85304966,0.08758563){\color[rgb]{0,0,0}\makebox(0,0)[lb]{\smash{1}}}%
    \put(0.18945761,0.05642255){\color[rgb]{0,0,0}\makebox(0,0)[lb]{\smash{1}}}%
    \put(0.90453589,0.02408773){\color[rgb]{0,0,0}\makebox(0,0)[lb]{\smash{$\delta_5$}}}%
    \put(0.18945761,0.05642255){\color[rgb]{0,0,0}\makebox(0,0)[lb]{\smash{1}}}%
  \end{picture}%
\endgroup%

%% file: figgluenew.pdf_t
\begin{picture}(0,0)%
\includegraphics{figgluenew.pdf}%
\end{picture}%
\setlength{\unitlength}{3947sp}%
\begingroup\makeatletter\ifx\SetFigFont\undefined%
\gdef\SetFigFont#1#2#3#4#5{%
  \reset@font\fontsize{#1}{#2pt}%
  \fontfamily{#3}\fontseries{#4}\fontshape{#5}%
  \selectfont}%
\fi\endgroup%
\begin{picture}(2214,1611)(1189,-2314)
\put(2626,-2236){\makebox(0,0)[lb]{\smash{{\SetFigFont{12}{14.4}{\rmdefault}{\mddefault}{\updefault}{\color[rgb]{0,0,0}$\tau_{i_j}$}%
}}}}
\put(2858,-1521){\makebox(0,0)[lb]{\smash{{\SetFigFont{12}{14.4}{\rmdefault}{\mddefault}{\updefault}{\color[rgb]{0,0,0}$\tau_{i_{j+1}}$}%
}}}}
\put(2351,-1647){\makebox(0,0)[lb]{\smash{{\SetFigFont{12}{14.4}{\rmdefault}{\mddefault}{\updefault}{\color[rgb]{0,0,0}$\tau_{[j]}$}%
}}}}
\put(1800,-1563){\makebox(0,0)[lb]{\smash{{\SetFigFont{12}{14.4}{\rmdefault}{\mddefault}{\updefault}{\color[rgb]{0,0,0}$\tau_{i_j}$}%
}}}}
\put(1576,-886){\makebox(0,0)[lb]{\smash{{\SetFigFont{12}{14.4}{\rmdefault}{\mddefault}{\updefault}{\color[rgb]{0,0,0}$\tau_{i_{j+1}}$}%
}}}}
\end{picture}%

%% file: figsnake.pdf_t
\begin{picture}(0,0)%
\includegraphics{figsnake.pdf}%
\end{picture}%
\setlength{\unitlength}{3947sp}%
\begingroup\makeatletter\ifx\SetFigFont\undefined%
\gdef\SetFigFont#1#2#3#4#5{%
  \reset@font\fontsize{#1}{#2pt}%
  \fontfamily{#3}\fontseries{#4}\fontshape{#5}%
  \selectfont}%
\fi\endgroup%
\begin{picture}(5300,2249)(1313,-3015)
\put(5637,-1621){\makebox(0,0)[lb]{\smash{{\SetFigFont{12}{14.4}{\rmdefault}{\mddefault}{\updefault}{\color[rgb]{0,0,0}$6$}%
}}}}
\put(5040,-2523){\makebox(0,0)[lb]{\smash{{\SetFigFont{12}{14.4}{\rmdefault}{\mddefault}{\updefault}{\color[rgb]{0,0,0}$1$}%
}}}}
\put(6033,-2523){\makebox(0,0)[lb]{\smash{{\SetFigFont{12}{14.4}{\rmdefault}{\mddefault}{\updefault}{\color[rgb]{0,0,0}$3$}%
}}}}
\put(6238,-1340){\makebox(0,0)[lb]{\smash{{\SetFigFont{12}{14.4}{\rmdefault}{\mddefault}{\updefault}{\color[rgb]{0,0,0}$1$}%
}}}}
\put(5635,-936){\makebox(0,0)[lb]{\smash{{\SetFigFont{12}{14.4}{\rmdefault}{\mddefault}{\updefault}{\color[rgb]{0,0,0}$1$}%
}}}}
\put(5635,-1340){\makebox(0,0)[lb]{\smash{{\SetFigFont{12}{14.4}{\rmdefault}{\mddefault}{\updefault}{\color[rgb]{0,0,0}$4$}%
}}}}
\put(5635,-1935){\makebox(0,0)[lb]{\smash{{\SetFigFont{12}{14.4}{\rmdefault}{\mddefault}{\updefault}{\color[rgb]{0,0,0}$3$}%
}}}}
\put(5287,-1340){\makebox(0,0)[lb]{\smash{{\SetFigFont{12}{14.4}{\rmdefault}{\mddefault}{\updefault}{\color[rgb]{0,0,0}$3$}%
}}}}
\put(5287,-1935){\makebox(0,0)[lb]{\smash{{\SetFigFont{12}{14.4}{\rmdefault}{\mddefault}{\updefault}{\color[rgb]{0,0,0}$2$}%
}}}}
\put(5040,-2134){\makebox(0,0)[lb]{\smash{{\SetFigFont{12}{14.4}{\rmdefault}{\mddefault}{\updefault}{\color[rgb]{0,0,0}$2$}%
}}}}
\put(5040,-2946){\makebox(0,0)[lb]{\smash{{\SetFigFont{12}{14.4}{\rmdefault}{\mddefault}{\updefault}{\color[rgb]{0,0,0}$4$}%
}}}}
\put(6033,-1935){\makebox(0,0)[lb]{\smash{{\SetFigFont{12}{14.4}{\rmdefault}{\mddefault}{\updefault}{\color[rgb]{0,0,0}$4$}%
}}}}
\put(6238,-1717){\makebox(0,0)[lb]{\smash{{\SetFigFont{12}{14.4}{\rmdefault}{\mddefault}{\updefault}{\color[rgb]{0,0,0}$4$}%
}}}}
\put(5635,-2939){\makebox(0,0)[lb]{\smash{{\SetFigFont{12}{14.4}{\rmdefault}{\mddefault}{\updefault}{\color[rgb]{0,0,0}$1$}%
}}}}
\put(1354,-1882){\makebox(0,0)[lb]{\smash{{\SetFigFont{12}{14.4}{\rmdefault}{\mddefault}{\updefault}{\color[rgb]{0,0,0}$5$}%
}}}}
\put(3324,-1888){\makebox(0,0)[lb]{\smash{{\SetFigFont{12}{14.4}{\rmdefault}{\mddefault}{\updefault}{\color[rgb]{0,0,0}$6$}%
}}}}
\put(1771,-1762){\makebox(0,0)[lb]{\smash{{\SetFigFont{12}{14.4}{\rmdefault}{\mddefault}{\updefault}{\color[rgb]{0,0,0}$1$}%
}}}}
\put(2722,-1568){\makebox(0,0)[lb]{\smash{{\SetFigFont{12}{14.4}{\rmdefault}{\mddefault}{\updefault}{\color[rgb]{0,0,0}$3$}%
}}}}
\put(2622,-2682){\makebox(0,0)[lb]{\smash{{\SetFigFont{12}{14.4}{\rmdefault}{\mddefault}{\updefault}{\color[rgb]{0,0,0}$4$}%
}}}}
\put(2157,-1522){\makebox(0,0)[lb]{\smash{{\SetFigFont{12}{14.4}{\rmdefault}{\mddefault}{\updefault}{\color[rgb]{0,0,0}$2$}%
}}}}
\put(1760,-2443){\makebox(0,0)[lb]{\smash{{\SetFigFont{12}{14.4}{\rmdefault}{\mddefault}{\updefault}{\color[rgb]{0,0,0}$\gamma$}%
}}}}
\put(4665,-2524){\makebox(0,0)[lb]{\smash{{\SetFigFont{12}{14.4}{\rmdefault}{\mddefault}{\updefault}{\color[rgb]{0,0,0}$5$}%
}}}}
\put(6019,-1336){\makebox(0,0)[lb]{\smash{{\SetFigFont{12}{14.4}{\rmdefault}{\mddefault}{\updefault}{\color[rgb]{0,0,0}$5$}%
}}}}
\put(5635,-2523){\makebox(0,0)[lb]{\smash{{\SetFigFont{12}{14.4}{\rmdefault}{\mddefault}{\updefault}{\color[rgb]{0,0,0}$2$}%
}}}}
\end{picture}%

%% file: figlemwo.pdf_tex
\begingroup%
  \makeatletter%
  \providecommand\color[2][]{%
    \errmessage{(Inkscape) Color is used for the text in Inkscape, but the package 'color.sty' is not loaded}%
    \renewcommand\color[2][]{}%
  }%
  \providecommand\transparent[1]{%
    \errmessage{(Inkscape) Transparency is used (non-zero) for the text in Inkscape, but the package 'transparent.sty' is not loaded}%
    \renewcommand\transparent[1]{}%
  }%
  \providecommand\rotatebox[2]{#2}%
  \ifx\svgwidth\undefined%
    \setlength{\unitlength}{357.475bp}%
    \ifx\svgscale\undefined%
      \relax%
    \else%
      \setlength{\unitlength}{\unitlength * \real{\svgscale}}%
    \fi%
  \else%
    \setlength{\unitlength}{\svgwidth}%
  \fi%
  \global\let\svgwidth\undefined%
  \global\let\svgscale\undefined%
  \makeatother%
  \begin{picture}(1,0.38629544)%
    \put(0,0){\includegraphics[width=\unitlength]{figlemwo.pdf}}%
    \put(0.19250796,0.06672356){\color[rgb]{0,0,0}\makebox(0,0)[lb]{\smash{$v$}}}%
    \put(0.2618773,0.19218906){\color[rgb]{0,0,0}\makebox(0,0)[lb]{\smash{$p_2$}}}%
    \put(0.3047002,0.27822262){\color[rgb]{0,0,0}\makebox(0,0)[lb]{\smash{$\beta_2$}}}%
    \put(0.06478183,0.27922366){\color[rgb]{0,0,0}\makebox(0,0)[lb]{\smash{$\beta_1$}}}%
    \put(0.10886428,0.19541872){\color[rgb]{0,0,0}\makebox(0,0)[lb]{\smash{$p_1$}}}%
    \put(0.00006302,0.25490578){\color[rgb]{0,0,0}\makebox(0,0)[lb]{\smash{$\gamma$}}}%
    \put(0.75234678,0.36911329){\color[rgb]{0,0,0}\makebox(0,0)[lb]{\smash{$v$}}}%
    \put(0.49431287,0.22235184){\color[rgb]{0,0,0}\makebox(0,0)[lb]{\smash{$\gamma$}}}%
    \put(0.74641474,0.20738299){\color[rgb]{0,0,0}\makebox(0,0)[lb]{\smash{$\tau_3$}}}%
    \put(0.79447058,0.16397771){\color[rgb]{0,0,0}\makebox(0,0)[lb]{\smash{$\tau_4$}}}%
    \put(0.65495368,0.10972116){\color[rgb]{0,0,0}\makebox(0,0)[lb]{\smash{$\tau_2$}}}%
    \put(0.62550013,0.06166533){\color[rgb]{0,0,0}\makebox(0,0)[lb]{\smash{$\tau_1$}}}%
  \end{picture}%
\endgroup%

%% file: delta.pdf_tex
\begingroup%
  \makeatletter%
  \providecommand\color[2][]{%
    \errmessage{(Inkscape) Color is used for the text in Inkscape, but the package 'color.sty' is not loaded}%
    \renewcommand\color[2][]{}%
  }%
  \providecommand\transparent[1]{%
    \errmessage{(Inkscape) Transparency is used (non-zero) for the text in Inkscape, but the package 'transparent.sty' is not loaded}%
    \renewcommand\transparent[1]{}%
  }%
  \providecommand\rotatebox[2]{#2}%
  \ifx\svgwidth\undefined%
    \setlength{\unitlength}{169.1296875bp}%
    \ifx\svgscale\undefined%
      \relax%
    \else%
      \setlength{\unitlength}{\unitlength * \real{\svgscale}}%
    \fi%
  \else%
    \setlength{\unitlength}{\svgwidth}%
  \fi%
  \global\let\svgwidth\undefined%
  \global\let\svgscale\undefined%
  \makeatother%
  \begin{picture}(1,0.60638522)%
    \put(0,0){\includegraphics[width=\unitlength]{delta.pdf}}%
    \put(0.72691667,0.43197109){\color[rgb]{0,0,0}\makebox(0,0)[lb]{\smash{$\gamma_1$}}}%
    \put(0.53136707,0.43010741){\color[rgb]{0,0,0}\makebox(0,0)[lb]{\smash{$p_{j+1}$}}}%
    \put(0.66687453,0.27898953){\color[rgb]{0,0,0}\makebox(0,0)[lb]{\smash{$\tau_{i_{j+1}}$}}}%
    \put(0.44722652,0.19025172){\color[rgb]{0,0,0}\makebox(0,0)[lb]{\smash{$\Delta_j$}}}%
    \put(0.55880779,0.09340221){\color[rgb]{0,0,0}\makebox(0,0)[lb]{\smash{$\tau_{[j]}$}}}%
    \put(0.1821572,0.15335081){\color[rgb]{0,0,0}\makebox(0,0)[lb]{\smash{$\tau_{i_j}$}}}%
    \put(0.14215057,0.38178482){\color[rgb]{0,0,0}\makebox(0,0)[lb]{\smash{$p_j$}}}%
  \end{picture}%
\endgroup%

%% file: fivecrossings.pdf_tex
\begingroup%
  \makeatletter%
  \providecommand\color[2][]{%
    \errmessage{(Inkscape) Color is used for the text in Inkscape, but the package 'color.sty' is not loaded}%
    \renewcommand\color[2][]{}%
  }%
  \providecommand\transparent[1]{%
    \errmessage{(Inkscape) Transparency is used (non-zero) for the text in Inkscape, but the package 'transparent.sty' is not loaded}%
    \renewcommand\transparent[1]{}%
  }%
  \providecommand\rotatebox[2]{#2}%
  \ifx\svgwidth\undefined%
    \setlength{\unitlength}{448.43042921bp}%
    \ifx\svgscale\undefined%
      \relax%
    \else%
      \setlength{\unitlength}{\unitlength * \real{\svgscale}}%
    \fi%
  \else%
    \setlength{\unitlength}{\svgwidth}%
  \fi%
  \global\let\svgwidth\undefined%
  \global\let\svgscale\undefined%
  \makeatother%
  \begin{picture}(1,0.49233865)%
    \put(0,0){\includegraphics[width=\unitlength]{fivecrossings.pdf}}%
    \put(0.18134694,0.39888298){\color[rgb]{0,0,0}\makebox(0,0)[lb]{\smash{$\gamma_1$}}}%
    \put(0.18010105,0.43185199){\color[rgb]{0,0,0}\makebox(0,0)[lb]{\smash{$\gamma_2$}}}%
    \put(0.54171644,0.39888298){\color[rgb]{0,0,0}\makebox(0,0)[lb]{\smash{$\gamma_1$}}}%
    \put(0.54046911,0.43185199){\color[rgb]{0,0,0}\makebox(0,0)[lb]{\smash{$\gamma_2$}}}%
    \put(0.90922342,0.39888298){\color[rgb]{0,0,0}\makebox(0,0)[lb]{\smash{$\gamma_1$}}}%
    \put(0.82439856,0.26563387){\color[rgb]{0,0,0}\makebox(0,0)[lb]{\smash{$\gamma_2$}}}%
    \put(0.18134694,0.15269045){\color[rgb]{0,0,0}\makebox(0,0)[lb]{\smash{$\gamma_1$}}}%
    \put(0.07737992,0.00474224){\color[rgb]{0,0,0}\makebox(0,0)[lb]{\smash{$\gamma_2$}}}%
    \put(0.54171666,0.15269045){\color[rgb]{0,0,0}\makebox(0,0)[lb]{\smash{$\gamma_1$}}}%
    \put(0.44131597,0.00474224){\color[rgb]{0,0,0}\makebox(0,0)[lb]{\smash{$\gamma_2$}}}%
  \end{picture}%
\endgroup%

%% file: figfansnake.pdf_tex
\begingroup%
  \makeatletter%
  \providecommand\color[2][]{%
    \errmessage{(Inkscape) Color is used for the text in Inkscape, but the package 'color.sty' is not loaded}%
    \renewcommand\color[2][]{}%
  }%
  \providecommand\transparent[1]{%
    \errmessage{(Inkscape) Transparency is used (non-zero) for the text in Inkscape, but the package 'transparent.sty' is not loaded}%
    \renewcommand\transparent[1]{}%
  }%
  \providecommand\rotatebox[2]{#2}%
  \ifx\svgwidth\undefined%
    \setlength{\unitlength}{299.29140625bp}%
    \ifx\svgscale\undefined%
      \relax%
    \else%
      \setlength{\unitlength}{\unitlength * \real{\svgscale}}%
    \fi%
  \else%
    \setlength{\unitlength}{\svgwidth}%
  \fi%
  \global\let\svgwidth\undefined%
  \global\let\svgscale\undefined%
  \makeatother%
  \begin{picture}(1,0.25907638)%
    \put(0,0){\includegraphics[width=\unitlength]{figfansnake.pdf}}%
    \put(0.29635102,0.19850304){\color[rgb]{0,0,0}\makebox(0,0)[lb]{\smash{$\tau_{[k]}$}}}%
    \put(0.29635102,0.05950807){\color[rgb]{0,0,0}\makebox(0,0)[lb]{\smash{$\tau_{[j]}$}}}%
    \put(-0.00101803,0.13568801){\color[rgb]{0,0,0}\makebox(0,0)[lb]{\smash{$\tau_{[k-1]}$}}}%
    \put(0.1566878,0.0554986){\color[rgb]{0,0,0}\makebox(0,0)[lb]{\smash{$\tau_{i_j}$}}}%
    \put(0.63782424,0.06685877){\color[rgb]{0,0,0}\makebox(0,0)[lb]{\smash{$j$}}}%
    \put(0.82493285,0.15840834){\color[rgb]{0,0,0}\makebox(0,0)[lb]{\smash{$k$}}}%
    \put(0.50617996,0.14170221){\color[rgb]{0,0,0}\makebox(0,0)[lb]{\smash{$e_j$}}}%
    \put(0.71199944,0.15840834){\color[rgb]{0,0,0}\makebox(0,0)[lb]{\smash{$k\!-\!1$}}}%
    \put(0.92984732,0.07487771){\color[rgb]{0,0,0}\makebox(0,0)[lb]{\smash{$e_k$}}}%
    \put(0.62846881,0.12432784){\color[rgb]{0,0,0}\makebox(0,0)[lb]{\smash{$-$}}}%
    \put(0.72870556,0.09826628){\color[rgb]{0,0,0}\makebox(0,0)[lb]{\smash{$+$}}}%
    \put(0.78483815,0.15707185){\color[rgb]{0,0,0}\makebox(0,0)[lb]{\smash{$+$}}}%
    \put(0.8242646,0.21855039){\color[rgb]{0,0,0}\makebox(0,0)[lb]{\smash{$+$}}}%
    \put(0.57166798,0.06619052){\color[rgb]{0,0,0}\makebox(0,0)[lb]{\smash{$-$}}}%
    \put(0.66856351,0.15774009){\color[rgb]{0,0,0}\makebox(0,0)[lb]{\smash{$-$}}}%
    \put(0.72937381,0.22122337){\color[rgb]{0,0,0}\makebox(0,0)[lb]{\smash{$-$}}}%
    \put(0.82894232,0.09826628){\color[rgb]{0,0,0}\makebox(0,0)[lb]{\smash{$-$}}}%
    \put(0.8790607,0.15774009){\color[rgb]{0,0,0}\makebox(0,0)[lb]{\smash{$-$}}}%
    \put(0.91380944,0.21855039){\color[rgb]{0,0,0}\makebox(0,0)[lb]{\smash{$-$}}}%
    \put(0.16203376,0.17578271){\color[rgb]{0,0,0}\makebox(0,0)[lb]{\smash{$\tau_{i_k}$}}}%
    \put(0.20747443,0.00938969){\color[rgb]{0,0,0}\makebox(0,0)[lb]{\smash{$\gamma$}}}%
  \end{picture}%
\endgroup%

%% file: figlem53.pdf_tex
\begingroup%
  \makeatletter%
  \providecommand\color[2][]{%
    \errmessage{(Inkscape) Color is used for the text in Inkscape, but the package 'color.sty' is not loaded}%
    \renewcommand\color[2][]{}%
  }%
  \providecommand\transparent[1]{%
    \errmessage{(Inkscape) Transparency is used (non-zero) for the text in Inkscape, but the package 'transparent.sty' is not loaded}%
    \renewcommand\transparent[1]{}%
  }%
  \providecommand\rotatebox[2]{#2}%
  \ifx\svgwidth\undefined%
    \setlength{\unitlength}{454.83793945bp}%
    \ifx\svgscale\undefined%
      \relax%
    \else%
      \setlength{\unitlength}{\unitlength * \real{\svgscale}}%
    \fi%
  \else%
    \setlength{\unitlength}{\svgwidth}%
  \fi%
  \global\let\svgwidth\undefined%
  \global\let\svgscale\undefined%
  \makeatother%
  \begin{picture}(1,0.14204012)%
    \put(0,0){\includegraphics[width=\unitlength]{figlem53.pdf}}%
    \put(0.05995862,0.08863334){\color[rgb]{0,0,0}\makebox(0,0)[lb]{\smash{$s\!-\!1$}}}%
    \put(0.14338715,0.05043681){\color[rgb]{0,0,0}\makebox(0,0)[lb]{\smash{$s$}}}%
    \put(0.06347636,0.02982898){\color[rgb]{0,0,0}\makebox(0,0)[lb]{\smash{$\tau_{[s-1]}$}}}%
    \put(0.34442441,0.04290277){\color[rgb]{0,0,0}\makebox(0,0)[lb]{\smash{$t$}}}%
    \put(0.39619328,0.08813476){\color[rgb]{0,0,0}\makebox(0,0)[lb]{\smash{$t\!+\!1$}}}%
    \put(0.38915781,0.03334672){\color[rgb]{0,0,0}\makebox(0,0)[lb]{\smash{$\tau_{[t]}$}}}%
    \put(0.52078211,0.0851156){\color[rgb]{0,0,0}\makebox(0,0)[lb]{\smash{$s\!-\!1$}}}%
    \put(0.61124611,0.05043681){\color[rgb]{0,0,0}\makebox(0,0)[lb]{\smash{$s$}}}%
    \put(0.53133532,0.02982898){\color[rgb]{0,0,0}\makebox(0,0)[lb]{\smash{$\tau_{[s-1]}$}}}%
    \put(0.81228333,0.0851156){\color[rgb]{0,0,0}\makebox(0,0)[lb]{\smash{$t$}}}%
    \put(0.86412131,0.03545839){\color[rgb]{0,0,0}\makebox(0,0)[lb]{\smash{$t\!+\!1$}}}%
    \put(0.861661,0.09418506){\color[rgb]{0,0,0}\makebox(0,0)[lb]{\smash{$\tau_{[t]}$}}}%
    \put(0.44298045,0.12853599){\color[rgb]{0,0,0}\makebox(0,0)[lb]{\smash{$\gamma_1$}}}%
    \put(0.90814551,0.00467543){\color[rgb]{0,0,0}\makebox(0,0)[lb]{\smash{$\gamma_1$}}}%
  \end{picture}%
\endgroup%

%% file: figlem53b.pdf_tex
\begingroup%
  \makeatletter%
  \providecommand\color[2][]{%
    \errmessage{(Inkscape) Color is used for the text in Inkscape, but the package 'color.sty' is not loaded}%
    \renewcommand\color[2][]{}%
  }%
  \providecommand\transparent[1]{%
    \errmessage{(Inkscape) Transparency is used (non-zero) for the text in Inkscape, but the package 'transparent.sty' is not loaded}%
    \renewcommand\transparent[1]{}%
  }%
  \providecommand\rotatebox[2]{#2}%
  \ifx\svgwidth\undefined%
    \setlength{\unitlength}{459.91323242bp}%
    \ifx\svgscale\undefined%
      \relax%
    \else%
      \setlength{\unitlength}{\unitlength * \real{\svgscale}}%
    \fi%
  \else%
    \setlength{\unitlength}{\svgwidth}%
  \fi%
  \global\let\svgwidth\undefined%
  \global\let\svgscale\undefined%
  \makeatother%
  \begin{picture}(1,0.10790333)%
    \put(0,0){\includegraphics[width=\unitlength]{figlem53b.pdf}}%
    \put(0.15017597,0.03470548){\color[rgb]{0,0,0}\makebox(0,0)[lb]{\smash{$s$}}}%
    \put(0.05833648,0.08663329){\color[rgb]{0,0,0}\makebox(0,0)[lb]{\smash{$\tau'_{[s'-1]}$}}}%
    \put(0.34899472,0.02029674){\color[rgb]{0,0,0}\makebox(0,0)[lb]{\smash{$t$}}}%
    \put(0.39323447,0.01780398){\color[rgb]{0,0,0}\makebox(0,0)[lb]{\smash{$\tau_{[t]}$}}}%
    \put(0.61287195,0.03470548){\color[rgb]{0,0,0}\makebox(0,0)[lb]{\smash{$s$}}}%
    \put(0.533843,0.01432507){\color[rgb]{0,0,0}\makebox(0,0)[lb]{\smash{$\tau'_{[s'-1]}$}}}%
    \put(0.81169066,0.02377566){\color[rgb]{0,0,0}\makebox(0,0)[lb]{\smash{$t$}}}%
    \put(0.86288827,0.02476182){\color[rgb]{0,0,0}\makebox(0,0)[lb]{\smash{$\tau_{[t]}$}}}%
    \put(0.43950532,0.09454823){\color[rgb]{0,0,0}\makebox(0,0)[lb]{\smash{$\gamma_1$}}}%
    \put(0.03647563,0.00462383){\color[rgb]{0,0,0}\makebox(0,0)[lb]{\smash{$\gamma_2$}}}%
    \put(0.90915916,0.09454823){\color[rgb]{0,0,0}\makebox(0,0)[lb]{\smash{$\gamma_1$}}}%
    \put(0.51253357,0.09241443){\color[rgb]{0,0,0}\makebox(0,0)[lb]{\smash{$\gamma_2$}}}%
  \end{picture}%
\endgroup%

%% file: smoothing.pdf_tex
\begingroup%
  \makeatletter%
  \providecommand\color[2][]{%
    \errmessage{(Inkscape) Color is used for the text in Inkscape, but the package 'color.sty' is not loaded}%
    \renewcommand\color[2][]{}%
  }%
  \providecommand\transparent[1]{%
    \errmessage{(Inkscape) Transparency is used (non-zero) for the text in Inkscape, but the package 'transparent.sty' is not loaded}%
    \renewcommand\transparent[1]{}%
  }%
  \providecommand\rotatebox[2]{#2}%
  \ifx\svgwidth\undefined%
    \setlength{\unitlength}{467.35703125bp}%
    \ifx\svgscale\undefined%
      \relax%
    \else%
      \setlength{\unitlength}{\unitlength * \real{\svgscale}}%
    \fi%
  \else%
    \setlength{\unitlength}{\svgwidth}%
  \fi%
  \global\let\svgwidth\undefined%
  \global\let\svgscale\undefined%
  \makeatother%
  \begin{picture}(1,0.13409181)%
    \put(0,0){\includegraphics[width=\unitlength]{smoothing.pdf}}%
    \put(0.10100088,0.05713352){\color[rgb]{0,0,0}\makebox(0,0)[lb]{\smash{$s$}}}%
    \put(0.29898894,0.05704542){\color[rgb]{0,0,0}\makebox(0,0)[lb]{\smash{$t$}}}%
    \put(0.39365648,0.0148207){\color[rgb]{0,0,0}\makebox(0,0)[lb]{\smash{$\gamma_1$}}}%
    \put(0.33888037,0.00455019){\color[rgb]{0,0,0}\makebox(0,0)[lb]{\smash{$\gamma_4$}}}%
    \put(0.37996246,0.03193824){\color[rgb]{0,0,0}\makebox(0,0)[lb]{\smash{$\gamma_6$}}}%
    \put(0.39365648,0.11752591){\color[rgb]{0,0,0}\makebox(0,0)[lb]{\smash{$\gamma_2$}}}%
    \put(0.33545687,0.12094942){\color[rgb]{0,0,0}\makebox(0,0)[lb]{\smash{$\gamma_3$}}}%
    \put(-0.00004681,0.02851474){\color[rgb]{0,0,0}\makebox(0,0)[lb]{\smash{$\gamma_5$}}}%
    \put(0.61795041,0.03316897){\color[rgb]{0,0,0}\makebox(0,0)[lb]{\smash{$s$}}}%
    \put(0.81593845,0.05704542){\color[rgb]{0,0,0}\makebox(0,0)[lb]{\smash{$t$}}}%
    \put(0.91060602,0.0148207){\color[rgb]{0,0,0}\makebox(0,0)[lb]{\smash{$\gamma_1$}}}%
    \put(0.85582991,0.00455019){\color[rgb]{0,0,0}\makebox(0,0)[lb]{\smash{$\gamma_4$}}}%
    \put(0.89691199,0.03193824){\color[rgb]{0,0,0}\makebox(0,0)[lb]{\smash{$\gamma_6$}}}%
    \put(0.91060602,0.11752591){\color[rgb]{0,0,0}\makebox(0,0)[lb]{\smash{$\gamma_2$}}}%
    \put(0.8524064,0.12094942){\color[rgb]{0,0,0}\makebox(0,0)[lb]{\smash{$\gamma_3$}}}%
    \put(0.52210412,0.08857797){\color[rgb]{0,0,0}\makebox(0,0)[lb]{\smash{$\gamma_5$}}}%
  \end{picture}%
\endgroup%

%% file: smoothing3.pdf_tex
\begingroup%
  \makeatletter%
  \providecommand\color[2][]{%
    \errmessage{(Inkscape) Color is used for the text in Inkscape, but the package 'color.sty' is not loaded}%
    \renewcommand\color[2][]{}%
  }%
  \providecommand\transparent[1]{%
    \errmessage{(Inkscape) Transparency is used (non-zero) for the text in Inkscape, but the package 'transparent.sty' is not loaded}%
    \renewcommand\transparent[1]{}%
  }%
  \providecommand\rotatebox[2]{#2}%
  \ifx\svgwidth\undefined%
    \setlength{\unitlength}{326.30473633bp}%
    \ifx\svgscale\undefined%
      \relax%
    \else%
      \setlength{\unitlength}{\unitlength * \real{\svgscale}}%
    \fi%
  \else%
    \setlength{\unitlength}{\svgwidth}%
  \fi%
  \global\let\svgwidth\undefined%
  \global\let\svgscale\undefined%
  \makeatother%
  \begin{picture}(1,0.26621383)%
    \put(0,0){\includegraphics[width=\unitlength]{smoothing3.pdf}}%
    \put(0.37460647,0.19803845){\color[rgb]{0,0,0}\makebox(0,0)[lb]{\smash{$\gamma_1$}}}%
    \put(0.30807346,0.23402877){\color[rgb]{0,0,0}\makebox(0,0)[lb]{\smash{$\gamma_4$}}}%
    \put(0.31297685,0.17518807){\color[rgb]{0,0,0}\makebox(0,0)[lb]{\smash{$\gamma_6$}}}%
    \put(0.15606831,0.00651711){\color[rgb]{0,0,0}\makebox(0,0)[lb]{\smash{$\gamma_2$}}}%
    \put(0.12664796,0.06241006){\color[rgb]{0,0,0}\makebox(0,0)[lb]{\smash{$\gamma_3$}}}%
    \put(0.00048691,0.23122972){\color[rgb]{0,0,0}\makebox(0,0)[lb]{\smash{$\gamma_5$}}}%
    \put(0.07761404,0.14576772){\color[rgb]{0,0,0}\makebox(0,0)[lb]{\smash{$s$}}}%
    \put(0.2541927,0.14828234){\color[rgb]{0,0,0}\makebox(0,0)[lb]{\smash{$s\!+\!1$}}}%
    \put(0.5952591,0.19803845){\color[rgb]{0,0,0}\makebox(0,0)[lb]{\smash{$\gamma_1$}}}%
    \put(0.86215675,0.21441521){\color[rgb]{0,0,0}\makebox(0,0)[lb]{\smash{$\gamma_4$}}}%
    \put(0.87196353,0.0967338){\color[rgb]{0,0,0}\makebox(0,0)[lb]{\smash{$\gamma_6$}}}%
    \put(0.80821944,0.00651711){\color[rgb]{0,0,0}\makebox(0,0)[lb]{\smash{$\gamma_2$}}}%
    \put(0.77879908,0.06241006){\color[rgb]{0,0,0}\makebox(0,0)[lb]{\smash{$\gamma_3$}}}%
    \put(0.65263805,0.23122972){\color[rgb]{0,0,0}\makebox(0,0)[lb]{\smash{$\gamma_5$}}}%
    \put(0.72976517,0.14576772){\color[rgb]{0,0,0}\makebox(0,0)[lb]{\smash{$d$}}}%
  \end{picture}%
\endgroup%

%% file: figlem72.pdf_tex
\begingroup%
  \makeatletter%
  \providecommand\color[2][]{%
    \errmessage{(Inkscape) Color is used for the text in Inkscape, but the package 'color.sty' is not loaded}%
    \renewcommand\color[2][]{}%
  }%
  \providecommand\transparent[1]{%
    \errmessage{(Inkscape) Transparency is used (non-zero) for the text in Inkscape, but the package 'transparent.sty' is not loaded}%
    \renewcommand\transparent[1]{}%
  }%
  \providecommand\rotatebox[2]{#2}%
  \ifx\svgwidth\undefined%
    \setlength{\unitlength}{288bp}%
    \ifx\svgscale\undefined%
      \relax%
    \else%
      \setlength{\unitlength}{\unitlength * \real{\svgscale}}%
    \fi%
  \else%
    \setlength{\unitlength}{\svgwidth}%
  \fi%
  \global\let\svgwidth\undefined%
  \global\let\svgscale\undefined%
  \makeatother%
  \begin{picture}(1,0.39965278)%
    \put(0,0){\includegraphics[width=\unitlength]{figlem72.pdf}}%
    \put(0.13432473,0.01472092){\color[rgb]{0,0,0}\makebox(0,0)[lb]{\smash{$\mathcal{G}_1$}}}%
    \put(0.74809367,0.01598727){\color[rgb]{0,0,0}\makebox(0,0)[lb]{\smash{$\mathcal{G}_2$}}}%
    \put(0.14767842,0.23430962){\color[rgb]{0,0,0}\makebox(0,0)[lb]{\smash{$s$}}}%
    \put(0.2266361,0.23430962){\color[rgb]{0,0,0}\makebox(0,0)[lb]{\smash{$s\!+\!1$}}}%
    \put(0.12812337,0.13712425){\color[rgb]{0,0,0}\makebox(0,0)[lb]{\smash{$s\!-\!1$}}}%
    \put(0.73101179,0.23430962){\color[rgb]{0,0,0}\makebox(0,0)[lb]{\smash{$s'$}}}%
    \put(0.80996941,0.23430962){\color[rgb]{0,0,0}\makebox(0,0)[lb]{\smash{$s'\!+\!1$}}}%
  \end{picture}%
\endgroup%

%% file: figlem72b.pdf_tex
\begingroup%
  \makeatletter%
  \providecommand\color[2][]{%
    \errmessage{(Inkscape) Color is used for the text in Inkscape, but the package 'color.sty' is not loaded}%
    \renewcommand\color[2][]{}%
  }%
  \providecommand\transparent[1]{%
    \errmessage{(Inkscape) Transparency is used (non-zero) for the text in Inkscape, but the package 'transparent.sty' is not loaded}%
    \renewcommand\transparent[1]{}%
  }%
  \providecommand\rotatebox[2]{#2}%
  \ifx\svgwidth\undefined%
    \setlength{\unitlength}{171.0249128bp}%
    \ifx\svgscale\undefined%
      \relax%
    \else%
      \setlength{\unitlength}{\unitlength * \real{\svgscale}}%
    \fi%
  \else%
    \setlength{\unitlength}{\svgwidth}%
  \fi%
  \global\let\svgwidth\undefined%
  \global\let\svgscale\undefined%
  \makeatother%
  \begin{picture}(1,0.83707201)%
    \put(0,0){\includegraphics[width=\unitlength]{figlem72b.pdf}}%
    \put(0.65601942,0.6506372){\color[rgb]{0,0,0}\makebox(0,0)[lb]{\smash{$s$}}}%
    \put(0.78898138,0.6506372){\color[rgb]{0,0,0}\makebox(0,0)[lb]{\smash{$s\!+\!1$}}}%
    \put(0.62308944,0.48698042){\color[rgb]{0,0,0}\makebox(0,0)[lb]{\smash{$s\!-\!1$}}}%
    \put(0.61790113,0.3225767){\color[rgb]{0,0,0}\makebox(0,0)[lb]{\smash{$s\!-\!2$}}}%
    \put(0.61797358,0.77308444){\color[rgb]{0,0,0}\makebox(0,0)[lb]{\smash{$s\!+\!1$}}}%
    \put(0.82401146,0.54960847){\color[rgb]{0,0,0}\makebox(0,0)[lb]{\smash{$s$}}}%
    \put(0.76941936,0.48627702){\color[rgb]{0,0,0}\makebox(0,0)[lb]{\smash{$s$}}}%
    \put(0.45989692,0.65134061){\color[rgb]{0,0,0}\makebox(0,0)[lb]{\smash{$s\!-\!1$}}}%
    \put(0.45440063,0.48995269){\color[rgb]{0,0,0}\makebox(0,0)[lb]{\smash{$s\!-\!2$}}}%
    \put(0.61820912,0.20766895){\color[rgb]{0,0,0}\makebox(0,0)[lb]{\smash{$s\!-\!3$}}}%
    \put(0.29037456,0.33180651){\color[rgb]{0,0,0}\makebox(0,0)[lb]{\smash{$k\!+\!2$}}}%
    \put(0.44815147,0.15538315){\color[rgb]{0,0,0}\makebox(0,0)[lb]{\smash{$k\!+\!2$}}}%
    \put(0.29133074,0.15538315){\color[rgb]{0,0,0}\makebox(0,0)[lb]{\smash{$k\!+\!1$}}}%
    \put(0.44664359,0.03745956){\color[rgb]{0,0,0}\makebox(0,0)[lb]{\smash{$k\!+\!1$}}}%
    \put(0.15592849,0.15236745){\color[rgb]{0,0,0}\makebox(0,0)[lb]{\smash{$k$}}}%
    \put(0.2639382,0.74577303){\color[rgb]{0,0,0}\makebox(0,0)[lb]{\smash{$[s\!-\!1]$}}}%
  \end{picture}%
\endgroup%

%% file: figlem75.pdf_tex
\begingroup%
  \makeatletter%
  \providecommand\color[2][]{%
    \errmessage{(Inkscape) Color is used for the text in Inkscape, but the package 'color.sty' is not loaded}%
    \renewcommand\color[2][]{}%
  }%
  \providecommand\transparent[1]{%
    \errmessage{(Inkscape) Transparency is used (non-zero) for the text in Inkscape, but the package 'transparent.sty' is not loaded}%
    \renewcommand\transparent[1]{}%
  }%
  \providecommand\rotatebox[2]{#2}%
  \ifx\svgwidth\undefined%
    \setlength{\unitlength}{368.54343262bp}%
    \ifx\svgscale\undefined%
      \relax%
    \else%
      \setlength{\unitlength}{\unitlength * \real{\svgscale}}%
    \fi%
  \else%
    \setlength{\unitlength}{\svgwidth}%
  \fi%
  \global\let\svgwidth\undefined%
  \global\let\svgscale\undefined%
  \makeatother%
  \begin{picture}(1,0.37538916)%
    \put(0,0){\includegraphics[width=\unitlength]{figlem75.pdf}}%
    \put(0.49742934,0.06220858){\color[rgb]{0,0,0}\makebox(0,0)[lb]{\smash{$k_5$}}}%
    \put(0.55070773,0.06220858){\color[rgb]{0,0,0}\makebox(0,0)[lb]{\smash{$k_5\!+\!1$}}}%
    \put(0.63319461,0.00577018){\color[rgb]{0,0,0}\makebox(0,0)[lb]{\smash{$k_5\!+\!1$}}}%
    \put(0.03462435,0.05815596){\color[rgb]{0,0,0}\makebox(0,0)[lb]{\smash{$s$}}}%
    \put(0.09618323,0.13372858){\color[rgb]{0,0,0}\makebox(0,0)[lb]{\smash{$s\!+\!2$}}}%
    \put(0.01803776,0.13372858){\color[rgb]{0,0,0}\makebox(0,0)[lb]{\smash{$s\!+\!1$}}}%
    \put(0.17315536,0.21070072){\color[rgb]{0,0,0}\makebox(0,0)[lb]{\smash{$s\!+\!4$}}}%
    \put(0.09500991,0.21070072){\color[rgb]{0,0,0}\makebox(0,0)[lb]{\smash{$s\!+\!3$}}}%
    \put(0.17315536,0.29318761){\color[rgb]{0,0,0}\makebox(0,0)[lb]{\smash{$s\!+\!5$}}}%
    \put(0.24188,0.21030147){\color[rgb]{0,0,0}\makebox(0,0)[lb]{\smash{$s\!+\!5$}}}%
    \put(0.09492667,0.26513803){\color[rgb]{0,0,0}\makebox(0,0)[lb]{\smash{$s\!+\!4$}}}%
    \put(0.16681286,0.13372858){\color[rgb]{0,0,0}\makebox(0,0)[lb]{\smash{$s\!+\!3$}}}%
    \put(0.01935403,0.19110782){\color[rgb]{0,0,0}\makebox(0,0)[lb]{\smash{$s\!+\!2$}}}%
    \put(0.15158467,0.05350567){\color[rgb]{0,0,0}\makebox(0,0)[lb]{\smash{grafting edge}}}%
    \put(0.71266419,0.13372858){\color[rgb]{0,0,0}\makebox(0,0)[lb]{\smash{$s\!-\!2$}}}%
    \put(0.63451871,0.13372858){\color[rgb]{0,0,0}\makebox(0,0)[lb]{\smash{$s\!-\!3$}}}%
    \put(0.80700199,0.21070072){\color[rgb]{0,0,0}\makebox(0,0)[lb]{\smash{$s$}}}%
    \put(0.71149086,0.21070072){\color[rgb]{0,0,0}\makebox(0,0)[lb]{\smash{$s\!-\!1$}}}%
    \put(0.78963633,0.29318761){\color[rgb]{0,0,0}\makebox(0,0)[lb]{\smash{$s\!+\!1$}}}%
    \put(0.72877325,0.26513803){\color[rgb]{0,0,0}\makebox(0,0)[lb]{\smash{$s$}}}%
    \put(0.77895238,0.13372858){\color[rgb]{0,0,0}\makebox(0,0)[lb]{\smash{$s\!-\!1$}}}%
    \put(0.635835,0.1867664){\color[rgb]{0,0,0}\makebox(0,0)[lb]{\smash{$s\!-\!2$}}}%
    \put(0.88445434,0.1989756){\color[rgb]{0,0,0}\makebox(0,0)[lb]{\smash{grafting edge}}}%
    \put(0.79197662,0.15977707){\color[rgb]{0,0,0}\makebox(0,0)[lb]{\smash{$s\!-\!1$}}}%
    \put(0.64017641,0.22149772){\color[rgb]{0,0,0}\makebox(0,0)[lb]{\smash{$s\!-\!2$}}}%
    \put(0.71445213,0.08051415){\color[rgb]{0,0,0}\makebox(0,0)[lb]{\smash{$s\!-\!3$}}}%
    \put(0.70120438,0.05318866){\color[rgb]{0,0,0}\makebox(0,0)[lb]{\smash{$s\!-\!3$}}}%
  \end{picture}%
\endgroup%

%% file: figlem76.pdf_tex
\begingroup%
  \makeatletter%
  \providecommand\color[2][]{%
    \errmessage{(Inkscape) Color is used for the text in Inkscape, but the package 'color.sty' is not loaded}%
    \renewcommand\color[2][]{}%
  }%
  \providecommand\transparent[1]{%
    \errmessage{(Inkscape) Transparency is used (non-zero) for the text in Inkscape, but the package 'transparent.sty' is not loaded}%
    \renewcommand\transparent[1]{}%
  }%
  \providecommand\rotatebox[2]{#2}%
  \ifx\svgwidth\undefined%
    \setlength{\unitlength}{386.63769531bp}%
    \ifx\svgscale\undefined%
      \relax%
    \else%
      \setlength{\unitlength}{\unitlength * \real{\svgscale}}%
    \fi%
  \else%
    \setlength{\unitlength}{\svgwidth}%
  \fi%
  \global\let\svgwidth\undefined%
  \global\let\svgscale\undefined%
  \makeatother%
  \begin{picture}(1,0.3241419)%
    \put(0,0){\includegraphics[width=\unitlength]{figlem76.pdf}}%
    \put(0.03097781,0.05929728){\color[rgb]{0,0,0}\makebox(0,0)[lb]{\smash{$k_5$}}}%
    \put(0.08176281,0.05929728){\color[rgb]{0,0,0}\makebox(0,0)[lb]{\smash{$k_5\!+\!1$}}}%
    \put(0.16038939,0.00550014){\color[rgb]{0,0,0}\makebox(0,0)[lb]{\smash{$k_5\!+\!1$}}}%
    \put(0.15625115,0.05929728){\color[rgb]{0,0,0}\makebox(0,0)[lb]{\smash{$k_5\!+\!2$}}}%
    \put(0.08176281,0.13378562){\color[rgb]{0,0,0}\makebox(0,0)[lb]{\smash{$k_5\!+\!2$}}}%
    \put(0.81755191,0.20788008){\color[rgb]{0,0,0}\makebox(0,0)[lb]{\smash{$k_{6}\!-\!1$}}}%
    \put(0.89133682,0.12887294){\color[rgb]{0,0,0}\makebox(0,0)[lb]{\smash{$i'_{k_6-1}$}}}%
    \put(0.75953723,0.18941966){\color[rgb]{0,0,0}\makebox(0,0)[lb]{\smash{$i'_4$}}}%
    \put(0.81978268,0.05588359){\color[rgb]{0,0,0}\makebox(0,0)[lb]{\smash{$i_3'$}}}%
    \put(0.68750129,0.11885398){\color[rgb]{0,0,0}\makebox(0,0)[lb]{\smash{$i'_2$}}}%
    \put(0.32606275,0.20497836){\color[rgb]{0,0,0}\makebox(0,0)[lb]{\smash{$s$}}}%
    \put(0.23502145,0.20497836){\color[rgb]{0,0,0}\makebox(0,0)[lb]{\smash{$s\!-\!1$}}}%
    \put(0.31174056,0.15229966){\color[rgb]{0,0,0}\makebox(0,0)[lb]{\smash{$s\!-\!1$}}}%
    \put(0.16704445,0.20499259){\color[rgb]{0,0,0}\makebox(0,0)[lb]{\smash{$s\!-\!2$}}}%
    \put(0.32562961,0.26657865){\color[rgb]{0,0,0}\makebox(0,0)[lb]{\smash{$i_1'$}}}%
    \put(0.69427279,0.05588359){\color[rgb]{0,0,0}\makebox(0,0)[lb]{\smash{$1$}}}%
    \put(0.76697572,0.05588359){\color[rgb]{0,0,0}\makebox(0,0)[lb]{\smash{$2$}}}%
    \put(0.8362074,0.27866396){\color[rgb]{0,0,0}\makebox(0,0)[lb]{\smash{$k_6$}}}%
    \put(0.76630872,0.13059016){\color[rgb]{0,0,0}\makebox(0,0)[lb]{\smash{$3$}}}%
    \put(0.69592749,0.00574773){\color[rgb]{0,0,0}\makebox(0,0)[lb]{\smash{$s$}}}%
  \end{picture}%
\endgroup%

%% file: lastfigure.pdf_tex
\begingroup%
  \makeatletter%
  \providecommand\color[2][]{%
    \errmessage{(Inkscape) Color is used for the text in Inkscape, but the package 'color.sty' is not loaded}%
    \renewcommand\color[2][]{}%
  }%
  \providecommand\transparent[1]{%
    \errmessage{(Inkscape) Transparency is used (non-zero) for the text in Inkscape, but the package 'transparent.sty' is not loaded}%
    \renewcommand\transparent[1]{}%
  }%
  \providecommand\rotatebox[2]{#2}%
  \ifx\svgwidth\undefined%
    \setlength{\unitlength}{82.96038208bp}%
    \ifx\svgscale\undefined%
      \relax%
    \else%
      \setlength{\unitlength}{\unitlength * \real{\svgscale}}%
    \fi%
  \else%
    \setlength{\unitlength}{\svgwidth}%
  \fi%
  \global\let\svgwidth\undefined%
  \global\let\svgscale\undefined%
  \makeatother%
  \begin{picture}(1,0.58371236)%
    \put(0,0){\includegraphics[width=\unitlength]{lastfigure.pdf}}%
    \put(0.04158612,0.09926731){\color[rgb]{0,0,0}\makebox(0,0)[lb]{\smash{$v\!-\!1$}}}%
    \put(0.33088083,0.40784834){\color[rgb]{0,0,0}\makebox(0,0)[lb]{\smash{$v\!+\!1$}}}%
    \put(0.38873977,0.09926731){\color[rgb]{0,0,0}\makebox(0,0)[lb]{\smash{$v$}}}%
    \put(0.74691421,0.37134211){\color[rgb]{0,0,0}\makebox(0,0)[lb]{\smash{$e_v$}}}%
  \end{picture}%
\endgroup%

%% file: tau.pdf_tex
\begingroup%
  \makeatletter%
  \providecommand\color[2][]{%
    \errmessage{(Inkscape) Color is used for the text in Inkscape, but the package 'color.sty' is not loaded}%
    \renewcommand\color[2][]{}%
  }%
  \providecommand\transparent[1]{%
    \errmessage{(Inkscape) Transparency is used (non-zero) for the text in Inkscape, but the package 'transparent.sty' is not loaded}%
    \renewcommand\transparent[1]{}%
  }%
  \providecommand\rotatebox[2]{#2}%
  \ifx\svgwidth\undefined%
    \setlength{\unitlength}{482.675bp}%
    \ifx\svgscale\undefined%
      \relax%
    \else%
      \setlength{\unitlength}{\unitlength * \real{\svgscale}}%
    \fi%
  \else%
    \setlength{\unitlength}{\svgwidth}%
  \fi%
  \global\let\svgwidth\undefined%
  \global\let\svgscale\undefined%
  \makeatother%
  \begin{picture}(1,1.05082795)%
    \put(0,0){\includegraphics[width=\unitlength]{tau.pdf}}%
    \put(0.02617833,0.72983745){\color[rgb]{0,0,0}\makebox(0,0)[lb]{\smash{1}}}%
    \put(0.14417229,0.84681385){\color[rgb]{0,0,0}\makebox(0,0)[lb]{\smash{$i$}}}%
    \put(0.20261467,0.84912843){\color[rgb]{0,0,0}\makebox(0,0)[lb]{\smash{$j$}}}%
    \put(0.31402809,0.96605296){\color[rgb]{0,0,0}\makebox(0,0)[lb]{\smash{$d$}}}%
    \put(0.58427071,0.72809985){\color[rgb]{0,0,0}\makebox(0,0)[lb]{\smash{1}}}%
    \put(0.70206695,0.84635999){\color[rgb]{0,0,0}\makebox(0,0)[lb]{\smash{$i$}}}%
    \put(0.68590595,0.90901029){\color[rgb]{0,0,0}\makebox(0,0)[lb]{\smash{$i\!+\!1$}}}%
    \put(0.83015653,0.78962258){\color[rgb]{0,0,0}\makebox(0,0)[lb]{\smash{$j\!-\!1$}}}%
    \put(0.84864672,0.84867457){\color[rgb]{0,0,0}\makebox(0,0)[lb]{\smash{$j$}}}%
    \put(0.95882838,0.9655991){\color[rgb]{0,0,0}\makebox(0,0)[lb]{\smash{$d$}}}%
    \put(0.02617833,0.37514749){\color[rgb]{0,0,0}\makebox(0,0)[lb]{\smash{1}}}%
    \put(0.14417229,0.49212387){\color[rgb]{0,0,0}\makebox(0,0)[lb]{\smash{$i$}}}%
    \put(0.20261467,0.49443844){\color[rgb]{0,0,0}\makebox(0,0)[lb]{\smash{$j$}}}%
    \put(0.31402809,0.61136292){\color[rgb]{0,0,0}\makebox(0,0)[lb]{\smash{$d$}}}%
    \put(0.58427071,0.37340984){\color[rgb]{0,0,0}\makebox(0,0)[lb]{\smash{1}}}%
    \put(0.70206695,0.49166996){\color[rgb]{0,0,0}\makebox(0,0)[lb]{\smash{$i$}}}%
    \put(0.68590595,0.55432028){\color[rgb]{0,0,0}\makebox(0,0)[lb]{\smash{$i\!+\!1$}}}%
    \put(0.83015653,0.43493257){\color[rgb]{0,0,0}\makebox(0,0)[lb]{\smash{$j\!-\!1$}}}%
    \put(0.84864672,0.49398453){\color[rgb]{0,0,0}\makebox(0,0)[lb]{\smash{$j$}}}%
    \put(0.95882838,0.61090911){\color[rgb]{0,0,0}\makebox(0,0)[lb]{\smash{$d$}}}%
    \put(0.02617833,0.02377234){\color[rgb]{0,0,0}\makebox(0,0)[lb]{\smash{1}}}%
    \put(0.14417229,0.14074872){\color[rgb]{0,0,0}\makebox(0,0)[lb]{\smash{$i$}}}%
    \put(0.20261467,0.1430633){\color[rgb]{0,0,0}\makebox(0,0)[lb]{\smash{$j$}}}%
    \put(0.31402809,0.25998778){\color[rgb]{0,0,0}\makebox(0,0)[lb]{\smash{$d$}}}%
    \put(0.58427071,0.02203469){\color[rgb]{0,0,0}\makebox(0,0)[lb]{\smash{1}}}%
    \put(0.70206695,0.14029481){\color[rgb]{0,0,0}\makebox(0,0)[lb]{\smash{$i$}}}%
    \put(0.68590595,0.20294513){\color[rgb]{0,0,0}\makebox(0,0)[lb]{\smash{$i\!+\!1$}}}%
    \put(0.83015653,0.08355737){\color[rgb]{0,0,0}\makebox(0,0)[lb]{\smash{$j\!-\!1$}}}%
    \put(0.84864672,0.14260938){\color[rgb]{0,0,0}\makebox(0,0)[lb]{\smash{$j$}}}%
    \put(0.95882838,0.25953397){\color[rgb]{0,0,0}\makebox(0,0)[lb]{\smash{$d$}}}%
    \put(0.15087385,1.03810264){\color[rgb]{0,0,0}\makebox(0,0)[lb]{\smash{\huge \it P}}}%
    \put(0.61635029,1.03810264){\color[rgb]{0,0,0}\makebox(0,0)[lb]{\smash{{\huge $\tau_{i,1}(P)$}}}}%
    \put(0.81855674,1.03810264){\color[rgb]{0,0,0}\makebox(0,0)[lb]{\smash{{\huge $\tau_{i,2}(P)$}}}}%
  \end{picture}%
\endgroup%

%% file: figlem34c.pdf_tex
\begingroup%
  \makeatletter%
  \providecommand\color[2][]{%
    \errmessage{(Inkscape) Color is used for the text in Inkscape, but the package 'color.sty' is not loaded}%
    \renewcommand\color[2][]{}%
  }%
  \providecommand\transparent[1]{%
    \errmessage{(Inkscape) Transparency is used (non-zero) for the text in Inkscape, but the package 'transparent.sty' is not loaded}%
    \renewcommand\transparent[1]{}%
  }%
  \providecommand\rotatebox[2]{#2}%
  \ifx\svgwidth\undefined%
    \setlength{\unitlength}{250.50859375bp}%
    \ifx\svgscale\undefined%
      \relax%
    \else%
      \setlength{\unitlength}{\unitlength * \real{\svgscale}}%
    \fi%
  \else%
    \setlength{\unitlength}{\svgwidth}%
  \fi%
  \global\let\svgwidth\undefined%
  \global\let\svgscale\undefined%
  \makeatother%
  \begin{picture}(1,0.25693778)%
    \put(0,0){\includegraphics[width=\unitlength]{figlem34c.pdf}}%
    \put(0.01985078,0.09123288){\color[rgb]{0,0,0}\makebox(0,0)[lb]{\smash{$s\!-\!1$}}}%
    \put(0.11698223,0.09348109){\color[rgb]{0,0,0}\makebox(0,0)[lb]{\smash{$\epsilon$}}}%
    \put(0.14251008,0.0176052){\color[rgb]{0,0,0}\makebox(0,0)[lb]{\smash{$-\epsilon$}}}%
    \put(0.16393809,0.1681489){\color[rgb]{0,0,0}\makebox(0,0)[lb]{\smash{$\epsilon$}}}%
    \put(0.23831518,0.09744278){\color[rgb]{0,0,0}\makebox(0,0)[lb]{\smash{$-\epsilon$}}}%
    \put(0.76677632,0.09123288){\color[rgb]{0,0,0}\makebox(0,0)[lb]{\smash{$-\epsilon$}}}%
    \put(0.74608219,0.04220126){\color[rgb]{0,0,0}\makebox(0,0)[lb]{\smash{$s'\!-\!1$}}}%
    \put(0.77467987,0.15736108){\color[rgb]{0,0,0}\makebox(0,0)[lb]{\smash{$s'$}}}%
    \put(0.78274384,0.23195735){\color[rgb]{0,0,0}\makebox(0,0)[lb]{\smash{$\epsilon$}}}%
    \put(0.15528409,0.09265253){\color[rgb]{0,0,0}\makebox(0,0)[lb]{\smash{$s$}}}%
    \put(0.70403531,0.15961924){\color[rgb]{0,0,0}\makebox(0,0)[lb]{\smash{$\epsilon$}}}%
    \put(0.84508079,0.16131284){\color[rgb]{0,0,0}\makebox(0,0)[lb]{\smash{$-\epsilon$}}}%
  \end{picture}%
\endgroup%

%% file: figlem34bb.pdf_tex
\begingroup%
  \makeatletter%
  \providecommand\color[2][]{%
    \errmessage{(Inkscape) Color is used for the text in Inkscape, but the package 'color.sty' is not loaded}%
    \renewcommand\color[2][]{}%
  }%
  \providecommand\transparent[1]{%
    \errmessage{(Inkscape) Transparency is used (non-zero) for the text in Inkscape, but the package 'transparent.sty' is not loaded}%
    \renewcommand\transparent[1]{}%
  }%
  \providecommand\rotatebox[2]{#2}%
  \ifx\svgwidth\undefined%
    \setlength{\unitlength}{184.42636719bp}%
    \ifx\svgscale\undefined%
      \relax%
    \else%
      \setlength{\unitlength}{\unitlength * \real{\svgscale}}%
    \fi%
  \else%
    \setlength{\unitlength}{\svgwidth}%
  \fi%
  \global\let\svgwidth\undefined%
  \global\let\svgscale\undefined%
  \makeatother%
  \begin{picture}(1,0.34476359)%
    \put(0,0){\includegraphics[width=\unitlength]{figlem34bb.pdf}}%
    \put(0.71402214,0.19151131){\color[rgb]{0,0,0}\makebox(0,0)[lb]{\smash{$t$}}}%
    \put(0.84539477,0.19399005){\color[rgb]{0,0,0}\makebox(0,0)[lb]{\smash{$t\!+\!1$}}}%
    \put(0.71712059,0.29995568){\color[rgb]{0,0,0}\makebox(0,0)[lb]{\smash{$\epsilon$}}}%
    \put(0.78218721,0.15680913){\color[rgb]{0,0,0}\makebox(0,0)[lb]{\smash{$-\epsilon$}}}%
    \put(0.69976949,0.08740474){\color[rgb]{0,0,0}\makebox(0,0)[lb]{\smash{$-\epsilon$}}}%
    \put(0.71402214,0.19151131){\color[rgb]{0,0,0}\makebox(0,0)[lb]{\smash{$t$}}}%
    \put(0.71712059,0.29995568){\color[rgb]{0,0,0}\makebox(0,0)[lb]{\smash{$\epsilon$}}}%
    \put(0.23258774,0.1040727){\color[rgb]{0,0,0}\makebox(0,0)[lb]{\smash{$\epsilon$}}}%
    \put(0.12414338,0.16046377){\color[rgb]{0,0,0}\makebox(0,0)[lb]{\smash{$-\epsilon$}}}%
    \put(0.13281893,0.1040727){\color[rgb]{0,0,0}\makebox(0,0)[lb]{\smash{$t$}}}%
    \put(0.10245451,0.25589481){\color[rgb]{0,0,0}\makebox(0,0)[lb]{\smash{$t\!+\!1$}}}%
    \put(0.13281893,0.00864166){\color[rgb]{0,0,0}\makebox(0,0)[lb]{\smash{$\epsilon$}}}%
    \put(-0.00165208,0.1040727){\color[rgb]{0,0,0}\makebox(0,0)[lb]{\smash{$-\epsilon$}}}%
    \put(0.02437457,0.25589481){\color[rgb]{0,0,0}\makebox(0,0)[lb]{\smash{$\epsilon$}}}%
  \end{picture}%
\endgroup%

%% file: figlem34.pdf_tex
\begingroup%
  \makeatletter%
  \providecommand\color[2][]{%
    \errmessage{(Inkscape) Color is used for the text in Inkscape, but the package 'color.sty' is not loaded}%
    \renewcommand\color[2][]{}%
  }%
  \providecommand\transparent[1]{%
    \errmessage{(Inkscape) Transparency is used (non-zero) for the text in Inkscape, but the package 'transparent.sty' is not loaded}%
    \renewcommand\transparent[1]{}%
  }%
  \providecommand\rotatebox[2]{#2}%
  \ifx\svgwidth\undefined%
    \setlength{\unitlength}{453.70987167bp}%
    \ifx\svgscale\undefined%
      \relax%
    \else%
      \setlength{\unitlength}{\unitlength * \real{\svgscale}}%
    \fi%
  \else%
    \setlength{\unitlength}{\svgwidth}%
  \fi%
  \global\let\svgwidth\undefined%
  \global\let\svgscale\undefined%
  \makeatother%
  \begin{picture}(1,0.56271278)%
    \put(0,0){\includegraphics[width=\unitlength]{figlem34.pdf}}%
    \put(0.24199636,0.08657618){\color[rgb]{0,0,0}\makebox(0,0)[lb]{\smash{$k'$}}}%
    \put(0.22413917,0.15006837){\color[rgb]{0,0,0}\makebox(0,0)[lb]{\smash{$k'\!+\!1$}}}%
    \put(0.61302885,0.46357656){\color[rgb]{0,0,0}\makebox(0,0)[lb]{\smash{$s'$}}}%
    \put(0.53763188,0.46357656){\color[rgb]{0,0,0}\makebox(0,0)[lb]{\smash{$s'\!-\!1$}}}%
    \put(0.15072633,0.15403663){\color[rgb]{0,0,0}\makebox(0,0)[lb]{\smash{$e_{k'}$}}}%
    \put(0.60509232,0.53897355){\color[rgb]{0,0,0}\makebox(0,0)[lb]{\smash{$e_{s'-1}$}}}%
    \put(0.00073526,0.06961424){\color[rgb]{0,0,0}\makebox(0,0)[lb]{\smash{matched}}}%
    \put(0.01947837,0.04462338){\color[rgb]{0,0,0}\makebox(0,0)[lb]{\smash{by $P_5$}}}%
    \put(0.82761808,0.47050485){\color[rgb]{0,0,0}\makebox(0,0)[lb]{\smash{matched}}}%
    \put(0.81562247,0.44198751){\color[rgb]{0,0,0}\makebox(0,0)[lb]{\smash{by $\tau_{d_3-v,1}(P_6)$}}}%
    \put(0.65753229,0.46357656){\color[rgb]{0,0,0}\makebox(0,0)[lb]{\smash{$s'\!+\!1$}}}%
  \end{picture}%
\endgroup%

%% file: tau1.pdf_tex
\begingroup%
  \makeatletter%
  \providecommand\color[2][]{%
    \errmessage{(Inkscape) Color is used for the text in Inkscape, but the package 'color.sty' is not loaded}%
    \renewcommand\color[2][]{}%
  }%
  \providecommand\transparent[1]{%
    \errmessage{(Inkscape) Transparency is used (non-zero) for the text in Inkscape, but the package 'transparent.sty' is not loaded}%
    \renewcommand\transparent[1]{}%
  }%
  \providecommand\rotatebox[2]{#2}%
  \ifx\svgwidth\undefined%
    \setlength{\unitlength}{411.73154297bp}%
    \ifx\svgscale\undefined%
      \relax%
    \else%
      \setlength{\unitlength}{\unitlength * \real{\svgscale}}%
    \fi%
  \else%
    \setlength{\unitlength}{\svgwidth}%
  \fi%
  \global\let\svgwidth\undefined%
  \global\let\svgscale\undefined%
  \makeatother%
  \begin{picture}(1,0.25941518)%
    \put(0,0){\includegraphics[width=\unitlength]{tau1.pdf}}%
    \put(0.03347177,0.1612244){\color[rgb]{0,0,0}\makebox(0,0)[lb]{\smash{$s$}}}%
    \put(0.08870148,0.1612244){\color[rgb]{0,0,0}\makebox(0,0)[lb]{\smash{$s\!+\!1$}}}%
    \put(0.21910684,0.1612244){\color[rgb]{0,0,0}\makebox(0,0)[lb]{\smash{$s'\!-\!1$}}}%
    \put(0.35889201,0.1612244){\color[rgb]{0,0,0}\makebox(0,0)[lb]{\smash{$s'\!+\!1$}}}%
    \put(0.30568269,0.1612244){\color[rgb]{0,0,0}\makebox(0,0)[lb]{\smash{$s'$}}}%
    \put(0.68514623,0.09516194){\color[rgb]{0,0,0}\makebox(0,0)[lb]{\smash{$s$}}}%
    \put(0.80358402,0.09516194){\color[rgb]{0,0,0}\makebox(0,0)[lb]{\smash{$s'\!-\!1$}}}%
    \put(0.87238905,0.16494701){\color[rgb]{0,0,0}\makebox(0,0)[lb]{\smash{$s'\!+\!1$}}}%
    \put(0.89202553,0.09516194){\color[rgb]{0,0,0}\makebox(0,0)[lb]{\smash{$s'$}}}%
    \put(0.66960212,0.02718224){\color[rgb]{0,0,0}\makebox(0,0)[lb]{\smash{$s\!-\!1$}}}%
    \put(0.66955053,0.16494701){\color[rgb]{0,0,0}\makebox(0,0)[lb]{\smash{$s\!+\!1$}}}%
    \put(0.01979331,0.0932447){\color[rgb]{0,0,0}\makebox(0,0)[lb]{\smash{$s\!-\!1$}}}%
    \put(0.86156597,0.24449724){\color[rgb]{0,0,0}\makebox(0,0)[lb]{\smash{$\mathcal{G}_2$}}}%
    \put(0.67732259,0.24369921){\color[rgb]{0,0,0}\makebox(0,0)[lb]{\smash{$\mathcal{G}_1$}}}%
    \put(0.05050114,0.24192341){\color[rgb]{0,0,0}\makebox(0,0)[lb]{\smash{$\mathcal{G}_1$}}}%
    \put(0.27925656,0.24192341){\color[rgb]{0,0,0}\makebox(0,0)[lb]{\smash{$\mathcal{G}_2$}}}%
  \end{picture}%
\endgroup%